\documentclass{amsart}

\usepackage[draft]{optional}

\usepackage{amsfonts, amsmath, wasysym}

\usepackage{tikz}	
\usepackage{mathtools}
\usepackage{esint}

\usepackage{amssymb,amsthm,
paralist
}

\usepackage{
latexsym,
}

\usepackage{enumitem}

\definecolor{darkgreen}{rgb}{0,0.5,0}
\definecolor{darkred}{rgb}{0.7,0,0}
\usepackage[colorlinks, 
citecolor=darkgreen, linkcolor=darkred
]{hyperref}

\theoremstyle{plain}
\newtheorem{lemma}{Lemma}[section]

\newtheorem{thm}[lemma]{Theorem}

\newtheorem{cor}[lemma]{Corollary}

\theoremstyle{definition}
\newtheorem{defn}[lemma]{Definition}

\newtheorem{ex}[lemma]{Example}
\newtheorem{rmk}[lemma]{Remark}

\numberwithin{equation}{section}

\newcommand{\loc}{{\rm loc}}

\newcommand{\subsub}{\subset \subset}

\newcommand{\supp}{ {\rm supp}}

\newcommand{\err}{{\rm err}}

\newcommand{\of}{{\circ}}

\newcommand{\boundary}{\partial}

\newcommand{\partt}{ {\frac{\partial}{\partial t} } }
\newcommand{\partr}{ {\frac{\partial}{\partial r} } }

\newcommand{\up}[1]{{ {}^{#1} } }
\renewcommand{\phi}{\varphi}
\newcommand{\ti}{\tilde}
\newcommand{\parts}{\frac{\partial}{\partial s} }
\newcommand{\al}{\alpha}
\newcommand{\be}{\beta}
\newcommand{\ga}{\gamma}

\newcommand{\de}{\delta}

\newcommand{\la}{\lambda}

\newcommand{\si}{\sigma}

\newcommand{\ep}{\varepsilon}

\newcommand{\curlL}{\mathcal L}

\newcommand{\Id}{Id}
\newcommand{\R}{\ensuremath{{\mathbb R}}}

\newcommand{\N}{\ensuremath{{\mathbb N}}}
\newcommand{\B}{\ensuremath{{\mathbb B}}}
\newcommand{\Z}{\ensuremath{{\mathbb Z}}}

\newcommand{\T}{\ensuremath{{\mathbb T}}}

\newcommand{\weak}{\rightharpoonup}
\newcommand{\downto}{\searrow}
\newcommand{\upto}{\nearrow}

\newcommand{\lap}{\Delta}

\newcommand{\grad} { {{}^h \nabla }}
\newcommand{\gradh} { {{}^h \nabla }}
\newcommand{\gradhi} { {{}^{h_i} \nabla }}

\newcommand{\gradtih}{  {{}^{\ti h} \nabla }}
\newcommand{\gradg}{\nabla}

\newcommand{\vol}{{ \rm Vol}}
\DeclareMathOperator{\Vol}{Vol}

\DeclareMathOperator{\inj}{inj}

\newcommand{\norm}[1]{\left\Vert#1\right\Vert}

\newcommand{\beq}{\begin{equation}}
\newcommand{\eeq}{\end{equation}}
\newcommand{\beqa}{\begin{equation}\begin{aligned}}
\newcommand{\eeqa}{\end{aligned}\end{equation}}
\newcommand{\brmk}{\begin{rmk}}
\newcommand{\ermk}{\end{rmk}}
\newcommand{\partref}[1]{\hbox{(\csname @roman\endcsname{\ref{#1}})}}

\newcommand{\Rm}{{\mathrm{Rm}}}
\newcommand{\Riem}{{\mathrm{Rm}}}
\newcommand{\Rc}{{\mathrm{Rc}}}

\newcommand{\Ricci}{{\mathrm{Rc}}}

\newcommand{\Sc}{{\mathrm{R}}}

\usepackage{soul}

\newsavebox\CBox
\newcommand\hcancel[2][0.5pt]{%
  \ifmmode\sbox\CBox{$#2$}\else\sbox\CBox{#2}\fi%
  \makebox[0pt][l]{\usebox\CBox}%
  \rule[0.5\ht\CBox-#1/2]{\wd\CBox}{#1}}

\begin{document}
\title[Ricci flow]{Ricci flow of $W^{2,2}$-metrics in four dimensions}

\author[T.~Lamm]{Tobias Lamm}
\address[T.~Lamm]{Institute for Analysis\\
Karlsruhe Institute of Technology \\ Englerstr. 2\\
76131 Karlsruhe\\ Germany}
\email{tobias.lamm@kit.edu}

\author[M. Simon]{Miles Simon}
\address[M. Simon]{Institut f\"ur Analysis und Numerik\\ Universit\"at Magdeburg\\ Universit\"atsplatz 2\\39106 Magdeburg\\ Germany}
\email{msimon@ovgu.de} 
\date{\today}
\date{\today}

\begin{abstract}\noindent 
 In this paper  we construct solutions to Ricci DeTurck flow in four dimensions on closed manifolds which are instantaneously smooth but whose initial  values $g$ are (possibly) non-smooth Riemannian metrics whose components  in smooth coordinates   belong to $W^{2,2}$ and satisfy  $ \frac{1}{a}h\leq g\leq a h$ for some $1<a<\infty$ and some smooth Riemannian metric $h$ on $M$. A Ricci flow related solution is constructed whose initial value is isometric in a weak sense to the initial value of the Ricci DeTurck solution. 
 Results  for a related non-compact setting are also 
 presented. Various $L^p$ estimates for Ricci flow, which we require for some of the main results, are also derived. As an application we present a possible definition of scalar curvature $\geq k$ for $W^{2,2}$ metrics $g$ on closed four manifolds which are bounded in the $L^{\infty}$ sense by $ \frac{1}{a}h\leq g\leq a h$ for some $1<a<\infty$ and some smooth Riemannian metric $h$ on $M$.
\end{abstract}
\maketitle
\tableofcontents
\parskip=10pt

\section{Introduction}
In this paper we construct  solutions to Ricci flow  and Ricci DeTurck flow which are instantaneously  smooth  but whose    initial 
values are (possibly) non-smooth Riemannian metrics whose components, in smooth coordinates, 
belong to certain Sobolev spaces. 

For a given smooth Riemannian manifold $(M,h)$, and an interval $I \subseteq \R,$ a smooth family $g(t)_{t\in I}$  of Riemannian metrics on $M$  is a solution to Ricci DeTurck $h$ Flow if
\begin{align}\label{Meq}
  \partt g_{ij}=&\,g^{ab} (\gradh_a\gradh_b g_{ij})
  -g^{kl}g_{ip}h^{pq}R_{jkql}(h)
  -g^{kl}g_{jp}h^{pq}R_{ikql}(h)\cr
  &\,+\tfrac12g^{ab}g^{pq}\left(\gradh_i g_{pa}\gradh_jg_{qb}
    +2\gradh_ag_{jp}\gradh_qg_{ib}-2\gradh_ag_{jp}
    \gradh_bg_{iq}\right.\cr
  &\,\left.\qquad\qquad\qquad-2\gradh_jg_{pa}\gradh_bg_{iq}
    -2\gradh_ig_{pa}\gradh_bg_{jq}\right),
\end{align}
 in the smooth sense on $M\times I$, where here, and in the rest of the paper,  $\gradh$ refers to the covariant derivative with respect to $h$. 
A smooth family $\ell(t)_{t\in I}$  of Riemannian metrics on $M$  is a solution to Ricci flow
if
\begin{eqnarray}
\frac{\partial \ell}{\partial t} = -2 \Ricci(\ell)  \label{RFlow}
\end{eqnarray}
 in the smooth sense on $M\times I$.  The Ricci flow was first introduced and studied by R. Hamilton in
 \cite{Hamilton}. Shortly after that, the   Ricci DeTurck flow was   introduced and studied by D. DeTurck in \cite{DeTurck}.
 Ricci DeTurck flow and  Ricci flow in the smooth  setting  are closely related : 
 given a Ricci DeTurck flow $g(t)_{t\in I}$ on a compact manifold and an $S \in I$ there  is a smooth  family of diffeomorphisms $\Phi(t):M \to M,$ $t\in I$ with $\Phi(S) = Id$ such that $\ell(t) = (\Phi(t))^*g(t)$  is a smooth solution to Ricci flow. The diffeomorphisms $\Phi(t)$ solve the following ordinary differential equation:
 \begin{align}
    \partt {\Phi}^{\al}(x,t) =& V^{\al}(\Phi(x,t),t),  \ \ \ \text{for all}\ \  
    (x,t)\in M^n \times I,\nonumber \\
     \Phi(x,S)=& x.  \label{ODEDe}
  \end{align}
  where $ {V}^{\al}(y,t) := -{g}^{\be \ga} \left({\up{g}
      \Gamma}^{\al}_{\be \ga} - {\up{h} \Gamma}^{\al}_{\be \ga}\right)
  (y,t)$

There are a number of papers on solutions to Ricci DeTurck flow and Ricci flow starting from non-smooth Riemannian metric/distance  spaces which immediately become smooth : Given
a non-smooth starting space $(M,g_0)$ or $(M,d_0)$, it is possible in some settings, 
to find smooth solutions $g(t)_{t\in (0,T)}$ to \eqref{Meq}, respectively $\ell(t)_{t\in (0,T) }$  to \eqref{RFlow} defined
for some $T>0$, where the initial values are achieved  in some weak sense. 
Here  is a non-exhaustive list of papers, where examples of this type are constructed :  \cite{SimonC0, Hoch,  KochLamm,Lai,SimonRicci,SimonTopping,SimonSchulze, Deruelle, BamlerCabezas-RivasWilking,HuangTam,LeeTam,Xu,ToppingYin}.   
The initial non-smooth data considered in these papers has certain structure, which when assumed in the smooth setting, leads to a priori estimates for solutions, which are then used to construct  solutions in the class being considered. 
In some papers this initial structure comes from geometric conditions, in others from regularity conditions on the initial
function space of the metric components in smooth coordinates.  
In the second instance, this is usually in the setting, that one has some $C^0$ control of the metric. That is,  the metric is close in the $L^{\infty}$ sense to the standard euclidean metric  in smooth coordinates: 
$(1-\ep)\de \leq g(0) \leq (1+\ep) \de$  for a sufficiently small $\ep$.
In the current paper, the structure  of the initial metric $g(0)$ comes from the assumption, in the four dimensional compact setting,  
that the components   in coordinates are in $W^{2,2}$,  and uniformly bounded from above and below :
$\frac{1}{c} \de \leq g(0) \leq c\de$ for some constant $c$. Closeness of the metric to $\de$ is not assumed.
With this initial structure, we show that a solution to Ricci DeTurck flow exists. 
In the non-compact setting, we further require a  uniform local {\it smallness} bound on the $W^{2,2}$ norm and a global uniform bound from above and below  in the $L^{\infty}$ sense,  both with respect to a geometrically controlled background metric.  
We also investigate the question of how 
the initial values are achieved, in the metric and distance sense, as time goes back to zero. See Theorem \ref{main1_start} in the next section for  details. 
 
Using this solution to Ricci DeTurck flow, we show without much trouble,  that there is a Ricci flow related solution. The Ricci flow solution is related to the Ricci DeTurck solution through a smooth family of isometries $(\Phi(t))_{ t\in (0,T)}$  defined for a positive time interval,  and having the property that $\Phi(S) = \Id$ for some $S>0$.
The convergence as time goes back to zero in the distance and metric sense is   investigated for this Ricci Flow solution. We require some new estimates on convergence in the $L^p$ sense for solutions to Ricci flow, in order to show that there is indeed a limiting weak Riemannian metric, as time approaches  to zero.  
We also show   that the  initial metric value of the Ricci flow that is achieved   is isometric, in a weak sense, to the initial value $g(0)$ of the Ricci DeTurck flow solution.  See Theorem \ref{WeakRicciStart} in the next section for details.

  Section \ref{anapplication} contains an application of the results obtained in the sections preceeding it.    We present a possible definition of   'the scalar curvature of $g$  is  bounded from below by $k\in \R$' for a metric $g \in W^{2,2}\cap L^{\infty}$  with $ \frac{1}{a}h\leq g\leq a h$ for some $1<a<\infty$ and some smooth Riemannian metric $h$ on a closed manifold $M$.

We conclude this introduction by noting, that there are metrics $g_0\in W^{2,\frac{n}{2}}(M) \cap L^{\infty}(M)$ on compact $n$-dimensional  manifolds, which satisfy  $\frac{1}{a} h \leq g_0 \leq ah $ for some $0<a<\infty$ and some smooth fixed Riemannian metric $h$ on $M$, but are not continuous. In particular,   $g_0\in W^{2,2}(M) \cap L^{\infty}(M)$  and $\frac{1}{a} h \leq g_0 \leq ah $  but $g_0$ is not continuous when $n=4$. In  the example we present below,  there is a point $p\in M$, such that the values $\frac{1}{a} h$ and $ah $  are achieved by $g_0$ infinitely often for  {\bf every}  neighbourhood of   $p$.
  
Let $(M,h)$ be a smooth compact $n$- dimensional manifold, $U \subseteq M$ open,  and $\phi:U \to \phi(U) = \B_1(0)$ be coordinates and $\ti h:= \phi_*h $, the push forward of $h$  with $\phi$ to $ \B_1(0).$ 
For $\ep, r >0,$ $c \in \R,$ let  $f= f_{\ep,r,c}:\B_1(0) \to \R$  be the $W^{2,\frac{n}{2}}(\B_1(0))$ function defined  by
$f(x) =  r(  \frac{1}{\ep}( 1 + \ep + \sin(c+  \log(\log( \frac{2}{|x|})))))$ for $x \neq 0$  and $f(0) = 0$. Then   $f$ is bounded from above and from below by $ f(\cdot) \in [r , r (\frac{2+\ep}{\ep})]$ and the values  $r$ and    $r (\frac{2+\ep}{\ep})$   are both achieved    infinitely many times on any neighbourhood of $0 \in \B_1(0),$ and consequently we see that $f$ is also not continuous. 
 Now we set $\ti g ( x)  = (1-\eta( x) )\ti h( x)  +   \eta(x)  \hat g( x)$ where $\eta $ is a smooth cut-off function $\eta \in [0,1]$ with support in $\B_{\frac 1 2}(0)$, where  $\hat g_{ij}(x)  = f_{\ep_i, r_i,c_i} ( x) \de_{ij},$ where $\ep_i,r_i,c_i \in \R,$ $i,j\in \{1,\ldots n\},$
 $ \ep_i,r_i>0$. Then the metric $g$ defined by   $g = \phi^*(\ti g) $ on $U,$ and  $ g= h$ on $M\backslash U$ is a metric on $M$ with $g \in W^{2,\frac{n}{2} }(M) \cap L^{\infty}(M),$    $ \frac{1}{a}h \leq g \leq a h$ for some $1<a< \infty,$  
 and $g$ is not continuous. 
\section{Main results}
  
   The assumptions we make on the smooth background metric are as follows 
\begin{align}
 & (M,h) \ \ \mbox{is a smooth, connected, complete manifold without boundary such that}\label{hassumptions}  \\
&\nu_i  := \sup_M  {}^h|\gradh^i\Riem(h)|  < \infty \mbox{  for all } i\in  \N_0,  \mbox{ and }\nonumber \\ 
& \inj(M,h) \geq i_0>0 \nonumber 
\end{align}
Such manifolds always satisfy a local uniform Sobolev inequality: there exist  constants $0< r_0(n,h), C_S(n)< \infty$  such that 
$(\int_M f^{\frac{2n}{n-2}} dh)^{\frac {n -2}{2} } \leq C_S(n) \int_M |\gradh f|^2 dh$ and
\\$ (\int_M f^{n}  dh)^{\frac {1}{2} } \leq C_S(n) \int_M |\gradh f|^{\frac{n}{2}}  dh$  for all smooth $f$ whose support is contained in a ball of radius $  r_0(n,h)>0$. For the readers convenience, we have included a proof in the Appendix \ref{geoapp} : See  Lemma \ref{balllemma} and Remark \ref{ballremark}.

Ultimately we would like to construct solutions to \eqref{Meq} on four manifolds starting with initial data $g_0$ which are uniformly bounded from above and below by a multiple of $h$, $g_0$ is locally in  $W^{2,2}$, and for which the homogeneous $W^{2,2}$ energy of $g_0$ is uniformly  bounded,
$E(g_0):=  \int_M (|\gradh g_0|^2  + |\gradh^2 g_0|^2)dh < \infty.$ 
That is, we assume that there exists an $a>0$ such that
 \begin{align}
 &  \frac 1 a h \leq g_0 \leq a h, \label{Mainassumptionstart}\\
& E(g_0) := \int_M (|\gradh g_0|^2  + |\gradh^2 g_0|^2)dh < \infty. \nonumber
\end{align}

In this setting we show, the following
\begin{thm}\label{main3start}
Let $1<a<\infty$ and $(M^4,h)$ be a four dimensional smooth Riemannian manifold satisfying   \eqref{hassumptions}, 
 and $g_0$  be a $W^{2,2} \cap L^{\infty}$  Riemannian metric, not necessarily smooth, which satisfies 
\begin{align*} 
& \ \ \frac 1 {a} h \leq g_0 \leq {a} h  \tag{a}  \label{a1}\\
 \end{align*} 

  and 
\begin{eqnarray*}
&&  \int_M (|\gradh g_0|^2 + |\gradh^2 g_0|^2)dh < \infty.
\end{eqnarray*}
Then  for any $0<\ep<1$ there exist  constants $0< T=T(g_0,h,a,\ep), r = r(g_0,h,a,\ep) , c_j=c_j(h,a,\ep) < \infty $   for all $j\in \N_0$ and a smooth solution
$(g(t))_{t \in (0,T]}$ to \eqref{Meq}, where $g(t)$ satisfies

\begin{align*}
 & \ \ \ \ \frac{1}{400 a} h \leq g(t) \leq 400  a h  \tag{${\rm a}_t$} \label{a_t1}  \\
&  \ \int_{B_r(x)}  (|\gradh g(\cdot,t)|^2 +   |\gradh^2 g(\cdot,t)|^2 )dh\leq \ep \ \  
 \tag{${\rm b}_t(r)$}  \label{b_tr} \\ 
 & \ \ \ \ |\gradh^jg(\cdot,t)|^2 \leq \frac{c_j}{t^{ j}}  \tag{${\rm c}_t$}  \label{c_t1}  
\end{align*}
 for all $j\in \N_0,$  $x \in M$, for all $t \in (0,T],$ where $c_j(h,a,\ep) \to 0$ as $\ep\downto 0$ 
and 
\begin{align*} 
&\int_{B_1(x_0)} (|g_0 - g(t)|^2 + |\gradh(g_0 - g(t))|^2 + |\gradh^2(g_0 - g(t))|^2)dh  \to 0 \mbox { as } t \downto 0 
\tag{$ {\rm d }_t$} \label{d_t1}  \\
 &\sup_{x\in B_1(x_0)} |\gradh^jg(\cdot,t)|^2t^j   \to  0 \mbox{ for }  t \downto 0  \tag{${\rm e}_t$} \label{e_t1}  \\ 
\end{align*} 
 and  for all $2\geq R>1$   there exists a $V(a,R)>0$ such that
\begin{align}
 & \int_{B_1(x_0)}  (|\gradh g(\cdot,t)|^2 +   |\gradh^2 g(\cdot,t)|^2)dh \tag{${\rm f}_t$} \label{f_t1}   \\ 
 & \leq
 \int_{B_{R}(x_0)} ( |\gradh g_0(\cdot)|^2 +   |\gradh^2 g_0(\cdot)|^2 )dh + V(a,R)t  \nonumber
\end{align}
for all $x_0\in M,$ for all $t \leq T.$
Furthermore, there exists $\ep_0 = \ep_0(g_0,h,a)$ such that if $\ep \leq \ep_0$ then the solution is unique in the class of solutions which satisfy \eqref{a_t1}, \eqref{b_tr}, \eqref{c_t1}, \eqref{d_t1} for the $r= r(g_0,h,a,\ep)>0$ defined above.
\end{thm}
\begin{proof}
See  Theorem \ref{main3} of Section \ref{existencechap} : The proof is given there.
\end{proof}

Assume \eqref{hassumptions} and \eqref{Mainassumptionstart} and that $M$ is four dimensional.  
Then for any $1>\ep>0,$ we can find  an $r>0$  such that 
\begin{align}
 &  \frac 1 a h \leq g_0 \leq a h, \label{MainassumptionL2be} \\
& \sup_{x\in M}  \int_{B_{r}(x)} ( |\gradh g_0|^4  +    |\gradh^2 g_0|^2)dh  < \ep,\nonumber
\end{align}
see Theorem \ref{smalllocallemma} in  Appendix \ref{geoapp} for a proof.
After scaling $h$ and $g_0$ once, and still calling the resulting metrics $g_0$ and $h$, we may assume 
\begin{align} 
 & (M,h) \ \ \mbox{is a smooth, connected, complete manifold without boundary such that}\label{hassumptionsscaled}\\
& \sup_M  {}^h|\gradh^i\Riem(h)|  < \infty \mbox{ for all } i\in \N_0 \cr 
& \sum_{i=0}^4 \sup_M  {}^h|\gradh^i\Riem(h)|  \leq \de_0(a)  \cr
 & \inj(M,h)  \geq 100,\nonumber
 \end{align}
for a small positive constant $\de_0(a)$ of our choice, in place of the assumptions \eqref{hassumptions},  and the scale invariant condition  
 $\frac 1 a h \leq g_0 \leq a h$ and
$\sup_{x\in M}  \int_{B_{1}(x)} ( |\gradh g_0|^4   +   |\gradh^2 g_0|^2)dh  < \frac {\ep}{2},$
is still correct, 
and hence, using  H\"older's ineqaulity, we have 
\begin{align}
 &  \frac 1 a h \leq g_0 \leq a h, \label{realMainassumption}\\
& \sup_{x\in M}  \int_{B_{1}(x)} ( |\gradh g_0|^2   +  |\gradh^2 g_0|^2)dh  < c(n)\sqrt{\ep}. \nonumber
\end{align}

Note further, if we assume \eqref{hassumptions}, then   \eqref{MainassumptionL2be} is a stronger  assumption than \eqref{Mainassumptionstart}: \eqref{hassumptions} and \eqref{Mainassumptionstart} $\implies$:  for any $\ep>0$ there exists an $r>0$ such that  \eqref{MainassumptionL2be} holds,   but for any given $\ep>0$ there are $g_0$ and $h$ and $r>0$  for which  \eqref{hassumptions} and \eqref{MainassumptionL2be} hold, but
$E(g_0):= \infty$. 

The main estimates required 
for the construction of solutions to \eqref{Meq} in the $W^{2,2}$ setting in this paper are proved in this setting, that is under the assumptions   \eqref{realMainassumption}  (with $c(n) \sqrt{\ep}$ replaced by $\ep$)   and \eqref{hassumptionsscaled},  and we also prove an existence result  in this setting: 

\begin{thm}\label{main1_start}
For any $1<a<\infty  $ 
 there exists a constant $\ep_1 =\ep_1(a) >0$ with the
following properties.
Let $(M^4,h)$  be a smooth four dimensional Riemannian manifold which satisfies  \eqref{hassumptionsscaled}. 
Let $g_0$  be  a $W_{loc}^{2,2} \cap L^{\infty}$  Riemannian metric, not necessarily smooth, which satisfies 
\begin{align} 
& \ \ \frac 1 { a } h \leq g_0 \leq { a } h  \tag{a}  \label{a2}\\
& \ \ \int_{B_2(x)}  (|\gradh g_0|^2 +   |\gradh^2 g_0|^2)dh \leq \ep
\ \ \mbox{ \rm for all } \ \ x \in M  ,
 \tag{b} \label{b2}
 \end{align}
where $\ep \leq \ep_1.$ 
Then there exists a constant $T=T(a,\ep)>0$ and a smooth solution
$(g(t))_{t \in (0,T]}$ to \eqref{Meq} such that

\begin{align}
 & \ \ \ \ \frac{1}{400 a } h \leq g(t) \leq 400 a h  \tag{${\rm a}_t$} \label{a_t2}  \\
 &  \ \ \ \ \int_{B_1(x)} ( |\gradh g(\cdot,t)|^2 +   |\gradh^2 g(\cdot,t)|^2 )dh\leq 2\ep \ \  \tag{${\rm b}_t$}  \label{b_t2} \\
 & \ \ \ \ |\gradh^jg(\cdot,t)|^2 \leq \frac{c_j(h,a,\ep)}{t^{ j}}  \tag{${\rm c}_t$}  \label{c_t2} 
\end{align}
  for all  $x \in M  , \ t \in [0,T],$ 
where $c_j(h,\ep,a) \to 0$ as $\ep \to 0, $ and 
\begin{align}
 & \int_{B_1(x)} (|g_0 - g(t)|^2 + |\gradh(g_0 - g(t))|^2 + |\gradh^2(g_0 - g(t))|^2)dh \to 0  \tag{${ d}_t$}  \label{d_t2}  \\
& \mbox{ \rm as }  t\downto 0 \mbox{ \rm for all }  x \in M \nonumber
\end{align}
The solution is unique in the class of solutions which satisfy   \eqref{a_t2}, \eqref{b_t2}, \eqref{c_t2}, and \eqref{d_t2}.
The solution also satisfies the local estimates
\begin{align}
 &\sup_{x\in B_1(x_0)} |\gradh^jg(\cdot,t)|^2t^j   \to  0 \mbox{ for }  t \to 0  \tag{${\rm e}_t$} \label{e_t2}  \\ 
 \end{align}
 and for all $ 1<R\leq 2$ there exists a $V(a,R)>0$ 
 \begin{align}
 & \int_{B_1(x_0)}  (|\gradh g(\cdot,t)|^2 +   |\gradh^2 g(\cdot,t)|^2)dh \tag{${\rm f}_t$} \label{f_t2}   \\ 
 & \leq
 \int_{B_{R}(x_0)} ( |\gradh g_0(\cdot)|^2 +   |\gradh^2 g_0(\cdot)|^2 )dh + V(a,R)t  \nonumber
\end{align}
for all $x_0\in M$, $2\geq R>1$  for all $t \leq T.$
\end{thm}
\begin{proof}
The Theorem  follows from Theorem \ref{main2}  and Remark \ref{aftermain2}  of Section \ref{existencechap}. 
\end{proof}

With  a solution of this type  at hand, we can without much trouble now construct a solution to Ricci flow
$(\ell(t))_{t\in (0,T)} = ( (\Phi(t))^*g(t))_{t\in (0,T)}   $ with $\ell(S) = g(S)$ and $\Phi(S) = Id$ for any given fixed $S>0.$
After some work it becomes clear, that  the Ricci Flow solution has initial starting data corresponding in some weak isometric sense to the starting data $g_0$ of the Ricci DeTurck flow solution.  More specifically, we  show  for all $p \in [1,\infty),$ that there is a weak limit $\ell_0 := \lim_{t\downto 0} \ell(t)$ in the $L^p_{loc}$ sense  and that $\ell_0$ is isometric to $g_0$ with the help of a $W^{1,p}$ isometry, and that there is a uniform limit $d_0:= \lim_{t\downto 0} d_t$ for $d_t:= d(g(t)),$ where  $d_0$ can be explicitly calculated from the starting data $g_0.$
These facts, and more details, are contained in the following theorem.

\begin{thm}\label{WeakRicciStart}
Let $1<a<\infty$ , $M= M^4$ be a four dimensional manifold, and $g_0$ and $h$ satisfy  the assumptions \eqref{hassumptionsscaled}, \eqref{a2} and \eqref{b2}, with $\ep \leq \ep_1$ where $\ep_1= \ep_1(a)>0$ is the constant coming from Theorem \ref{main1_start}, and let 

$(M,g(t))_{t\in (0,T]}$  be the   smooth solution to  \eqref{Meq} 
constructed in Theorem \ref{main1_start}.
Then \begin{itemize}

\item[(i)] 
there exists a constant $c(a)$ and a smooth solution  $\Phi :M \times (0,T] \to M$ to   \eqref{ODEDe} with 
$\Phi(T/2) = Id$ such that  $\Phi(t):= \Phi(\cdot,t) : M \to M$ is a diffeomorphism and 
 $d_h(\Phi(t)(x), \Phi(s)(x)) \leq c(a)  \sqrt {|t - s|}$ for all $x\in M.$  The metrics $\ell(t):= (\Phi(t))^*g(t), t\in (0,T]$ solve the  Ricci flow equation.
 Furthermore there are   well defined limit maps $\Phi(0): M \to M,$  $\Phi(0): = \lim_{t\downto 0} \Phi(t),$
 and $W(0): M \to M,$  $W(0): = \lim_{t\downto 0} W(t),$ where  $W(t)$ is the inverse of $\Phi(t)$ and  
  these limits  are obtained uniformly  on compact subsets, and  $\Phi(0), W(0)$  are  homeomorphisms inverse to one another.

 \item[(ii)] For the Ricci flow solution $\ell(t)$ from (i), there is a value $\ell_0(\cdot) = \lim_{t\downto 0} \ell (\cdot,t) $ well defined up to a set of measure zero, where the limit exists in the $L^p_{loc}$ sense, for any $p \in [1,\infty)$, such that, $\ell_0$ is positive definite, and  in $W_{loc}^{1,2}$ and for any $x_0 \in M$ and $0<s<t\leq T$  we have 
 \begin{eqnarray}
 && \int_{B_{1}(x_0)}  |\ell(s)-\ell_0|^p_{\ell(t)} d\ell(t) \leq  c(g_0,h,p,x_0) s \cr
 &&  \int_{ B_{1}(x_0)}  |(\ell(0))^{-1}-(\ell(s))^{-1}|^p_{\ell(t)} d\ell(t) \leq c(g_0,h,p,x_0) |s |^{\frac 1 4}  \cr
&& \int_{B_{1}(x_0) )}|\gradg  \ell_0 |^2_{\ell(t)}   d\ell(t) \leq   c(g_0,h,p,x_0) t^{\si} \cr
 && \int_{B_{1}(x_0)} |\Rm(\ell)|^2(x,t) d\ell(x,t)   + \int_0^t \int_{B_{\ell(s)}(x_0,1)}  |\gradg \Rm(\ell)|^2 (x,s) d\ell(x,s) ds   \leq  c(g_0,h,p,x_0) \cr   
 && \sup_{B_{1}(x_0)}  |\gradg^j \Rc(\ell(t))|^2 t^{j+2}     \to 0  \mbox{ as } t \downto 0   \mbox{ for all } j\in \N_0 \cr
\end{eqnarray}
for a universal constant $\si>0, $
where $\gradg$ refers to the gradient with respect to $\ell(t),$ $c(g_0,h,p,x_0)$ is a constant depending on 
 $g_0,h,p,x_0$ but not on $t$ or $s$.
  \item[(iii)]
 
The  limit maps $\Phi(0): M \to M,$  $\Phi(0): = \lim_{t\downto 0} \Phi(t),$
 and $W(0): M \to M,$  $W(0): = \lim_{t\downto 0} W(t),$ from (i) 
   are also  obtained in the $W_{loc}^{1,p}$ sense for $p\in [1,\infty)$.
 Furthermore,  for any smooth coordinates $\phi :U \to \R^n$, and $\psi: V \to \R^n$  with 
 $W(0)(V) \subsub U,$ 
the functions   $(\ell_0)_{i j} \of W(0):V \to \R $ are in $L_{loc}^p$ for all $p \in [1,\infty)$ and  
$(g_0)_{\al \be}: V \to \R $ and $(\ell_0)_{ij}: U \to \R$ are related by the identity 
$$(g_0)_{\al \be}  = D_{\al} (W(0))^i    D_{\be} (W(0))^j  ( ({\ell}_0)_{i j } \of W(0)),$$
which holds almost everywhere.
In particular: $\ell_0$ is isometric to $g_0$ almost everywhere through the map $W(0)$ which is in $W_{loc}^{1,p},$
for all $p\in [1,\infty)$.

 \item[(iv)]
 We define 
  $d_t(x,y) = d(g(t))(x,y)$ and $ \ti d_t( p,q)  = d(\ell(t))(p,q)$,   for all  $x,y,p,q \in M,$  $t \in (0,T)$.
 There  are   well defined limit metrics $ d_0,\ti d_0: M \times M \to \R_0^+,$  $d_0(x,y) = \lim_{t\downto 0}  d_t(x,y) $, 
 and $\ti d_0 :=  M \times M \to \R_0^+,$  $\ti d_0(p,q) = \lim_{t\downto 0} \ti d_t(p,q),$ 
 and they satisfy $\ti d_0(x,y) = d_0(\Phi(0)(x),\Phi(0)(y)).$ That is, $ (M, \ti d_0)$ and $(M,d_0)$ are isometric to one another through the map $\Phi(0)$.\\ 
The  metric $d_0$ satisfies $d_0(x,y):= \liminf_{\ep \downto 0} \inf_{\ga \in C_{\ep,x,y}} L_{g_0}(\ga),$ where 
$ C_{\ep, x,y} $ is the space of {\it $\ep$-approximative Lebesgue curves} between $x$ and $y$ with respect to $g_0$: This space is defined/examined in Definition \ref{lebesgueparlines}.


\end{itemize}

 \end{thm}

\begin{proof}
See Theorem \ref{mainthm} of Section \ref{Ricciflowrelatedsolution}: The proof is given there.
\end{proof}

\begin{rmk} 
An attempt to  construct  a Ricci flow solution $\ell(t)$ with $\Phi(0) = Id$ and $\ell(0) = \Phi(0)^*g(0) = g(0)$,  using  similar  methods  to those we  use to construct  the Ricci flow solution in Theorem \ref{WeakRicciStart},  could lead to a   non-smooth Ricci flow solution,  which does not immediately become smooth (we say the solution is in {\it a non-smooth gauge}),    as we now explain.
The solutions  $g(t)$ constructed in Theorem \ref{main3start} are limits of solutions
$g_i(t)$ with initial data $g_i(0)$ where $g_i(0) \to g(0)$ in $W_{loc}^{2,2}$. For  $M = \T^4$, the four dimensional Torus, whose circles  have radius $10$  ,  with $h$ the   standard flat metric on $\T^4,$ let $g_i(0) = \phi(i)^*h,$  where $\phi(i):
\T^4  \to \T^4$ are diffeomorphisms, equal to the identity outside a ball $\B_1(0)$ of radius  one (which we identify with the standard euclidean ball of radius one), 
and 
$\phi(i)|_{\B_1(0)}:\B_1(0)  \to  \phi(i)(\B_1(0)) = \B_1(0)  $ are uniformly Bi-Lipschitz diffeomorphsims, $\frac{1}{B}|x-y| \leq   |\phi(i)(x)  - \phi(i)(y)| \leq B |x-y| $ for all $x,y \in \B_1(0),$ with $\phi_i(0) \to \psi$ as $i\to \infty$ in the $W^{3,2}$ sense. Assume that $\psi$ is not smooth.  For example we can take $\phi(i)(x)  =  x (1+ \eta \si   f_i(x)) $ with $
f_i(x) :=   (2 + \sin( \log(\log(\frac{2}{\sqrt{|x|^2+\frac{1}{i}}})))),$ $\si$ a small   positive constant, and $\eta$ a smooth cut-off function with $\eta =1$ on $\B_{1/2}(0),$ $\eta = 0$ on $(\B_{3/4}(0))^c.$ 
Notice that the $\phi(i)$ are uniformly Bi-Lipschitz, as we now explain. 
Assume that $|x|\leq |y|$. Then 
\begin{eqnarray*}
&&|\phi(i)(x)-\phi(i)(y)| = |(x-y) + x\eta \si f_i(x)  - y\eta \si f_i(y) |  \\
&& =  |(x-y) + x\si (\eta(x) f_i(x)  -\eta(y)f_i(y) )   + (x- y)\eta \si f_i(y) | \\
&&  \geq \frac{9}{10} |x-y| - 2\si |x| |\eta(x) - \eta(y)| - 2\si |x-y|  - \si |x| |f_i(x) -f_i(y)|\\
&& \geq \frac{9}{10} |x-y| - 2\si |x| |D_v\eta(c)| |x-y|  - 2\si |x-y|  - \si |x| |D_v f_i(b)| |x-y| \\
&& \geq \frac{1}{2}|x-y| - \si |x| |D_v f_i(b)| |x-y|  \\
\end{eqnarray*}
where $b$ and $c$ are points in the line between $x$ and $y$  and $v$ is a length one vector pointing in the direction of the line between  $x$ and $y$. 
A calculation shows us that 
\begin{eqnarray*}
| D_v  f_i(b)|  && =  |\cos(...)|   |\frac{1}{|\log(\frac{1}{\sqrt{| b|^2+ \frac{1}{i}} })|}  2\frac{\langle b,v\rangle |}{(|b|^2 + \frac{1}{i})}  \\
&&\leq  |\cos(...)|  \frac{1}{|b|},
\end{eqnarray*}
which,  combined with the fact that $|x|\leq |b|,$   gives us $\si |x| |D_v f_i(b)| |x-y|  \leq \si|x-y| $, and hence 
$|\phi(i)(x) -\phi(i)(y)| \geq \frac{1}{4} |x-y|$. A similar calculation shows us that  $|\phi(i)(x) -\phi(i)(y)| \leq 4 |x-y|$.  The definition of the $\phi(i)'s$ guarantees that 
 
$\phi(i):\B_1(0) \to \R^n$ are smooth Bi-Lipshitz diffeomomorphisms whose image lies in $B_1(0)$. Furhtermore $\phi(i)(tx)$ is a continuous line, for $t$ bewteen $0$ and $1$ lying on the standard line $tx$ between $0$ and $x$. Hence, $\{ \phi(i)(tx) \ | \ t \in [0,1] \} = \{ tx \ | \ t \in [0,1]\}.$ This shows that $\phi(i):\B_1(0) \to \B_1(0)$ is also onto.

Then
$\phi(i) \to \psi$ in the $W^{3,2}$ sense,  with $\psi(x) =   x (1+ \eta \si f(x)),$  $ f(x) :=  (2 + \sin( \log(\log(\frac{2}{ |x| }))))$ for $x \neq 0$ $f(0):=0,$ 
 $g_i(0) \to g(0)$  in  the $W^{2,2}$ sense, but $g(0)$ is not smooth, and there exists an $1<a= a(B,K)< \infty$ such that 
  $   \frac{1}{a}h \leq  g_i(0) \leq a h .$ Hence,  Theorem  \ref{main3start} is applicable and a limit
 solution $g(t)_{t\in (0,T)} = \lim_{i\to \infty} g_i(t)_{t\in (0,T)}$ exists with $g(t) \to g(0)$ in the $W^{2,2}$ sense  as $t \downto 0.$ However, the Ricci-Flow of $\ell_i(0) = g_i(0)$ is $\ell_i(t) =  \ell_i(0),$ as the metric $g_i(0)$ is flat.
Hence   $\ell_i \to \ell$  in the $W^{2,2}$ sense where $\ell(t) = \ell(0) = g(0)$ for all $t\in (0,T)$: By construction $g(0)$ is non-smooth. 
We avoid these non-smooth Gauges by choosing  $\Phi(S) = Id$ for some $S>0$ in Theorem \ref{WeakRicciStart}.   

\end{rmk}

In order to prove the  relationships of Theorem \ref{WeakRicciStart}, in particular the existence of the  limit $\ell_0,$  we require some new estimates which hold for solutions to 
Ricci flow  of the type constructed  here, and for a 
 more general class. The theorems, lemmata  that we use to prove  these estimates  are contained in Section \ref{ricciflowestimates}.
  
The existence of  the weak metric $\ell_0$  is achieved with the following theorem

\begin{thm}\label{LpcorStart}
For all $p\in [2,\infty)$ and $n\in \N$ there exists an $\al_0(n,p) >0$ such that the following holds. 
Let $\Omega$ be a smooth $n$-dimensional manifold 
 and $(\Omega^n,\ell(t))_{t\in (0,T]}$      be a smooth solution to Ricci flow   satisfying 
\begin{align*}
&\int_{\Omega} |\Rc(\ell(t))| d\ell(t)  \leq \ep  \\
&|\Rc (\ell(t))| \leq \frac{\ep}{t} \ \mbox{ on } \Omega  \nonumber
\end{align*}
for all $t\in (0,T],$ where $\ep \leq \al_0$.
Then there exists a unique, positive definite,   symmetric two tensor $\ell_0 \in L^p$ such that
$ \ell(s) \to \ell_0$ in $L^p(\Omega)$ as $s \downto 0$ where $\ell_0,$ and $ \ell^{-1}(s) \to (\ell_0)^{-1}$ in $L^p(\Omega)$ as $s \downto 0$.
\end{thm} 

\begin{proof}
See Theorem \ref{Lpboundsriccithm} of Section \ref{ricciflowestimates}: The proof is given there. 

\end{proof}

The  proof of the existence of a homeomorphism $\Phi(0)$ at time zero  in Theorem \ref{WeakRicciStart} can also be applied with no change,  to the  setting of a Ricci DeTurck flow coming out of a $C^0$ metric  on an $n$-dimensional Riemannian manifold, respectively for the Ricci flow related solution. This fact is stated in the following theorem. 
\begin{thm}\label{appstart}
For any $n\in \N$ there exists an $\de_0(n) >0$  such that the following holds. 
Let $(M^n,h)$ be a smooth $n$-dimensional manifold satisfying 
the assumptions \eqref{hassumptionsscaled}, where now $\de_0= \de_0(n)$ is a small constant of our choice, and assume $g_0$ is a $C^0$ metric satisfying
$ (1-\de_0(n))h \leq g_0 \leq (1+\de_0(n))h.$ Let  $(M,g(t))_{t\in (0,T)}$ be the solution to \eqref{Meq}, where
$g(t) \to g_0$ as $t\downto 0$ in the $C^0_{\loc}$ sense,
constructed in \cite{SimonC0} or \cite{KochLamm}, and let $\Phi:M \times (0,T) \to M$ be the solution to  \eqref{ODEDe}, with $\Phi(\cdot,T/2) = Id(\cdot).$ 
Then there exists a homeomorphism $\Phi(0):M \to M$ such that
$\Phi(t) \to \Phi(0)$ locally uniformly, and $d(g(t)) \to d(g(0))=:d_0$
locally uniformly    and 
 $d(\ell(t)) \to \ti d_0$   locally uniformly as $t\downto 0,$ where
 $ \ti d_0(\ti x,\ti y) = d_0( \Phi(0)(\ti x) , \Phi(0)(\ti y))$
 for all $\ti x, \ti y \in M$. 
\end{thm}
\begin{proof}
The solutions constructed in \cite{SimonC0} respectively \cite{KochLamm} satisfy 
$|\gradh g|^2_h(t) + |\grad^2 g|_h(t)  \leq \frac{c(\de_0,n)}{t}$
where $c(\de_0,n) \to 0$ as $ \de_0 \to 0$. 
These  facts are required in the proof of (i) of Theorem \ref{mainthm}. We may now copy and paste the proof of  (i) of Theorem \ref{mainthm} given in Section \ref{Ricciflowrelatedsolution}  to here  and in doing so we  obtain the existence of a homeomorphism $\Phi(0)$ which is obtained locally uniformly as the limit, in the $C^0$ norm, of $\Phi(t)$ with $t\to 0$. \\
Also, the solutions    constructed in \cite{SimonC0} respectively \cite{KochLamm}  satisfy $g(t) \to g(0)$ locally uniformly in the $C^0$ norm as $t\to 0$ and hence 
$d(g(t)) \to d(g(0))$ locally uniformly, and consequently, $d(\ell(t))= (\Phi(t))^*(d(g(t)))  \to \ti d_0 = (\Phi(0))^*(d_0)$ locally uniformly. 
\end{proof}

In Section \ref{anapplication} we prove the following Theorem, Theorem \ref{anappstart}, which  is an application of the above results.  
Compare the paper  \cite{B-G}  : There,   sequences of smooth  Riemannian metrics with scalar curvature bounded from below which approach a $C^0$ metric  with respect to the   $C^0$ norm  are considered. We consider $W^{2,2}$ metrics which have scalar curvature bounded from below in the following weak sense: 

\begin{defn}\label{weakscalar}
Let $M$ be a four dimensional smooth closed manifold and
$g$ be a $W^{2,2}$ Riemannian metric (positive definite everywhere)  
 and let $k\in \R.$ 
Locally the scalar curvature may be written  
\begin{eqnarray*}
\Sc(g) && = g^{jk} (\partial_i \Gamma(g)_{jk}^{i} - \partial_j \Gamma(g)_{ik}^{i} \cr
 && \ \ + \Gamma(g)_{ip}^{i}\Gamma(g)_{jk}^p - \Gamma(g)_{jp}^{i}\Gamma(g)_{ik}^p)
 \end{eqnarray*}
where   $\Gamma(g)_{ij}^m = \frac{1}{2} g^{mk} ( \partial_i g_{jk} + \partial_j g_{ik} - \partial_k g_{ij})$, and hence  
$\Sc(g)$ is well defined in the $L^2$ sense  for a $W^{2,2}$ Riemannian metric.  Let $k\in \R.$ We say the  scalar curvature $\Sc(g)$ is  weakly  bounded from below by $k$, $\Sc(g) \geq k$,  if this is true almost everywhere, for all local smooth coordinates.
\end{defn}

\begin{thm}\label{anappstart}
Let $(M,h)$ be four dimensional closed and satisfy \eqref{hassumptionsscaled}.
Assume that $(M,g_0)$ is a $W^{2,2}$ metric such that
$\frac{1}{a} h \leq g_0\leq a h$  for some $ \infty>a>1$ and $\Sc(g_0) \geq k$ in the weak sense of definition \eqref{weakscalar}.  
Then the solution  $g(t)_{t\in (0,T)}$ to Ricci DeTurck flow  respectively  $\ell(t)_{t\in (0,T)}$ to Ricci Flow constructed in Theorem \ref{mainthm}, with initial value $g(0)=g_0$,  has $\Sc(g(t)) \geq k$ and $\Sc(\ell(t)) \geq k$ for all $t \in (0,T)$.
\end{thm}
\begin{proof}
See Theorem  \ref{sctheorem} : The proof is given there.
\end{proof}

\begin{rmk}
From this theorem we see that   for a metric  $g_0 \in L^{\infty}\cap W^{2,2}(M^4)$ with  $\frac{1}{a} h \leq g_{0} \leq a h$    for some positive constant $a>0$ : $g_0$ has scalar curvature $\geq k$ in the weak sense of Definition \ref{weakscalar} $\iff$ 
there exists a sequence of smooth Riemannian metrics
$ g_{i,0}$ 
 with $\frac{1}{b} h \leq g_{i,0} \leq b h$   for some $1<b<\infty$ 
and 
 $\Sc(g_{i,0}) \geq k $  and  $g_{i,0} \to g_0 \in   W^{2,2}(M^4)$ as $i\to \infty$ 
$\iff$ the Ricci DeTurck flow of $g_0$ constructed in Theorem \ref{main3start} has $\Sc(g(t)) \geq k$ for all $t\in (0,T)$.

\end{rmk}

\section{Outline of the paper}
The paper contains twelve  sections and four appendices,   A,B,C and D.
Section one is an  introduction and Section two contains statements of the main results, and this section gives an outline of  the paper. 
In Section \ref{Linfintygradchapt} we prove a priori  $C^1$ and $L^{\infty}$ estimates for {\it smooth} solutions to the Ricci DeTurck flow. The  $L^{\infty}$ estimates we are concerned with in this paper  take the form $\frac{1}{b}h \leq    g \leq b h$ for some constant $1<b<\infty,$ for the fixed background metric $h$ which is used to define the Ricci DeTurck flow in \eqref{Meq}. 
In particular we show in Theorem \ref{gradientbound},  that smooth compact solutions  which  a priori  have small local $W^{2,2}$ energy along the flow and   satisfy an initial  $L^{\infty}$ estimate must also satisfy $L^{\infty}$ and $C^1$ estimates along the flow. In the non-compact setting, we require further that the smooth solution satisfies a regularity condition in order to obtain the same result : see Theorem  \ref{gradientbound}.  
In Section \ref{smallenergychap}   we prove various local estimates for integral quantities, assuming our solution satisfies an $L^{\infty}$ bound and  
   has  small  local $W^{2,2}$ energy:   See Theorem  \ref{smallenergy}. This also leads to estimates on the convergence as time goes back to zero of the solution, as explained in, for example, Theorem \ref{W22convergence} and    Corollary  \ref{L2continuityCor}.  
Section \ref{existencechap} uses the a priori estimates of the previous sections with well known  existence theory for parabolic equations to show 
that solutions in the classes considered in those sections exist, even when the initial data is non-smooth. That is, solutions to Ricci DeTurck flow exist, if the  initial metric is locally in $W^{2,2}$ and has small local   initial energy and satisfies an $L^{\infty}$ bound with respect to $h$. The solutions obtained 
continue to have small  local  energy and  satisfy  an $L^{\infty}$ bound. 
In  Section \ref{ricciflowestimates} estimates are proved for solutions to Ricci flow in a setting which includes the class of Ricci flows we construct  using the Ricci DeTurck flow of Section \ref{existencechap}. In particular, it is shown  in the setting of Section \ref{ricciflowestimates}, that a weak initial value of the Ricci flow exists.  
 
In Section \ref{Ricciflowrelatedsolution} a Ricci flow is constructed from the Ricci DeTurck flow of of Section \ref{existencechap} and 
in Theorem \ref{mainthm},  the relationship between the two solutions is investigated.  In particular,  relations   between the distance and the weak Riemannian metrics at time zero, as well as the convergence properties as time goes to zero are stated. Further  necessary Lemmata, Theorems, etc., that   we require  to prove this theorem are contained in Sections 
  \ref{ricciflowestimates}, \ref{metricconvergencesobolevspaces} and \ref{distanceconvergencesobolev}.

Section \ref{metricconvergencesobolevspaces} is concerned with convergence properties of Riemannian metrics in certain Sobolev spaces, 
and Section  \ref{distanceconvergencesobolev} is concerned with a  definition  of distance/respectively convergence properties of distances   
for  Riemannian metrics defined in certain Sobolev spaces.
Theorem  \ref{uniqueness} in Section \ref{uniquenesssection} proves uniqueness of the Ricci DeTurck solutions in a  class  which includes the class of solutions that are constructed in this paper.

In Section \ref{anapplication} we present an application for $W^{2,2}\cap L^{\infty}$ metrics with  scalar curvature bounded from below in the weak sense.

Appendix A,B,C,D are technical appendices containing certain  estimates, statements,    the calculation/verification of which,    are not included  in  the other  sections of the paper, in order to facilitate  reading.

Appendix A contains a short time existence result for Ricci DeTurck flow, using the method of W.-X. Shi.
In Appendix B we state and prove some facts about Sobolev inequalities and norms thereof adapted to the setting of the paper.
Appendix C contains estimates for ordinary differential equations which are required  at many points in  the paper.

Appendix D contains statements  which compare pointwise norms and $L^p$ norms of different Riemannian metrics. The estimates contained in the statements are also used at many points of the paper.

\section{$L^{\infty}$- and $C^1$-estimates of the Ricci DeTurck flow}\label{Linfintygradchapt}

In this section we derive an a priori $L^{\infty}$ time independent bound on the evolving metric $g$, and show that the gradient thereof is bounded by $\frac{1}{\sqrt{t}}$ under the a priori assumptions that: 
  we have an $L^{\infty}$ bound $ \frac 1 a h  \leq  g_0 \leq a h$  at time zero, the  $W^{2,2}$ norm of the solution restricted to balls of radius one are  {\it small}, and the time interval being considered has {\it small} length, where here the notion of {\it small} depends on $n$ and $a$.

As a first Lemma, we show that if we already have a $L^{\infty}$ and a time
dependent gradient bound, then all other derivatives may be estimated.

\begin{lemma}\label{smoothnesslemma}
Let $(M,h)$ be $n$ dimensional and satisfy \eqref{hassumptionsscaled} and $g(\cdot,t)_{t \in [0,T)}$, $ T \leq 1$ be a smooth family of metrics which solves
\eqref{Meq} and satisfies the a priori bounds
\begin{align}
 & \frac 1 a h \leq g(t) \leq ah \label{c0c1est} \\
& \sup_{x\in M} |\gradh g|^2(x,t)  \leq \frac b t, \nonumber
\end{align}
for all $t \in [0,T)$ and some $1<a,b<\infty$.
Then for all $i \in \N$ there exist constants $N_i = N_i(a,b,n,h)$   such that
\begin{equation}
\sup_{x\in M} |\gradh^i g|^2(x,t)  \leq 
\frac {N_i} {t^i}\label{c2est}
\end{equation}
for all $t\in [0,T).$ 

\end{lemma}
\begin{proof}
We start with the case $i=2$.
Let $g(\cdot,t)_{t \in [0,T)}$, $ T \leq 1$  be a solution to \eqref{Meq}
which satisfies \eqref{c0c1est}.
Let $N_2 \in \N$ be large (to be determined in the proof) and 
assume that \eqref{c2est} doesn't hold. That is,  for $N := N_2$,
there exists a $0 <t_0 < T$ and an $x_0 \in M$ such that
$|\gradh^2 g|^2(x_0,t_0)  > \frac {N} {t^2_0}$.
Define $\ti g(x, t):= c g(x, \frac { t} c)$ and $\ti h(y) =
ch(y)$ for a $c>0$ to be chosen.  Then $\ti g( t)$ solves the
$\ti h$ -flow for $t \in [0,T c) $ and we have the scaling relations
\[
|  {\gradtih}^i  \ti g|^2(x,t)=c^{-i} |\gradh^i g|^2(x,\frac{t}{c}).
\] 
By choosing $c = \frac {\sqrt{N}} {\sqrt{V} t_0}$, we get a solution $\ti g$ which has
\begin{align*}
 & \frac 1 a \ti h \leq \ti g(t) \leq a \ti h \\
& \sup_{x\in M} |  \gradtih  \ti g|^2( x,t)  \leq \frac {b} {t} \\
& |\gradtih^2  \ti g|^2(x_0,\sqrt{\frac N V})  \geq   {V}
\end{align*}
for all $t \in [0,t_0c] = [0,  \sqrt{\frac N V}]$.
This implies $|\gradtih \ti g|^2(t)  \leq \frac {b}
{t} \leq \frac {b}  { \sqrt{\frac N V} -10  } \leq \ep $ for $ t \in (-10+ \sqrt{\frac N V}, \sqrt{\frac N V}]   $  for any $\ep>0$ as long as $N= N(b,V,n,\ep)$ is chosen
large enough.
In that which follows we use once again $g$ to denote the solution $\ti
g$ and $h$ to denote $\ti h$. That is, we have a smooth solution
$g(t)_{t \in [0,\sqrt{ \frac{N}{V} } ] } $ of the Ricci DeTurck flow with
\begin{align*}
 & \frac 1 a h \leq g(t) \leq a h \\
&  \sup_{x\in M} |\gradh  g|^2(x,t)  \leq \ep \\
& |\gradh^2 g|^2(x_0,\sqrt{\frac N V} )  \geq {V}
\end{align*}
 for all $ t\in  (-10+  \sqrt{\frac{N}{V}}, \sqrt{\frac{N}{V}} ]$.
As shown in \cite{Shi2}, the evolution for $|\grad^m g|^2$ is given by
\begin{align}
\partt &  |\grad^m g |^2(x,t)  =g^{ij}(x,t) \grad^2_{ij} |\grad^m g|^2 (x,t)
 - 2g^{ij}(x,t)
(\grad_i \grad^m g ,\grad_j \grad^m g)_h(x,t) \cr
& + 
\sum_{  \substack{0 \leq k_1,k_2, \ldots ,k_{m+2} \leq m+1, \\k_1 + \ldots +
  k_{m+2} \leq m+2}} \grad^{k_1}g(x,t) * \grad^{k_2}g(x,t) * \ldots
\grad^{k_{m+2}}g(x,t) * \grad^m g * P(h)_{k_1 k_2 \ldots k_{m+2}}(x,t) \label{ShiEqny}
\end{align}
where 
\begin{align*} 
P(h)_{k_1,\ldots,k_{m+2}} (x,t)= 
P(h)_{k_1,\ldots,k_{m+2}}(x,t)(g,g^{-1},\Rm(h),\grad \Rm(h), \ldots,
  \grad^m \Rm(h))
\end{align*}
is a polynomial in the terms appearing in the brackets,  and $$ g^{ij}(\grad_i T, \grad_j T) 
= g^{ij}h^{s_1 r_1}h^{s_2 r_2}\ldots h^{s_m r_m} \grad_i T_{s_1 \ldots s_m} \grad_j T_{r_1 \ldots r_m}$$
for a $(0 \ \ m)$ tensor $T$. We have
$|P(h)|^2(x,t) \leq c(a,m,n)$
since without loss of generality, the norm of the curvature of $h$
( after scaling ) and all its
derivatives up to order $m$ are bounded by a constant (see \eqref{hassumptionsscaled}).
In particular, for $m=1$ we obtain
\begin{align}
\partt& |\grad g |^2(x,t) - g^{ij}(x,t) \grad^2_{ij} |\grad g|^2 (x,t)\leq 
 -2 g^{ij}(x,t) (\grad_i \grad g ,\grad_j \grad g)_h(x,t)\nonumber \\
&+  \sum_{  \substack{0 \leq k_1,k_2 ,k_{3} \leq 2,\\ k_1 + k_2+
  k_{3} \leq 3}} \grad^{k_1}g(x,t) * \grad^{k_2}g(x,t) * 
\grad^{k_{3}}g(x,t) * \grad g * P(h)_{k_1 k_2 k_{3}}(x,t)\nonumber \\
&\leq 
 -2 g^{ij}(x,t) (\grad_i \grad g , \grad_j \grad g)(x,t) \nonumber \\
&+ c(n,a) | \grad g| ( |\grad^2 g||\grad g| + |\grad^2 g| + |\grad
g|^3 + |\grad g|^2 + |\grad g| +c(n,a)). \label{m=1}
\end{align}
 Here $c(n,a)$ denotes a constant which may change from line to line but only depends on $n$ and $a$.
Combinations of constants involving $b,a,n$ multiplied by $\ep$ shall
sometimes be written as $\ep$.
In what follows, we restrict ourselves to the region  
$t  \in (-10+  \sqrt{\frac N V}, \sqrt{\frac N V}]$.
Using $\sup_{x\in M}|\grad g|(x,t) \leq \ep \leq 1$ for all $t  \in (-10+  \sqrt{\frac N V}, \sqrt{\frac N V}]$ and $ \frac{1}{a} h \leq g \leq a h $ we get 
\begin{align*}
\partt& |\grad g |^2(x,t) - g^{ij}(x,t) \grad^2_{ij} |\grad
g|^2 (x,t) \leq - \frac{2}{a} |\grad^{2} g |^2(x,t)\cr
&+ c(n,a) \ep  | ( |\grad^2 g|\ep + |\grad^2 g| + 3\ep +c(n,a)) \cr
&\leq -  \frac{2}{a} |\grad^{2} g |^2(x,t) + c(n,a) \ep   |\grad^2 g| +
c(n,a) \ep \cr
&\leq -\frac{1}{a} |\grad^{2} g |^2(x,t) +  c(n,a) \ep
\end{align*}
in view of Young's inequality. 
Similarly we estimate
\begin{align}
\partt & |\grad^2 g |^2(x,t) - g^{ij}(x,t) \grad^2_{ij} |\grad^2 g|^2
(x,t) \nonumber \\
& \leq 
 -\frac{2}{a}|\grad^3 g|^2 
+  \sum_{  \substack{0 \leq k_1,k_2,  ,k_{3},k_4 \leq 3,\\ k_1 + \ldots +
  k_{4} \leq 4}} \grad^{k_1}g(x,t)  * \ldots
\grad^{k_{4}}g(x,t) * \grad^2 g * P(h)_{k_1 k_2 k_3 k_4}(x,t)\nonumber \\
&\leq 
 -\frac{2}{a}|\grad^3 g|^2 + c(n,a) |\grad^2 g| \Big( |\grad^3 g| ( |\grad g| + c(n,a))
+|\grad^2g|(|\grad^2g| + |\grad g|^2 + |\grad g| + c(n,a)) \label{m=2} \\
&
 \ \ \ \  \ \ \ \ \ \ \ \ \ \ \ \   \ \ \ \ \ \   \ \  \ \ \ \ \ \  \ \ \ \ \ \  \ \ \  + |\grad g| (|\grad g|^3 +  |\grad g|^2 +  |\grad g| +c(n,a))
 +c(n,a) \Big) \nonumber \\
 &\leq 
 -\frac{2}{a}|\grad^3 g|^2 
+c(n,a) |\grad^2 g|( |\grad^3 g|c(n,a) + |\grad^2g|^2 + c(n,a))\nonumber \\
&\leq -\frac{1}{a}|\grad^3 g|^2  +c(n,a)|\grad^2 g| + c(n,a)|\grad^2g|^3 \nonumber \\
& \leq -\frac{1}{a}|\grad^3 g|^2 + c(n,a)|\grad^2g|^3  + c(n,a), \nonumber
\end{align}
where we have used Young's inequality a number of times.
Combining these two evolution inequalities, we see that 
$f = (|\grad g|^2 + 1)(|\grad^2 g|^2)$ satisfies
\begin{align*}
\partt f &-g^{ij}\grad^2_{ij} f  \cr 
& \leq - \frac{1}{a}|\grad^2 g|^4  +  \ep |\grad^2 g|^2 \cr
&+ (|\grad g|^2 + 1)( -\frac{1}{a}|\grad^3 g|^2 + c(n,a)|\grad^2g|^3  +
c(n,a)) \cr
& - 2g^{ij} \grad_i  (|\grad g|^2 + 1)\grad_j (|\grad^2 g|^2)  \cr 
&\leq - \frac{1}{2a}|\grad^2 g|^4  -\frac{1}{a}|\grad^3 g|^2 
+c(n,a)(1+ |\gradh g|^2)(1+  |\gradh^2g|^3 )
+c(n,a) |\grad g||\grad^2 g|^2|\grad^3 g| 
\end{align*}
for the $t$ that we are considering. 

Now using once again that   $\sup_{x\in M} |\grad g|(x,t) \leq \ep$, which is true by assumption, we see that
\begin{align*}
\partt &f -g^{ij}\grad^2_{ij} f  \leq -(1/4a)f^2 +c(n,a).
\end{align*}
Standard techniques (cut off function and a Bernstein type argument:
see \cite{Shi} or \cite{SimonC0}) now show that
$f \leq c_1(n,a)$ at $t = \sqrt{\frac N V}$ which implies that
 $|\grad^2 g|^2 \leq c_1(n,a)$ at $t = \sqrt{\frac N V}$ and this contradicts the estimate $|\grad^2 g|^2(x_0, \sqrt{\frac N V}) \geq {V}$
 if $V = \max(100 c_1(n,a),r(h,m)),$ where $r(h,m) $ is chosen large so that 
 the curvature of $h$ and all of its covariant derivatives up to order $m=2$ are bounded by one after scaling (which was used in the proof). 
 
 For the readers convenience, we explain the Bernstein type argument  in more detail.
 By translating in time, we may assume that the time $\sqrt{\frac N V}$ corresponds to time $10$ and the time $\sqrt{\frac N V} -10 $ corresponds to time zero.
 We multiply $f$ by a cut off function $\eta$ in space (with support in a Ball $B_1(y_0)$ ball around any point $y_0$) and such that $\frac{|\grad\eta|^2}{\eta} \leq C$. Next we consider a point $(x_0,t_0) \in  B_1(0) \times [0,10] $ where $t \eta f$ achieves its positive maximum (assuming $f$ is not identically zero).  
 The point $x_0 $ must be in the interior of $B_1(0)$, since the support of $\eta$ is contained in $B_1(0))$, and $t_0$ must be larger than zero, since $ t\eta f= 0$ for $t=0$,  and hence 
, by calculating at $(x_0,t_0)$, we obtain 
 \begin{eqnarray*}
0&& \leq \partt(t f\eta) \cr
&&\leq g^{ij}\grad^2_{ij} (t f \eta)  -\frac{1}{4a} tf^2 \eta -2g^{ij}\grad_i (tf) \grad_j \eta \cr
&&\ \ - (tf)g^{ij} \grad^2_{ij} \eta +tc(n,a)\eta + f\eta \cr
&& \leq -\frac{1}{4a} \frac{(t\eta f)^2}{\eta t}-\frac{ 2g^{ij}\grad_i (tf\eta) \grad_j \eta}{\eta} + c(n,a)tf\frac{|\grad \eta|^2}{\eta}
+ c(n,a) \frac{ft\eta}{\eta} +c(n,a) +f\eta  \cr
&& \leq -\frac{1}{4a} \frac{(\eta t f)^2}{\eta t} +c(n,a) \frac{f t\eta}{\eta t}  + c(n,a),
\end{eqnarray*}
where we used that $\grad_i (tf\eta)(x_0,t_0) = 0$ and Young's inequality.  
Multipying by $t \eta$, we see that 
$\frac{1}{4a}  (\eta t f)^2(x_0,t_0)  - c(n,a)   (t\eta) f(x_0,t_0)     \leq  c(n,a)$ 
and hence 
$f(x_0,t_0)\eta(x_0,t_0)  t_0 \leq  \hat c(n,a) \eta(x_0,t_0) t_0$ which implies $f(x_0,t_0) \leq  c(n,a)$. 

Next we assume by induction that for $i\geq 2$
\[
\sup_{x\in M} |\gradh^m g|^2(x,t)  \leq \frac {N_m} {t^m}
\]
for all $m\leq i$, $t\in [0,T)$ and we want to show that there exists a constant $N_{i+1}$ so that
\[
\sup_{x\in M} |\gradh^{i+1} g|^2(x,t)  \leq \frac {N_{i+1}} {t^{i+1}}
\]
for all $t\in [0,T)$.
Again we argue by contradiction. For this we assume that there is a large constant $N$ (to be determined later) and $x_0\in M$ resp. $0<t_0<T$ so that
\[
|\gradh^{i+1} g|^2(x_0,t_0) >N/ t_0^{i+1}.
\]
Using the same scaling argument as above we can arrange that we obtain a solution $g$ of the $h$-flow so that
\begin{align*}
\frac{1}{a} h\leq& g \leq ah,\\
\sup_{x\in M} |\gradh^m g|^2(x,t)\leq& \varepsilon \ \ \ \forall t\in (-10+\sqrt[i+1]{\frac{N}{V}},\sqrt[i+1]{\frac{N}{V}}],\, m\leq i \\
|\gradh^{i+1} g|^2(x_0,\sqrt[i+1]{\frac{N}{V}}) \geq& V.
\end{align*}
As before, resp. as in the paper of Shi, (\cite{Shi}, proof of Lemma 4.2), we now obtain
\[
\frac{\partial}{\partial t} |\gradh^i g|^2\leq g^{jk} \grad^2_{jk} |\gradh^i g|^2-\frac{1}{2a} |\gradh^{i+1} g|^2+c(i,n,a)
\]
and

\[
\frac{\partial}{\partial t} |\gradh^{i+1} g|^2\leq g^{jk} \grad^2_{jk} |\gradh^{i+1} g|^2-\frac{1}{2a} |\gradh^{i+2} g|^2+c(i,n,a)|\gradh^{i+1} g|^2+c(i,n,a).
\]
for all $t\in (-10+\sqrt[i+1]{\frac{N}{V}},\sqrt[i+1]{\frac{N}{V}}]$: The first estimates $|\grad g|^2(\cdot,t) + |\grad^2 g|^2(\cdot,t) \leq c/t$ simplify  the calculation for general $i\geq 2$ (See \eqref{ShiEqny} ).  Calculating as before we thus obtain for $f:=|\gradh^{i+1} g|^2(1+|\gradh^i g|^2)$ that
\[
\partial_t f -g^{jk} ( \gradh^2_{jk} f)\leq-\frac{1}{4a} f^2+c(i,n,a)
\]
for all $t\in (-10+\sqrt[i+1]{\frac{N}{V}},\sqrt[i+1]{\frac{N}{V}}]$ and we obtain the same contradiction as before.
\end{proof}

Now we show that if we
have $L^{\infty}$ control on our initial metric and 
the $L^2$ norm of the gradient and the second
gradient of $g$ remain locally, uniformly small, then we have an estimate on the
$L^{\infty}$ norm of the evolving metric (and its inverse) and a time
dependent gradient estimate.

\begin{thm}\label{gradientbound}
For every  $1 \leq a \in \R$  $n\in \N$ there exist (small) $\ep_0(a,n),S_1(a,n)>0$ such that the following holds. Let $g_0$ be smooth and satisfy
\begin{align}\label{init0}
 & \frac{1}{a} h \leq g_0 \leq a h 
\end{align}
where  $(M,h)$ is an  $n$ -dimensional manifold satisfying \eqref{hassumptionsscaled},
and assume that we have a smooth solution $g$ to \eqref{Meq} on
$[0,T]$ which satisfies
\begin{align} 
\int_{B_1(x)}( |\gradh g |^{\frac n 2 }  +  |\gradh^2 g |^{\frac n 2 })(t)dh\leq  \ep_0 \label{dcon}
\end{align}
 for all $x \in
M$ for all $t \in [0,T]$ where $T < 1$. We also assume $\sup_{M \times
  [0,T]} |\gradh
g|^2 +  | \gradh^2 g|^2 +   | \gradh^3 g|^2 +  F+ \varphi
< \infty,$  where $\varphi(x,t): = g^{ij}(x,t)h_{ij}(x)$ and $F(x,t):=
g_{ij}(x,t)h^{ij}(x).$
Then 
\begin{align}
  \frac{1}{ 20 n a} h  < g(t) <& 20 n a h \label{c0est} \\
\sup_{x\in M} |\grad g|^2(x,t) <& \frac{1}{t} \label{c1est}
\end{align}
for all $t \leq S_1(n,a).$ 
\end{thm}
\begin{rmk}
The functions $\varphi(x,t): = g^{ij}(x,t)h_{ij}(x)$ and $F(x,t):=
g_{ij}(x,t)h^{ij}(x)$ are both well defined smooth functions.
The assumption that 
\begin{equation}\label{control}
\sup_{M \times
  [0,T]} |\gradh
g|^2 +  | \gradh^2 g|^2 +  | \gradh^3 g|^2  + F+ \varphi
< \infty 
\end{equation}
is always satisfied on a compact manifold due to smoothness and
compactness.

We will use this result in the proof of Theorem \ref{main1} and in that situation this condition is satisfied. 
\end{rmk}
\begin{proof}
We may replace the condition \eqref{dcon}, by the scale invariant condition
 \begin{align} 
\int_{B_1(x)} (|\gradh g|^{n} + |\gradh^2 g |^{ \frac n 2})dh \leq  \ep_0  \label{better}
\end{align}
for all $x \in M$ for all $t\in [0,T]$ in view of (v) in Lemma \ref{balllemma} of Appendix \ref{geoapp},
after replacing $c(n)\ep_0$ by $\ep_0$.
Let  $\de = \de(n,a)= \frac{1}{(an)^{100}}<<1$ (we are assuming $n\geq 2$).
Let 
\[
S_1 = \sup \{ s \in [0,T] \ \ |  \ \ \frac{1}{ 20 n a} h \leq g(t) \leq 20 n a h \ \ \ \text{and}  \ \ \ \sup_{x\in M} |\grad g|^2(x,t) \leq \frac{\de}{t} \ \ \ \text{hold on} \ \ 
[0,s] \}.
\]
We have $S_1 >0$ due to the inequality \eqref{control} and the fact that $g$
satisfies \eqref{Meq}. Next we want to show that $S_1$ can be bounded from below by a constant depending only on $n$ and $a$.

For this we argue by contradiction and we assume that $S_1$ is extremely small, so that
if we rescale the background metric $h$ by $1/S_1$, then the resulting Riemannian manifold is as close to the
standard euclidean space $\R^n$ on balls as large as we like  in the $C^k$ norm ($ k
\in \N$ chosen as we please) in geodesic coordinates, due to the
conditions on $h$, as we explained at the beginning of the paper.

Let us now scale $g$ and $h$ via
$\ti g(x,  t) = c g(x,  \frac {t}{c} )$, $\ti h = ch$ with
$c = 1/S_1 >>1$.  We denote $\ti g$ and $\ti h$ once again by $g$ resp. $h$.
We have for the rescaled solution that 
$1 = S_1 =  \sup \{ s \in [0,T) | $ \eqref{c0est} and \eqref{c1est} hold on
$[0,s] \}$  and \eqref{better} still holds, and hence, 
\eqref{dcon} holds, in view of H\"older's inequality, after replacing $c(n)\ep_0$ by $\ep_0$.
Due to the definition of $S_1  (=1)$, the smoothness of all metrics and \eqref{control}  we see that,
\begin{eqnarray*}
 &\frac 1  {20 n a} h  \leq g \leq 20 n ah  \cr
&\sup_{x\in M} |\grad g|^2(x,t) \leq \frac{\de}{t} 
\end{eqnarray*}
for all $t \in [0,1]$, and
there must exist a point $x_0\in M$ with either
\begin{align*}
&(a) \ \ g(x_0,1)(v,v) < \frac {1}{10na}h(x_0)(v,v) = \frac {1}{10na} \\
 & \ \ \ \ \ \ \ \ \ \ \   \mbox{ for some } h \mbox{ length one vector } v,  \  \mbox{ or }\\
 &(b) \ \ g(x_0,1)(v,v)>10 n a h(v,v) = 10na \\
 & \ \ \ \ \   \ \ \ \ \ \ \ \mbox{ for some } h \mbox{ length one vector } v,  \  \mbox{ or }\\
& (c) \ \     |\grad g|^2(x_0,1)  >  \frac{\de}{2} ,
\end{align*}
otherwise, using the the smoothness of $g$ and \eqref{control}, we get a contradiction to the definition of $S_1=1$. 

We rule out the case (c) first.
We argue by contradiction and by the smoothness Lemma \ref{smoothnesslemma}, 
we know that $|\grad g|^2(\cdot,1) \geq \frac  \de 4$ on a ball  of
radius $R(n,a,\de)= R(n,a)>0$ around $x_0$, and hence
\begin{align*}
\ep_0 \geq \int_{B_1(x_0)} (|\grad g|^2(y,1)\, dh(y) \geq \frac{\de(n,a)}{ 8} \omega_n (R(n,a,\de))^n
\end{align*}
which leads to a contradiction if $\ep_0= \ep_0(n,a) $ is chosen small
enough. Note that here we used that the manifold is very close to the euclidean space.
This contradiction shows that (c) doesn't occur.

Now we rule out (a) and (b).
Note that in our case
\begin{align}
\int_{B_1(x)} \varphi(y,0)\, dh(y) \leq  \int_{B_1(x)} na \, dh(y) \leq \frac{3}{2}\omega_n na \ \ \ \text{resp.} \label{init1} \\
\int_{B_1(x)} F(y,0)\, dh(y) \leq  \int_{B_1(x)} na \, dh(y) \leq \frac{3}{2}\omega_n n a, \label{init2} 
\end{align}
where we have used the initial conditions \eqref{init0}.
From the evolution equation for $g$, we have
\begin{align*}
\partt \int_{B_1(x)} \varphi dh \leq  c(n,a) \int_{B_1(x)}( |\grad g|^2 +
|\grad^2 g| + |\Riem(h)|)dh \leq C_S(n)c(n,a) (\ep_0)^{\frac 2 n}  \\
\partt \int_{B_1(x)} F dh \leq  c(n,a) \int_{B_1(x)} (|\grad g|^2 +
|\grad^2 g| + |\Riem(h)| )dh\leq C_S(n)c(n,a)(\ep_0)^{\frac 2 n} 
\end{align*}
and thus
\begin{align*}
 (\int_{B_1(x_0)} \varphi(x,1) dh(x) )  & \leq \int_{B_1(x_0)} \varphi_0 dh +  C_S(n)c(n,a)(\ep_0)^{\frac 2 n} \cr
 & \leq (3/2)\omega_n nra+  c(n,a)(\ep_0)^{\frac 2 n}  \cr
 & \leq  2 \omega_n na
 \cr
 (\int_{B_1(x_0)} F(x,1)dh(x))  &\leq \int_{B_1(x_0)} F_0 dh+  c(n,a)(\ep_0)^{\frac 2 n}\cr
  & \leq
 (3/2)\omega_n  n a +  c(n,a)(\ep_0)^{\frac 2 n} \leq   2\omega_n n a
\end{align*}
if $\ep_0(n,a)$ is sufficiently small. 
Here we used the initial conditions \eqref{init1} and \eqref{init2} freely, and the H\"older and Sobiolev inequalities to obtain $\int_{B_1(x)}( |\grad g|^2 +
|\grad^2 g| dh \leq C_S(n)c(n,a) (\ep_0)^{\frac 2 n} $.
In particular, there must be a point $y_0$ in $B_1(x_0)$ with
$\varphi (y_0,1) \leq  4 na$ and a point $y_1$ in  $B_1(x_0)$ with
$F (y_1,1) \leq  4 na$.
First we consider $\varphi$. At $y_0$ we choose a basis so that $h_{ij}(y_0) =
\de_{ij}$ and $g_{ij}(y_0,1) = \la_i \de_{ij}$ is diagonal.
Then we see that $\varphi (y_0,1) \leq  4 na$ implies that
$\la_i \geq \frac 1 {4na}$ for each $i \in 1, \ldots,n$ and hence
$g(y_0,1)  \geq  \frac 1 {4na}h(y_0)$.
Using the fact that $|\grad g| \leq \de$ and that $ (M,h)$ is very
close to the standard $\R^n$, in particular 
$|\Gamma(h)^i_{jk}(x)| \leq \eta_0$ in geodesic coordinates on a ball
of radius $10$ centred at $x_0$ where  $\eta_0$ as small as we like,
we get
$|\partial_i g_{kl} | \leq \de + \si(\eta_0,a,n)$ with 
$\si(\eta_0,a,n) \to 0$ as $\eta_0 \to 0$  and hence, without loss of
generality,  $\si(\eta_0,a,n)
\leq \de $, that is 
$|\partial_i g_{kl} | \leq 2\de = \frac{2}{(an)^{100}}$ on $B_1(x_0)$
in geodesic coordinates (for $h$).
This combined with  $g(y_0,1)  \geq  \frac 1 {4na}h(y_0),$  and the
fact that $ (1-\eta_0)\de_{ij} \leq h_{ij} \leq (1+\eta_0) \de_{ij}$ ( $\eta_0$
as small as we like ) leads
to $g_{ij}(y,1)  \geq  ( \frac 1 {4na} -\frac{4}{(an)^{100}} )h_{ij}(y)
\geq   \frac 1 {8na} h_{ij}(y)$ 
for all $y \in B_1(x_0)$  which contradicts the fact that 
$g(x_0,1)(v,v) <\frac {1}{10n a}h(x_0)(v,v) = \frac {1}{10n a}$. 
Hence (a) doesn't occur.
The argument to show that  (b) doesn't occur is
essentially the same.

\end{proof}

\section{Preservation of smallness of the $W^{2,2}$ Energy and $W^{2,2}$ continuity of $g$ 
 in time}\label{smallenergychap}
In this section we consider smooth four dimensional solutions to the Ricci DeTurck flow which satisfy  $\frac{1}{a}h \leq   g(t) \leq a h$ for some uniform constant $a$ and our fixed background metric $h,$ and whose initial $W^{2,2}$-energy is locally small.  
Under these assumptions, we prove an estimate on the growth of the local $W^{2,2}$-energy, which  shows that this smallness is preserved under the flow, if the time interval being considered is small enough. 
We see, in Theorem \ref{L2continuity} and Theorem \ref{W22convergence},  that these estimates also imply 
estimates on the modulus of continuity of the local $L^2$ and   $W^{2,2}$-energy of a solution, respectively limits of smooth solutions  to \eqref{Meq}.
\begin{thm}\label{smallenergy}
Let $(M,h)$ be four dimensional and satisfy \eqref{hassumptionsscaled}. 
For all $0<a \in \R$,  there exists a $\de= \de(a)>0$  such that
for any smooth  solution $g\in C^\infty(M\times [0,T))$ of the Ricci-DeTurck flow with 
\begin{eqnarray*}
&& \sup_{x\in M} \int_{B_1(x)}( |\gradh g(\cdot,t)|^2 + |\gradh^2 g(\cdot,t) |^2)dh \leq  \delta \ \ \ \text{and} \\
&&  \frac {1}{a} h(\cdot) \leq g(\cdot,t) \leq ah(\cdot)  
\end{eqnarray*} 
for all $t \in  [0,T),$
the following holds : For every $\frac 1 2 \leq R_0<R_1\leq 2$  there exists an $V(R_0,R_1,a)>0$ such that 
\begin{eqnarray*}
&& \int_{B_{ R_0 }(x_0)}( |\gradh g(\cdot,t)|^2 + |\gradh^2 g(\cdot,t) |^2)dh \\
&& \leq   \int_{B_{ R_1 }(x_0)}( |\gradh g(\cdot,0)|^2 + |\gradh^2 g(\cdot,0) |^2)dh + V(R_0,R_1,a)t  
\end{eqnarray*}
for any $x_0 \in M$, for all $t \in [0,T).$
\end{thm}
\begin{rmk}
The condition $\sup_{x\in M} \int_{B_1(x)}( |\gradh g(\cdot,t)|^2 + |\gradh^2 g(\cdot,t) |^2)dh \leq  \delta$ means we restrict to the class of solutions which stay locally small in $W^{2,2}$. Later we will see that this is not a restriction for the solutions that we construct, starting with initial data which is locally sufficiently small in $W^{2,2}$,  as they do indeed satisfy this condition.
\end{rmk}

\begin{rmk}
The constant $V(R_0,R_1,a) \to \infty$ for $R_0\upto R_1$.
\end{rmk}

\begin{cor}\label{smallenergycor}
Let $(M,h)$ be four dimensional and satisfy \eqref{hassumptionsscaled}. 
For all $0<b \in \R$,  there exists a $\de= \de(b)>0$ and universal constant $c_0>0$ such that the following holds. 
For every $\ep>0$ there exists an $S_2 = S_2(b, \ep)>0$ such that if 
 $g\in C^\infty(M\times [0,T))$ is a smooth solution to  the Ricci-DeTurck flow with initial data $g_0$ which satisfies
 \begin{eqnarray*}
&& \sup_{x\in M} \int_{B_1(x)}( |\gradh g(\cdot,t)|^2 + |\gradh^2 g(\cdot,t) |^2)dh \leq  \delta \ \ \ \text{and} \\
&& \frac {1}{b} h(\cdot) \leq g(\cdot,t) \leq bh(\cdot) 
\end{eqnarray*} 
for all $t \in  [0,T),$
 and   for some $x_0\in M$   we have 
\begin{eqnarray*}
  \int_{B_2(x_0)}( |\gradh g_0|^2 + |\gradh^2 g_0 |^2) dh \leq  \ep.
\end{eqnarray*}
 Then 
\begin{eqnarray*}
  \int_{B_{1 }(x_0)}( |\gradh g(\cdot,t)|^2 + |\gradh^2 g(\cdot,t) |^2)dh \leq \frac{3}{2}\ep,
\end{eqnarray*}
for all $t \in [0,S_2) \cap [0,T).$\\
\end{cor}

{\it Proof of Corollary \ref{smallenergycor}}.
Theorem \ref{smallenergy} implies
\begin{eqnarray*}
&& \int_{B_{1 }(x_0)}( |\gradh g(\cdot,t)|^2 + |\gradh^2 g(\cdot,t) |^2)dh \cr
&& \leq   \int_{B_{ 2 }(x_0)}( |\gradh g(\cdot,0)|^2 + |\gradh^2 g(\cdot,0) |^2)dh + V(R_0,R_1,b)t \cr
&& \leq \ep + V(1,2,b)t\cr
&& \leq  \frac{3}{2}\ep
\end{eqnarray*}

for $t \leq  \frac{\ep}{  2 V(1 , 2,b)} =:S_2$ 
$\ \ \  \ \ \ \ \  \ \ \ \ \ \ \ \  \ \ \ \  \Box $

Before proving  Theorem \ref{smallenergy} we need a version of the Gagliardo-Nirenberg inequality. 
\begin{lemma}\label{est}
Let $(M,h)$ be four dimensional and satisfy \eqref{hassumptionsscaled}, $\frac 1 2  \leq R_0 < R_1 \leq 2$ and  $g$ be a smooth metric on $M$ satisfying  $ \frac{1}{a}g \leq h \leq ag,$ $x \in M$ and  $\eta \in C^\infty_c(B_{R_1}(x_0))$ be a standard cut-off function which is equal to $1$ on $B_{R_0}(x)$ and equal to zero outside of $B_{\frac{R_1 + R_0}{2}}(x_0):$  We choose  $\eta$  so  that $\sqrt{\eta}\in C^\infty(B_{R_1}(x))$ with $|\grad \eta|\leq c(R_1,R_0)$ for some constant $c(R_1,R_0).$ Then there exists a $C=C(a,R_0,R_1)$ and a $B=B(a)$ such that
\begin{align*}
 \int_M \eta^4 |\gradh^2 g|^3 dh \le\ &  \left(\int_{B_{R_1}(x)} |\gradh^2 g|^2dh\right)^{\frac12} \left( \int_{B_{R_1}(x)} (B \eta^4 |\gradh^3 g|^2+ C |\gradh^2 g|^2)dh \right)\\
\int_M \eta^4 |\gradh g|^6 dh \le&   B \left(\int_{B_{R_1}(x)} |\gradh^2 g|^2dh\right)^{\frac12} \left(\int_{B_{R_1}(x)} \eta^4 |\gradh^3 g|^2  dh\right) \\
& +C\left(\int_{B_{R_1}(x)} |\gradh g|^2dh\right)^{\frac12} \left(\int_{B_{R_1}(x)} ( |\gradh^2 g|^2+|\gradh g|^2)dh \right)\nonumber \\
& +C\left(\int_{B_{R_1}(x)} |\gradh^2 g|^2dh\right)^{\frac32}. 
\end{align*}
\end{lemma}
{\it Proof of Lemma \ref{est}}. 
In the following $C$ refers to a constant which depends on $a,R_0,R_1$  and $B$ refers to a constant which only depends on
$a$ . Both constants may vary from line to line, but continue to be denoted by $C$ respectively $B$.
Using H\"older's inequality and the Sobolev inequality applied to the function  $f=\eta^2 |\gradh^2 g|,$ we obtain
\begin{align}
\int_M \eta^4 |\gradh^2 g|^3 dh \le& (\int_{B_{R_1}(x)} |\gradh^2 g|^2 dh)^{\frac12}(\int_M \eta^8|\gradh^2 g|^4 dh)^{\frac12} \cr
\le& (\int_{B_{R_1}(x)} |\gradh^2 g|^2dh)^{\frac12} \left(\int_{B_{R_1}(x)} (B \eta^4 |\gradh^3 g|^2+C\eta^2 |\gradh^2 g|^2)dh \right), 
\label{firstest} 
\end{align}
which is the first estimate.
For the second estimate we integrate by parts with respect  one of the covariant derivatives $\grad$ and use H\"older's inequality, to get
\begin{align*}
\int_M \eta^4 |\gradh g|^6 dh\le& B\int_M \left( \eta^4 |g| |\gradh g|^4 |\gradh^2 g|+C\eta^3 |g| |\gradh g|^5\right)dh\\
\le& B\big(\int_M \eta^4 |\gradh g|^6 dh\big)^{\frac23}\Big( (\int_M \eta^4 |\gradh^2 g|^3 dh)^{\frac13}+C(\int_M \eta |\gradh g|^3 dh)^{\frac13}\Big)
\end{align*}
which implies, 

\begin{align*}
\int_M \eta^4 |\gradh g|^6 dh\le& B\int_M \eta^4 |\gradh^2 g|^3dh+C\int_M \eta |\gradh g|^3 dh\\
 \le& B\int_M \eta^4 |\gradh^2 g|^3dh+C (\int_{B_{R_1(x)}}   |\gradh g|^2 dh)^{\frac 1 2} \cdot (\int_M (\sqrt{\eta} |\gradh g|)^4 dh)^{\frac 1 2}.
\end{align*}
Using  \eqref{firstest}  to estimate the first term of the right hand side of this inequality, and 
the Sobolev inequality, applied to the function  $\sqrt{\eta}|\gradh g|$  to estimate the second   term, we conclude
\begin{align*}
\int_M \eta^4 |\gradh g|^6 dh\le& B\int_M \eta^4 |\gradh^2 g|^3dh+C (\int_{B_{R_1(x)}}   |\gradh g|^2 dh)^{\frac 1 2} \cdot (\int_M (\sqrt{\eta} |\gradh g|)^4 dh)^{\frac 1 2} \\
\le&  B(\int_{B_{R_1}(x)} |\gradh^2 g|^2dh)^{\frac12} \left(\int_{B_{R_1}(x)} (\eta^4 |\gradh^3 g|^2+C|\gradh^2 g|^2)dh \right)\\
&+C(\int_{B_{R_1}(x)} |\gradh g|^2 dh)^{\frac12} \left(\int_{B_{R_1}(x)} ( |\gradh^2 g|^2+|\gradh g|^2)dh \right),
\end{align*}

as required. Note, with loss of generality $|\grad \sqrt{\eta}|\leq C_1(n=4):$ if not replace $\eta$ by $\eta^2$.
$ \ \ \ \ \ \ \ \ \ \ \ \ \ \ \  \ \ \ \ \  \ \ \ \ \ \ \Box$

In the following proof,  $C$ will once again be a constant which may change from  line to line and depends on $a,R_0,R_1$,
and  $B$ denotes a constant  which can change from line to line but only depends on $a$. \\
{\it Proof of Theorem \ref{smallenergy}}.
Using equation \eqref{m=1}, and Young's inequality, we see that 
\begin{align}
\partt& |\grad g |^2- g^{ij} \grad^2_{ij} |\grad g|^2 +\frac{2}{a} |\grad^2 g|^2\cr
& \le  B | \grad g| \left( |\grad^2 g| |\grad g|  + |\grad^2 g|  + |\grad
g|^3  + |\grad g|^2  + |\grad g| + 1 \right) \cr
& \le  \frac{1}{2a}|\grad^2 g|^2  + B ( |\grad g|^4 + 1)
\label{gradgeqn}
\end{align} 
Integration by parts (once), and Young's inequality  yields 
\begin{align}
|\int_M \eta^4 g^{ij} \grad^2_{ij} |\grad g|^2  dh| \le& B\int_M \eta^4  |\gradh g|^2 |\gradh^2 g| dh+C\int_M \eta^3 |\gradh g| |\gradh^2 g| dh,  
\cr
\le &  \frac{1}{2a} \int_M  \eta^4  |\gradh^2 g|^2   +   C \int_M  \eta^2 (|\grad g|^4 + 1)dh,  
\label{intgradgeqn}
\end{align} 
for $\eta$ a standard cut-off function as in Lemma \ref{est}. 
Multiplying the above differential inequality \eqref{gradgeqn}  with $\eta^4$, integrating and using the inequality
 \eqref{intgradgeqn}, we get
\begin{align*}
\partt& (\int_M \eta^4 |\grad g |^2 dh )+\frac{2}{a}\int_M \eta^4 |\grad^2 g|^2dh \\
\le& \frac{1}{2a}\int_M \eta^4 |\grad^2 g|^2dh+ C \int_M \eta^2 ( |\gradh g|^4 + 1 )dh\\
\le& \frac{1}{2a}\int_M \eta^4 |\grad^2 g|^2dh+ C\left(\int_{B_{R_1}(x_0)} (|\gradh^2 g|^2+|\gradh g|^2+1)dh\right)^2,
\end{align*}
where we used  the Sobolev inequality, applied to the function $ f= \sqrt{\eta}|\grad g|,$ 
 and  $|\grad \sqrt \eta|^2 = \frac{|\grad \eta|^2}{4 \eta} \leq C$   in the last step.  Absorbing the first term on the right hand side into the left hand side and integrating from $0$ to $S$ we conclude 
\[
 \int_M \eta^4 |\grad g |^2(\cdot,S) dh \le  \int_M \eta^4 |\grad g_0 |^2 dh+CS(\delta+1)^2 
\]
for all $S\in [0,T]$. 
Now we turn to the corresponding estimate for the second derivatives. Recalling \eqref{m=2} and using Young's inequality, we see that
\begin{align*}
\partt & |\grad^2 g |^2(x,t) - g^{ij}(x,t) \grad^2_{ij} |\grad^2 g|^2+\frac{2}{a}|\grad^3 g|^2
(x,t) \\
\leq&  B |\grad^2 g| \Big( |\grad^3 g| ( |\grad g| + 1)\\
&+|\grad^2g|(|\grad^2g| + |\grad g|^2 + |\grad g| + 1) \\
&
 + |\grad g| (|\grad g|^3 +  |\grad g|^2 +  |\grad g| +1)
 +1 \Big)\\
\leq& \frac{1}{2a} |\grad^3 g|^2+B \big( |\grad^2 g|^3+|\grad^2 g|^2+|\grad g|^6+|\grad g|^2 +1)
\end{align*}
holds. As above, we note that integration by parts (once) followed by applications of  Young's inequality yields
\begin{align*}
|\int_M \eta^4 g^{ij}(x,t) \grad^2_{ij} |\grad^2 g|^2 (x,t) dh| \le& B\int_M \eta^4  |\grad g| |\grad^2 g| |\grad^3 g| dh\\
&+C\int_M \eta^3 |\grad^2 g| |\grad^3 g| dh, \\
\le& \int_M   \frac{\eta^4}{4a} |\grad^3 g|^2  + B \eta^4  |\grad g|^2 |\grad^2 g|^2dh\\
&+C\int_M \eta^2 |\grad^2 g|^2  dh,\\
 \leq & \int_M  \frac{\eta^4}{4a} |\grad^3 g|^2  + B \eta^4 |\grad^2 g|^3 + B \eta^4  |\grad g|^6 + C\eta^2 |\grad^2 g|^2
  dh
\end{align*} 
Multiplying the differential inequality for $|\grad^2 g |^2$ again with $\eta^4$, integrating and using the above two estimates, we obtain
\begin{align*}
\partt& (\int_M \eta^4 |\grad^2 g |^2 dh) +\frac{1}{a}\int_M \eta^4 |\grad^3 g|^2dh \\
&\leq \int_{B_{R_1} }  B \eta^4( |\grad^2 g|^3 +   |\grad g|^6)dh + C(|\grad^2 g|^2 +|\grad g|^2+1) dh .
\end{align*}
Using the  estimates from Lemma \ref{est}, and the assumption, with this estimate, we see that this implies
\begin{align*}
 \partt (\int_{M} \eta^4 |\grad^2 g |^2  ) &+\frac{1}{4 a} \int_M \eta^4 |\grad^3 g|^2 dh \\
 \leq&    C \int_{B_{R_1}}\left(   |\grad^2 g|^2+ |\grad g|^2+1\right) + C   \Big( \int_{B_{R_1}(x)} |\grad^2 g|^2 dh\Big)^{\frac 3 2 } \\
 &+ C\Big( \int_{B_{R_1}} |\grad g|^2 dh\Big)^{\frac 1 2} \Big( \int_{B_{R_1}} |\grad^2g|^2 +|\grad g|^2 dh \Big)     \\
  & + B \Big(\int_{B_{R_1}(x)} |\grad^2 g|^2 dh \Big)^{\frac 1 2  } \Big( \int_{B_{R_1}} \eta^4 |\grad^3 g|^2 dh \Big) \\
 \leq & C  \Big(\int_{B_{R_1}}  |\grad^2 g|^2+ |\grad g|^2+1 dh\Big) + C\Big(\int_{B_{R_1}}  |\grad^2 g|^2+ |\grad g|^2+1 dh\Big)^2  \\
&  + B \de  \int_{B_{R_1}} \eta^4 |\grad^3 g|^2 dh
\end{align*}
and hence 
\begin{eqnarray*}
&& \partt (\int_M \eta^4 |\grad^2 g |^2 dh)  +\frac{1}{8 a} \int_M \eta^4 |\grad^3 g|^2 dh
 \leq   C(\delta+1)^2
\end{eqnarray*}
if $B(a) \de \leq \frac{1}{10 a} $. 
Integration in time from $0$ to $S$ gives 
\begin{align*}
\int_M \eta^4 |\grad^2 g |^2(\cdot,S) dh \le  \int_M \eta^4 |\grad^2 g_0 |^2 dh+CS(\delta+1)^2 
\end{align*}
as required.  $ \ \ \ \ \ \ \ \ \   \ \ \ \ \ \ \ \ \ \ \ \ \ \ \ \ \ \ \ \ \  \ \ \ \  \ \ \ \ \ \ \ \ \ \ \ \ \ \ \ \ \  \ \ \ \  \ \ \ \ \ \ \ \ \ \ \ \ \ \ \ \ \ \ \ \  \ \ \ \ \Box$

\begin{lemma}\label{L2continuity}
Let $(M,h)$ be $n$-dimensional and satisfy \eqref{hassumptionsscaled}, $g$ be  a smooth solution to \eqref{Meq} on $M \times (0,T],$ $T\leq 1$ and assume that there exist $0<a\in \R$, $\de  \in \R$   so that
\begin{align*} 
&\frac {1}{a} h \leq g(\cdot,t) \leq ah\ \ \ \text{and}  \\
& \sup_{x\in M} \int_{B_1}|\gradh g(\cdot,t)|^2 dh \leq \de\ \ \ \forall \, t\in (0,T], \\
&K_0:= \sup_{x\in M} |\Rm(h)|  \leq 1
\end{align*}
Then there exists a $B=B(n,a)$ such that  
$$\int_{B_1(x)}  |g(\cdot,t)-g(\cdot,s)|^2  dh \leq B(\de+K_0)| t -s|,$$ 
$$\int_{B_1(x)}  |g^{-1}(\cdot,t)-g^{-1}(\cdot,s)|^2  dh \leq B(\de+K_0)| t -s|,$$
$$ |\int_{B_1(x)}  |g(\cdot,t)|^2 -|g(\cdot,s)|^2  dh|  \leq B(\de+K_0)|t-s|,$$ and
$$ |\int_{B_1(x)}  |g^{-1}(\cdot,t)|^2 -|g^{-1}(\cdot,s)|^2  dh|  \leq B(\de+K_0)|t-s|,$$
for all $x \in M$ for all $t,s \in (0,T]$ and for all $x \in M$.
\end{lemma}
\begin{cor}\label{L2continuityCor}
Let $(M,h)$ be $n$-dimensional and satisfy \eqref{hassumptionsscaled}, $g$ be a smooth solution to \eqref{Meq} on $M \times [0,T],$ $T\leq 1$ and assume that there exist $0<a\in \R$, $\de  \in \R$   so that
\begin{align*} 
&\frac {1}{a} h \leq g(\cdot,t) \leq ah\ \ \ \text{and}  \\
& \sup_{x\in M} \int_{B_1}|\gradh g(\cdot,t)|^2 dh \leq \de\ \ \ \forall \, t\in [0,T],\\
&K_0:= \sup_{x\in M} |\Rm(h)|  \leq 1
\end{align*}
Then there exists a $B=B(n,a)$ such that  
$$\int_{B_1(x)}  |g(\cdot,t)-g(\cdot,0)|^2  dh \leq B(\de+K_0)| t,$$
$$\int_{B_1(x)}  |g^{-1}(\cdot,t)-g^{-1}(\cdot,0)|^2  dh \leq B(\de+K_0)| t,$$  
$$ |\int_{B_1(x)}  |g(\cdot,t)|^2 -|g(\cdot,0)|^2  dh|  \leq B(\de+K_0)t$$ 
and
$$ |\int_{B_1(x)}  |g^{-1}(\cdot,t)|^2 -|g^{-1}(\cdot,0)|^2  dh|  \leq B(\de+K_0)t,$$ for all $t \in [0,T]$ and for all $x \in M$.
\end{cor}
{\it Proof of Corollary \ref{L2continuityCor}}. 
For any sequence $t_i \to 0$ and any $x \in M$, we have $\int_{B_1(x)}  |g(\cdot,t)-g(\cdot,t_i)|^2  dh \leq B(\de+1)| t -t_i|.$  
 Letting $i \to \infty$ implies the first estimate   view of the smoothness of the solution. The other estimates follow with an almost identical argument.
$ \ \ \ \ \Box$\\
\begin{proof}[Proof of Lemma  \ref{L2continuity}:] 
We calculate for a standard  cut-off function $\eta$ with $\eta=1$ on $B_{1}(x)$ and $\eta = 0$ on $(B_2(x))^c$ and $|\grad \eta|^2\leq c(n)|\eta|$ that  
\begin{align*}
 & \partt ( e^{-B(n,a)t} \int_{M}      \eta |g(\cdot,t) -g(\cdot,s)|^2  dh) \cr 
 &  =
e^{-B(n,a)t}  \int_{M}  \eta \partt h^{ik} h^{jl} ( g(t)_{ij} - g(s)_{ij} )( ( g(t)_{kl} - g(s)_{kl} ) dh\cr
& \ \ \ \ - B(n,a)e^{-B(n,a)t} \int_{M}      \eta |g(\cdot,t) -g(\cdot,s)|^2  dh \cr
& =e^{-B(n,a)t} \int_{M} 2   \eta h^{ik} h^{jl} \partt g(t)_{ij} ( ( g(t)_{kl} - g(s)_{kl} )dh\cr
& \ \ \ \ - B(n,a)e^{-B(n,a)t} \int_{M}      \eta |g(\cdot,t) -g(\cdot,s)|^2  dh \cr 
& = e^{-B(n,a)t}\int_{M} 2  \eta  h^{ik} h^{jl} \curlL(g(t),h)_{ij} ( g(t)_{kl} - g(s)_{kl} dh)\cr
& \ \ \ \ - B(n,a)e^{-B(n,a)t} \int_{M}      \eta |g(\cdot,t) -g(\cdot,s)|^2  dh  \cr
& \leq  e^{-B(n,a)t}B(n,a)  \int_{B_2(x)} ( |\gradh g(t)|^2  + K_0  +  |\grad g(s)|^2 )dh   \cr
&   \ \ \ \
+ e^{-B(n,a)t}B(n,a)  \int_{M}  \eta |g(\cdot,t) -g(\cdot,s)|^2  dh  \cr
& \ \ \ \ - B(n,a)e^{-B(n,a)t} \int_{M}      \eta |g(\cdot,t) -g(\cdot,s)|^2  dh \cr  
& \leq  e^{-B(n,a)t}B(n,a)(\de+K_0) , 
\end{align*}
where $\curlL(g(t),h)_{ij}$ is the right hand side of the equation \eqref{Meq}, and  
 we used integration by parts, with respect to $\grad$,  in the second to last step, and a covering of  $B_2(x)$ by $c(n)$ balls of radius one in the last step.
 Integrating from $s$ to $t$ implies the first estimate. 
Also,
 \begin{align*}
 & \partt ( e^{-B(n,a)t} \int_{M}      \eta |g^{-1}(\cdot,t) -g^{-1}(\cdot,s)|^2  dh) \cr 
 &  =
e^{-B(n,a)t}  \int_{M}  \eta \partt h_{ik} h_{jl} ( g(t)^{ij} - g(s)^{ij} ) ( g(t)^{kl} - g(s)^{kl} ) dh\cr
& \ \ \ \ - B(n,a)e^{-B(n,a)t} \int_{M}      \eta  |g^{-1}(\cdot,t) -g^{-1}(\cdot,s)|^2 dh \cr
& =e^{-B(n,a)t} \int_{M} 2   \eta h^{ik} h^{jl} g^{iv}g^{jw}\partt g(t)_{vw} ( ( g(t)^{kl} - g(s)^{kl} )dh\cr
& \ \ \ \ - B(n,a)e^{-B(n,a)t} \int_{M}      \eta  |g^{-1}(\cdot,t) -g^{-1}(\cdot,s)|^2  dh \cr 
& = e^{-B(n,a)t}\int_{M} 2  \eta  h^{ik} h^{jl}g^{iv}g^{jw}  \curlL(g(t),h)_{vw } ( ( g(t)^{kl} - g(s)^{kl} ) dh)\cr
& \ \ \ \ - B(n,a)e^{-B(n,a)t} \int_{M}     \eta  |g^{-1}(\cdot,t) -g^{-1}(\cdot,s)|^2 dh  \cr
& \leq  e^{-B(n,a)t}B(n,a)  \int_{B_2(x)} ( |\gradh g(t)|^2  + K_0  +  |\grad g(s)|^2 )dh   \cr 
& \ \ \ \ 
+ e^{-B(n,a)t}B(n,a)  \int_{M}  \eta  |g^{-1}(\cdot,t) -g^{-1}(\cdot,s)|^2   dh  \cr
& \ \ \ \ - B(n,a)e^{-B(n,a)t} \int_{M}      \eta  |g^{-1}(\cdot,t) -g^{-1}(\cdot,s)|^2  dh \cr  
& \leq  e^{-B(n,a)t}B(n,a)(\de+K_0).
\end{align*}

 Integrating from $s$ to $t$ implies the second estimate.
 
Also,
\begin{align*}
& \partt ( e^{-B(n,a)t}\int_{M}   \eta |g(\cdot,t)|^2 dh ) \cr
&  =
 e^{-B(n,a)t} \int_{M} \eta \partt h^{ik} h^{jl}  g(t)_{ij}  g(t)_{kl} dh  
 -B e^{-B(n,a)t} \int_{M}   \eta |g(\cdot,t)|^2 dh )\cr 
& =e^{-B(n,a)t} \int_{M} 2  \eta h^{ik} h^{jl} \partt g(t)_{ij} g(t)_{kl} dh
 -B e^{-B(n,a)t} \int_{M}   \eta |g(\cdot,t)|^2 dh )\cr 
& =e^{-B(n,a)t} \int_{M} 2  \eta  h^{ik} h^{jl} \curlL(g(t),h)_{ij}  g(t)_{kl} dh 
 -B e^{-B(n,a)t} \int_{M}   \eta |g(\cdot,t)|^2 dh )\cr 
& \leq  e^{-B(n,a)t}B(n,a)  \int_{B_2(x)} ( |\gradh g(t)|^2   +K_0 )dh   +   Be^{-B(n,a)t}B(n,a)  \int_{M} \eta |g(\cdot,t)|^2 dh\cr
&  \ \ \ \ \ 
 -B e^{-B(n,a)t} \int_{M}   \eta |g(\cdot,t)|^2 dh )\cr 
& \leq  e^{-B(n,a)t}B(n,a)(\de + K_0) ,
\end{align*}
where we used integration by parts in the second to last step. Integrating from $s$ to $t$ implies the third  estimate.
The fourth estimate follows similarly. 
\end{proof}

The previous Lemma  showed us that solutions which are smooth on $M \times [0,T]$ and whose  $W^{1,2}$ energy is bounded on balls of radius one by $\de$, and which are uniformly (independent of time) equivalent to $h$, $\frac{1}{a}h \leq g(t) \leq a h$, converge strongly in the $L^2$ norm back to $g_0$ as time goes to zero, and the rate of convergence depends only on $a,n$ and $\de$.
If we  only assume that the solution is smooth on $M \times (0,T)$, then the previous Lemma shows us that the solution is Cauchy in the $L^2(B_1(x))$ norm in time, and hence there exists a well defined $L^2_{loc}$ limit, $g_0$  at time $t=0$ on $M$.
In the four dimensional setting, the assumption that the solution is bounded in $W^{2,2}(B_1(x))$  for all $x \in M$ means that there must be a sequence of times $t_i$, such that  $g(t_i)$ 
 converge weakly in $H:= W^{2,2}(B_1(x))$ back to  $g_0 \in H$ (the details are given in the proof of Theorem \ref{W22convergence}.).
The estimates of Lemma  \ref{est}, show us that in fact the convergence is also strong in $W^{2,2}_{loc},$ if the solution is a limit of smooth solutions, whose initial data converge locally strongly in $W^{2,2}$, as is explained in the following theorem. Note that this is  precisely the situation which we study in the next section.
\begin{thm}\label{W22convergence}
For all $0<a  \in \R$  there exists $\de=\de(a)$ so that the following holds.
Let $(M,h)$ be four dimensional and satisfy \eqref{hassumptionsscaled} and 
$(M,g(t),p)_{t\in (0,T]}$ be the smooth limit, on compact subsets of $M \times (0,T]$, of $(M,g(i)(t),p_i)|_{t\in (0,T]}$ as $i\to \infty$ 
of a  sequence of smooth solutions $g(i)$ to \eqref{Meq} defined on $M \times [0,T]$ which satisfy 
\begin{eqnarray}
&& \frac {1}{a} h \leq g(i)(\cdot,t) \leq ah \label{uniformbound} \ \ \ \text{and} \\
&& \sup_{x\in M} \int_{B_1(x)}( |\gradh g(i)(\cdot,t)|^2 + |\gradh^2 g(i)(\cdot,t) |^2)dh \leq \de\label{W22bound}
\end{eqnarray}
for all $t \in  [0,T]$. Assume further that  the initial data $g(i)(\cdot,0)$ converge strongly locally in $W^{2,2}_{loc}$ to some $g_0 \in W^{2,2}_{loc}$ as $i\to \infty$, that is $\norm{g(i)(0) -g_0}_{W^{2,2}(K)} \to 0$ for all compact sets $K \subseteq M$  as $i\to \infty$. Then,  $g(t) \to g_0$ as $t\downto 0$ locally strongly in the   
$W^{2,2}$ norm, that is 
\begin{eqnarray}
&& \int_{B_1(x)}  |g(\cdot,t)-g_0(\cdot)|^2 dh 
 +  \int_{B_1(x)}  |\grad (g(\cdot,t)-g_0(\cdot))|^2(\cdot,t)  dh \\
&& \ \ \ \ + \int_{B_1(x)}  |\grad^2 (g(\cdot,t)-g_0(\cdot))|^2(\cdot,t)  dh
\to 0\nonumber
\end{eqnarray}
as $t\downto 0,$ for any $x \in M$. 
\end{thm}
\begin{proof}

The  solutions $g(i)$ defined on $M\times [0,T]$ are smooth and satisfy the hypotheses of Theorem \ref{smallenergy}. Without loss of generality $T \leq 1$. Hence, the conclusions of that theorem hold and we get
\begin{eqnarray*}
&& \int_{B_{ R_0 }(x_0)}( |\gradh g(i)(\cdot,t)|^2 + |\gradh^2 g(i)(\cdot,t) |^2)dh \cr
&& \leq   \int_{B_{ R_1 }(x_0)}( |\gradh g(i)(\cdot,0)|^2 + |\gradh^2 g(i)(\cdot,0) |^2)dh + V(R_0,R_1,a)t 
\end{eqnarray*}
for any $x_0 \in M$ and for all $t \in [0,T)$. The third estimate of Corollary \ref{L2continuityCor} implies additionally for any $x_0\in M$ and all $t\in [0,T),$ ( remembering $T\leq 1$)
\[
 \int_{B_{ R_0 }(x_0)} | g(i)(\cdot,t)|^2 \, dh  \leq   \int_{B_{ R_1 }(x_0)} |g(i)(\cdot,0)|^2dh+ V(R_0,R_1,a)t.
 \] 
Letting $i \to \infty$ for fixed $t \in (0,T)$, and using that the solutions converges smoothly locally away from  $t=0$ and in the $W^{2,2}_{loc}$ norm 
at time zero, we see that the limit solution also satisfies
\begin{eqnarray}
&& \int_{B_{ R_0 }(x_0)}( |\gradh g(\cdot,t)|^2 + |\gradh^2 g(\cdot,t) |^2 + |g(t)|^2  )dh \cr
&& \leq   \int_{B_{ R_1 }(x_0)}( |\gradh g_0|^2 + |\gradh^2 g_0 |^2 + |g_0|^2 )dh + V(R_0,R_1,a)t \label{mainresult2}
\end{eqnarray}
for any $x_0 \in M$, for all $t \in (0,T),$ that is
\begin{eqnarray}
 \norm{ g(t)}^2_{ W^{2,2}(B_{R_0}(x_0))}    \leq   \norm{ g_0}^2_{W^{2,2}(B_{R_1}(x_0))}  + V(R_0,R_1,a)t \label{mainresult3}
\end{eqnarray}
for any $x_0 \in M$, for all $t \in (0,T),$ where $  \norm{ T }^2_{W^{2,2}(B_{R_0}(x_0))} 
:=  \int_{B_{ R_0 }(x_0)}( |T|^2_h + |\gradh T|^2 + |\gradh^2T|^2   )dh$
for any zero-two tensor defined on $B_{ R_0 }(x_0)$ whose components are in $W^{2,2}$.
Furthermore, we have $\int_{B_1(x_0)} |g(t)-g_0|^2 dh \to 0$ for $t\downto 0$ in view of Corollary \ref{L2continuityCor}, and the fact that $g(i)(0) \to g_0$ as $i \to \infty$ in the $W^{2,2}_{loc}$ norm. 
 
Fixing $x_0\in M,$  and $R_0:=1$  we define the Hilbert space $H:= W^{2,2}(B_1(x_0))$ to be the space of zero-two tensors whose components are in $W^{2,2}(B_1(x_0))$ 
and whose scalar product is defined by 
$ (T,S)_H:= \int_{B_{1 }(x_0)}  (T,S)_h + (  \grad T,\grad S)_h +  (\gradh^2 T, \grad^2 S)_h   dh .
$  Using this notation, we may write \eqref{mainresult2} as 
\begin{eqnarray}
(g(t),g(t))_H  \leq    \norm{ g_0}^2_{W^{2,2}(B_{R_1}(x_0))}  + V(R_0,R_1,a)t  \label{mainresult4}
\end{eqnarray}
 for all $t \in (0,T),$ for any $2>R_1>R_0=1.$
  We are going to show that every sequence $(g(t_i))_{i\in \N}$ with $t_i \downto 0$ contains a strongly converging subsequence with limit $g_0$ in $H$. This then clearly implies that  $g(s) \to g_0$ in $H$ as $s \downto 0$ since otherwise we can find a sequence $t_i \to 0$, and a $\de>0$ s.t. $\norm{g(t_i) -g_0}_{H} >\de$, and hence no subsequence of $g(t_i)$ will converge to $g_0$ in $H$, which would be a contradiction.
  
  Now to the details.
For  any sequence of times $0<t_i \to 0$ as $i\to \infty,$
there exists a subsequence, $g(t_{i_j})=: g_{i_j} $ of $g(t_i)$ such that $g_{i_j} \weak z$ as $j\to \infty$ for some $z \in H,$ in view of the definition of a Hilbert space and weak convergence.

But $g_{i_j}$ must then converge strongly to $z$ in $L^2(B_1(x_0))$ and hence $z=g_0$. 
Setting $r_j := t_{i_j}$ this means we have 
$g(r_j) \to g_0$ strongly in $L^2(B_1(x_0))$ and $g(r_j) \weak g_0$ weakly
in $H$. It remains to show that 
$g(r_j) \to g_0$ strongly in $H=W^{2,2}(B_1(x_0))$ for all $x_0\in M$.

Assume that this  is not true for some $x_0 \in M$.
Then we can find a subsequence $s_k := r_{i_k},$ $k\in \N$ of $(r_{k})_{k\in \N}$ and a $\de>0$ such that
$ \norm{g(s_k) -g_0}^2_{H} \geq \de >0$ for all $k\in \N$.
But then 
\begin{eqnarray*}
\de && \leq (g(s_k) -g_0, g(s_k) -g_0)_{H}  \cr
&& = (g(s_k), g(s_k))_H + (g_0,g_0)_{H} -2(g(s_k),g_0)_H \cr
&& = \norm{g(s_k)}^2_{H} + \norm{g_{0}}^2_{H} -2(g(s_k),g_0)_H  
\end{eqnarray*}
for all $k\in \N$, and hence, using \eqref{mainresult4}, 
\begin{eqnarray}
\de && \leq  \norm{ g_0}^2_{W^{2,2}(B_{R_1}(x_0))} + \norm{g_{0}}^2_{H} -2(g(s_k),g_0)_H + V(1,R_1,a)s_k.
\end{eqnarray}
Since $g_0$ in $W^{2,2}(B_2(x_0))$ (here we use the covering argument from Lemma \ref{balllemma}), there must exist a $ 1<R_1<2$ such that 
\[
\norm{g_{0}}^2_{W^{2,2}(B_{R_1}(x_0)) } \leq \norm{g_{0}}^2_{H} +\frac{\de}{8} 
\] 
and hence we obtain
\begin{eqnarray*}
\de && \leq   2\norm{g_{0}}^2_{H} -2(g(s_k),g_0)_H  + \frac{\de}{8} + V(1,R_1,a)s_k 
\end{eqnarray*}
for this choice of $R_1$ independent of  $k\in \N$.
Letting $k \to \infty,$ we obtain a contradiction, since $2\norm{g_{0}}^2_{H} -2(g(s_k),g_0)_H  \to 0$ as $k\to \infty$.
\end{proof}

\section{Existence and regularity}\label{existencechap}

In this section we prove the main results for the Ricci DeTurck flow of data which is initially $W^{2,2}$ .

\begin{thm}\label{main1}
Let $1<a<\infty$ and $(M,h)$ be four dimensional and satisfy \eqref{hassumptionsscaled}. Then there exists a constant $\ep_1 =\ep_1(a) >0$ with the
following properties.
Assume  $g_0$ is
a smooth Riemannian metric on $M$  which is uniformly bounded  in $W_{\loc}^{2,2} \cap L^{\infty}$ in the following sense:
\begin{align} 
& \ \  \frac{1}{a}  h \leq g_0 \leq  a h  \tag{a}  \label{aaa}\\
& \ \ \int_{B_2(x)} ( |\gradh g_0|^2 +   |\gradh^2 g_0|^2)dh \leq \ep
\ \ \mbox{ \rm for all } \ \ x \in M  ,
 \tag{b} \label{bbb}
 \end{align}
where $\ep \leq \ep_1(a),$
and $g_0$ satisfies 
$$\sup_M  |\grad^i g_0|^2 < \infty $$ for all $i\in N$.
Then there exists   constants $T=T(a,\ep)>0$  and  $c_j= c_j(a,h)>0,$  and a smooth solution
$(g(t))_{t \in [0,T]}$ to \eqref{Meq} with $g(\cdot,0) = g_0(\cdot)$ such that

\begin{align}
 & \ \ \ \ \frac{1}{400 a  } h \leq g(t) \leq 400 a h  \tag{${\rm a}_t
$} \label{aaa_t}  \\
&  \ \ \ \ \int_{B_1(x)} ( |\gradh g(\cdot,t)|^2 +   |\gradh^2 g(\cdot,t)|^2 )dh< 2\ep \ \  \mbox{ \rm for all }  x \in M  , \ t 
\in [0,T]  \tag{${\rm b}_t$}  \label{bbb_t} \\
 & \ \ \ \ |\gradh^jg(\cdot,t)|^2 \leq \frac{c_j(a,h)}{t^{ j}}  
 \tag{${\rm c}_t$}  \label{ccc_t} 
\end{align}
and 
$$ \sup_M  |\gradh ^j g(t)|^2 < \infty $$ for all $j\in N$ for all $t\in [0,T].$
\end{thm}
\begin{proof}

Using the existence theory for Parabolic equations, for example the  method of Shi, Theorem  \cite{Shi} Section 3 and 4 (which in turn uses   Theorem 7.1, Section VII of \cite{LSU} ) we see that we have a solution to \eqref{Meq} for a short time $[0,V]$ for some $V>0$ and  $\sup_{M \times[0,V]}|\gradh^j g|^2<\infty$ for all $j \in \N$: See  Theorem \ref{standardexist}.
We assume, without loss of generality,  that $\ep_1(a) \leq \frac{\ep_0}{4} =  \frac{\ep_0(a,4)}{4},$  where $\ep_0$ is the constant from  Theorem  \ref{gradientbound}.   
Let $\hat S:= \sup\{  t\in [0,V] \ | \  \eqref{bbb_t}$ holds for $t\leq \hat S \} $:  $\hat S >0$ due to smoothness (and boundedness of covariant derivatives of $g$). We have 
$|\grad g(\cdot,t)|\leq \frac{1}{t}$ and  $\frac{1}{80 a  } h \leq g(t) \leq 80 a h ,$ 
 for  $t\leq \min( \hat S,S_1(4,a) )$, in view of Theorem \ref{gradientbound}, where $S_1(4,a)$ is the constant from that theorem. 
Hence,  for such $t \leq \min(S_1(4,a),\hat S )$, we have 
$|\grad^i  g(\cdot,t)|^2\leq \frac{ N_i(80a,1,4,h) }{t^i}$  in view of 
Lemma \ref{smoothnesslemma}.
Also,   using  Corollary
 \ref{smallenergycor} with $b=80a$,  we  can improve the estimate \eqref{bbb_t} to  $\int_{B_1(x)} (|\grad g|^2 + |\grad g|^2) dh \leq  \frac 3 2 \ep < 2 \ep,   $ 
for $t\leq \min(S_1(4,a),\hat S, S_2(b=80a,\ep) )$, where $S_2(b,\ep)>0$ is the constant from that Corollary,  since $2\ep \leq 2 \ep_1(a)   <  \ \ep_0$ 
   and without loss  of generality, 
   $ \ep_0 \leq \de= \de(80a)$ , where $\de$ is the constant appearing in that   corollary. 
Hence $\hat S \geq \min(S_1(4,a), V, S_2(b=80a,\ep) ),$ and  
\eqref{aaa_t},\eqref{bbb_t} and \eqref{ccc_t} hold for $t\leq  \min(S_1(4,a), V, S_2(b=80a,\ep) ),$ 
with $c_1=2,$ $ c_i = N_i( 80a,1,4,h)$ for all $ i\in \N, i\geq 2.$ 
Applying   Theorem \ref{standardexist}   and repeating this argument as often as necessary,    we may extend this solution to a smooth   solution   
$g(t)_{t\in [0,  T]}$  satisfying \eqref{aaa_t},\eqref{bbb_t} and \eqref{ccc_t} for $t\leq    T := \min(S_1(4,a), S_2(b=80a,\ep) ) .$   The estimates 
$\sup_{M \times[0,V]}|\gradh^j g|^2<\infty$ for all $j \in \N$ and \eqref{ccc_t} guarantee that
$\sup_{M \times[0,T]}|\gradh^j g|^2<\infty$ for all $j \in \N$.

\end{proof}
\begin{rmk}\label{usefulrema}
In the proof of Theorem \ref{main1}   we obtained  $ \int_{B_1(x)} ( |\gradh g(\cdot,t)|^2 +   |\gradh^2 g(\cdot,t)|^2 )dh< \frac{3}{2}\ep \ \  \mbox{ \rm for all }  x \in M  , \ t 
\in [0,T] ,$ and $ \frac{1}{80a} h g(t) \leq  80a h $ for $ t 
\in [0,T]$ : we will use these facts in the proof of the next theorem.
\end{rmk}

\begin{thm}\label{main2}
Let  $(M,h)$ be four dimensional and satisfy  \eqref{hassumptionsscaled} and  $\infty> a >1$. Assume  $g_0$  is a $W^{2,2}_{\loc}  \cap L^{\infty}$  Riemannian metric, not necessarily smooth, which satisfies (\ref{aaa}) and (\ref{bbb}), where $\ep \leq \frac{\ep_1(2 a)}{2}$  and $\ep_1$ is the constant from Theorem \ref{main1}.

Then there exists a constant $S=S(a,\ep)>0$ and a smooth solution
$(g(t))_{t \in (0,S]}$ to \eqref{Meq}, where $g(t)$ satisfies
 (\ref{aaa_t}),(\ref{bbb_t}) and (\ref{ccc_t}) for all $x \in M$, for all $t \in (0,S],$ and 
\begin{align}
 \int_{B_1(x)} (|g_0 - g(t)|^2 + |\gradh(g_0 - g(t))|^2 + |\gradh^2(g_0 - g(t))|^2)dh \to 0  \ \ \ \ \  \tag{${\rm d}_t$}  \label{ddd_t}  
 \end{align} 

as $t\downto 0$ for all $x \in M$. 
The solution is unique in the class of solutions satisfying  (\ref{aaa_t}),(\ref{bbb_t}) , (\ref{ccc_t})  and  (\ref{ddd_t}).
The solution also satisfies the local estimates
\begin{align}
 &\sup_{x\in B_1(x_0)} |\gradh^jg(\cdot,t)|^2t^j   \to  0 \mbox{ for }  t \to 0 \ \ \  \tag{${\rm e}_t$}  \label{eee_t} \\ 
 & \int_{B_1(x_0)}  (|\gradh g(\cdot,t)|^2 +   |\gradh^2 g(\cdot,t)|^2)dh    \tag{${\rm f}_t$}  \label{fff_t}  \\
 & \leq
 \int_{B_{R}(x_0)} ( |\gradh g_0(\cdot)|^2 +   |\gradh^2 g_0(\cdot)|^2 )dh + V(a,R)t    \nonumber
\end{align}
for all $x_0\in M$, $2>R>1$  for all $t \leq T,$ for some constant $0<V(a,R) < \infty.$

\end{thm}

\begin{proof}
Let $R>0$ be given, and  $\eta :M \to [0,1] \subseteq \R$ be a smooth cut-off function  as in (iv) of Lemma \ref{balllemma}:
$\eta = 1$ on $B_R(x_0),$ $\eta = 0$ on $M \backslash (B_{C R}(x_0)),$  
$|\gradh^2 \eta| + |\gradh \eta|^2/\eta \leq \frac{C}{  R^2} $  on $M$ (here $n=4$), $|\gradh^i \eta|^2 \leq c_i(h)$ for all $i\in \N$. We mollify the metric $g_0$ everywhere locally, to obtain a metric $\hat g_{0,R},$ which is smooth,  
and then define  $g_{0,R}(\cdot) := \eta(\cdot)\hat g_{0,R}(\cdot) + (1-\eta(\cdot))h(\cdot)$.
We choose the mollification fine enough to guarantee that  
$ g_{0,R}(\cdot)\to g_0$ in $W^{2,2}(B_r(0))$ for all $r>0$ fixed as $R\to\infty$ and so that
\eqref{aaa}, and \eqref{bbb}  still hold for $g_{0,R}$  up to a factor  $\frac{10}{9},$ for all $R>0$ sufficiently large. 
That is we have
\begin{align} 
&   \frac{9}{10  a}  h \leq g_{0,R} \leq  \frac{10}{9} a h \tag{$\ti {\rm a}$} \label{aaaz}\\
 & \int_{B_2(x)} ( |\gradh g_{0,R}|^2 +   |\gradh^2 g_{0,R}|^2)dh \leq \frac{10}{9} \ep
\ \ \mbox{ \rm for all } \ \ x \in M  \tag{$\ti {\rm b}$} \label{bbbt}
 \end{align}
 
Furthermore $g_{0,R} = h$ outside of $B_{C R}(x_0)$ and so  
$\sup_M  |\grad^j  g_{0,R}|^2 < \infty $ for all $j\in \N$.

 Theorem \ref{main1}, with $a$ replaced by $\frac {10} 9 a$ and $\ep$ by $\frac {10}{9}$, and  Remark \ref{usefulrema}, guarantee the existence of a solution  
$(g_R(t))_{t \in [0,T]}$ with $T=T(a,\ep)>0$ to \eqref{Meq} satisfying  \eqref{aaa_t},\eqref{bbb_t} and \eqref{ccc_t}  for all $x \in M$, for all $t \in [0,T],$  with $g_{R}(0) = g_{0,R}:$ Note that the constants $c_j(\frac {10} 9 a,h)$ of \eqref{ccc_t} do not depend on $R.$ 
Hence there exists a limit solution $(M,g(t)_{t\in (0,T]}$ (in the $C^{\infty}_{loc}$ sense on $M \times (0,T)$), by taking the sequence of radii $R(i)=i  \to \infty$, which satisfies  \eqref{aaa_t},\eqref{bbb_t} and \eqref{ccc_t}  for all $x \in M$, for all $t \in (0,T].$ 
Theorem \ref{smallenergy} applied to each $g_{R(i)}$  implies \eqref{fff_t}, and Theorem   \ref{W22convergence} implies \eqref{ddd_t}, for the limit  solution  $g(t)_{t\in (0,T]}.$
  Note without loss of generality, that  $\ep_1(2a)$ from Theorem \ref{main1} is less than $\frac{\de(400a)}{c(a,n)},$ where  $\de$ is the constant from Theorem  \ref{smallenergy}, and $c(a,n)$ is a large constant of our choice. 
  Hence, without  
  loss of  generality, the scale invariant condition  \begin{eqnarray}
  (\int_{B_1(x)} |\gradh g(\cdot,t)|^4 dh )^{\frac 1 2}  +  \int_{B_1(x)} |\gradh^2 g(\cdot,t)|^2 dh \leq  \frac{\de(400a)}{c(a,n)} 
  \label{inbetween} 
  \end{eqnarray}
   holds for all $x \in M$, for all $t \in (0,T],$    in view of (v) from Lemma \ref{balllemma}. 
For any $x \in M$ we claim that 
\begin{align*} 
& \sup_{B_1(x)} (t  |\gradh g(\cdot,t)|^2  )   \to 0 \mbox { as } t \downto 0.
\end{align*} 
Assume that this is not the case for some $x \in M$. Then we obtain a sequence of points $y_j \in B_1(x) \subseteq M$ and $0<t_j  \to 0$  for $j\in \N, j \to \infty$ and an  $r>0$   such that
$ t_j  |\gradh g(y_j,t_j)|^2  \geq r>0.$  Taking a subsequence we see $y_j \to y \in M$ and hence 
  $ \int_{B_{2 \sqrt{t_j}}(y_j)}  (|\gradh g_0|^4  +    |\gradh^2 g_0|^2 ) dh  \to 0$ as $j\to \infty.$ 
   Scale the solution  to time $1$:  that is we define $   g_j(\cdot,\ti t)  = \frac{1}{t_j} g(\cdot,t_j \ti t),$ 
  $ h_j(\cdot)  = \frac{1}{t_j} h(\cdot)$ and $g_{j,0} = \frac{1}{t_j} g_0.$   Now we have 
  $ |{}^{h_j} \nabla g_j(y_j,1)|^2  \geq r>0$ and $ \int_{B_{2}(y_j)}  (|\gradh g_{j,0}|^4  +    |\gradh^2 g_{j,0}|^2 ) dh  \to 0$ as $j\to \infty,$  
     and $ \frac{1}{400 a} h_j(\cdot ) \leq  g_j(\cdot,t) \leq 400 a h_j(\cdot)$ for all $t$ where the solution is defined.
Theorem  \ref{smallenergy}, and the fact that \eqref{inbetween} also holds for the scaled solutions,  now implies that  \\
    $  \int_{B_{1}(y_j)}    |{}^{h_j} \nabla g_j(\cdot,1)|^2 +   |{}^{h_j} \nabla^2 g_j(\cdot,1)|^2 dh_j   \to 0$ as $j\to \infty,$ and $ \frac{1}{400 a} h_j(\cdot ) \leq  g_j(\cdot,t) \leq 400 a h_j(\cdot)$ for all $t$ where the solution is defined. 
But  these  estimates  combined with  (\ref{ccc_t}) then  imply    that $|{}^{h_j} \nabla^k g_j(y_j,1)|^2  \to 0$ for all $k\in \N$  as $j\to \infty$, which is a contradiction : To see  this, one can write all quantities in geodesic coordinates with  respect to the metric $h_j$ centred at $y_j$.
The  estimate 
 $\sup_{B_1(x)} (t^j  |\gradh^jg(\cdot,t)|^2  )   \to 0 \mbox { as } t \downto 0  $  for the other $j \in \N$  follow  from  an
  almost identical argument.  That is \eqref{eee_t} also holds. 
 The uniqueness of the solution in the class of solutions satisfying  (\ref{aaa_t}),(\ref{bbb_t}), (\ref{ccc_t})  and (\ref{ddd_t}) follows immediately from Theorem \ref{uniqueness}. 

\end{proof}

\begin{rmk}\label{aftermain2}
In fact the constants $c_j(h,a)$ in \eqref{ccc_t} of Theorem \ref{main2}  can be replaced by $c_j(h,a,\ep)$  where $ c_j(h,a,\ep) \to 0$ as $\ep \to 0:$ 

Assume $g_{i}(t)_{t\in (0,T(\frac{1}{i},a)]}$ are solutions obtained in Theorem    \ref{main2}  with $\ep$ in \eqref{bbb_t} given by $\ep = \frac{1}{i}$ and assume $|\grad g_i(x_i,t_i) |^2 \geq \frac{\al}{t_i}>0$ for a $t_i \in (0,T(\frac{1}{i},a)]$ for some $\al>0$. 
From  Theorem \ref{balllemma} (v), we have $\int_{B_1(x)} |\grad g_i(\cdot,t)|^4 +  |\grad g_i^2(\cdot,t)|^2 \leq 2 \frac{1}{i} 
$ for all $t\in (0,T(\frac{1}{i},a)].$
Scaling the solutions by $\hat g_i(t):= \frac{1}{t_i} g_i(t t_i),$ we obtain smooth solution defined on 
$ (0,1]$ which satisfying    $|\gradhi^j \hat g_i(\cdot,1)|^2 \leq  c_j  $   for all $j\in \N$ in view of \eqref{ccc_t}, and 
 $|\gradhi \hat g_i(x_i,1) |^2  \geq \al$ and $\int_{B_1(x)} (|\gradhi \hat g_i(\cdot,1)|^4 +  |\gradhi \hat g_i^2(\cdot,1)|^2 )d_{h_i }(x)\leq 2 \frac{1}{i}$ for all $x \in M$, where we denote the scaled metric $ \frac{1}{t_i} h$   by $h_i$.
 But this means   $|\gradhi^j \hat g_i(\cdot,1)|^2 \to 0$  for all $j\in \N$ as $i\to 0$, as can be seen by writing all quantities in geodesic coordinates with respect to $h_i$ at $x_i$. This contradicts   $|\gradhi \hat g_i(x_i,1) |^2  \geq \al>0$ 
 for all $i \in \N.$
Hence $ c_1(a,h) = c_1(h,a,\ep) \to 0$ as $\ep \to 0.$ An almost identical argument shows that $c_j(h,a)$ in \eqref{ccc_t} of Theorem \ref{main2}  can be replaced by $c_j(h,a,\ep)$  where $ c_j(h,a,\ep) \to 0$ as $\ep \to 0.$ 

\end{rmk}

In the case that the energy is bounded uniformly, then a scaling argument leads to the setting of the previous theorem, and hence we may find a solution to the equation \eqref{Meq} for a short time, which satisfies the conclusions of the previous theorems, for any $\ep>0$, if we shorten the length of the time interval.

\begin{thm}\label{main3}
Let $(M,h)$  be four dimensional and satisfy  \eqref{hassumptions}, and $1<a<\infty,$
 and $g_0$  be a $W^{2,2} \cap L^{\infty}$  Riemannian metric, not necessarily smooth, which satisfies (\ref{aaa}) and 
\begin{eqnarray}
&&  \int_M (|\gradh g_0|^2 + |\gradh^2 g_0|^2 )dh< \infty.
\end{eqnarray}
Then for any $\ep>0$ there exists a constant $T=T(g_0,a,\ep)>0,$ a $C=C(g_0,a,\ep)$ and a smooth solution
$(g(t))_{t \in (0,T]}$ to \eqref{Meq},  such that 
 after scaling the solution and the background metric $h$ once,  
  \eqref{aaa_t},\eqref{bbb_t}, \eqref{ccc_t}, \eqref{ddd_t},  \eqref{eee_t} and \eqref{fff_t} hold, 
  and $(g(t))_{t \in (0,T]}$  is the unique solution in the class of solutions satisfying the conditions \eqref{aaa_t},\eqref{bbb_t}, \eqref{ccc_t}, \eqref{ddd_t}.
The constants $c_j(h,a)$ in \eqref{ccc_t}   can be replaced by $c_j(h,a,\ep)$  where $ c_j(h,a,\ep) \to 0$ as $\ep \to 0$ 
\end{thm}
\begin{rmk}
In this setting, we cannot expect 
\begin{eqnarray}
&&  \int_M (|\gradh g(t)|^2 + |\gradh^2 g(t)|^2)dh< \infty
\end{eqnarray}
for any $t>0$ as the 
example, Example \ref{W22example}, below shows.
\end{rmk}

\begin{proof}
As explained in the introduction : By scaling the initial data $g_0$ and the background metric $h$ once, we can guarantee that the new initial data and background metric, which we also call $g_0$  and $h$, 
 satisfy (\ref{aaa}) and (\ref{bbb}), where $\ep \leq \frac{\ep_1(2a)}{2} $  and $\ep_1$ is the constant from Theorem \ref{main1}, and that $h$ satisfies \eqref{hassumptionsscaled}.  
 Using Theorem \ref{main2}, we obtain a solution $g(t),$ $t>0$ $t \in [0,T]$ which satisfies (\ref{aaa_t}),(\ref{bbb_t}),  (\ref{ccc_t}) , (\ref{ddd_t}),
 (\ref{eee_t}) and (\ref{fff_t}). 
   The uniqueness of the solution in the class of solutions satisfying  (\ref{aaa_t}),(\ref{bbb_t}), (\ref{ccc_t})  and (\ref{ddd_t}) follows immediately from Theorem \ref{uniqueness}.
The fact that the $c_j(h,a)$ in \eqref{ccc_t}   can be replaced by $c_j(h,a,\ep)$  where $ c_j(h,a,\ep) \to 0$ as $\ep \to 0$ 
follows from Remark \ref{aftermain2}. 
\end{proof}

\begin{ex}\label{W22example}

Let $n=2$ and $g_0$ be a smooth metric on \\$C_1(0):= \{ (x_1,x_2)  \in \R^2 \ | \ 
\max(|x_1|, |x_2| ) <1 \}$ such that 
 $ (1-\ep)\de \leq g_0 \leq (1+\ep)\de,$ $0< \ep<<1$ and $g_0 = \de $ on 
 $C_1(0) \backslash C_{\frac 1 2}(0)$, and so that the curvature of $g_0 $ is a constant  $\si_1>0$ on $C_{\si_2}(0),$ for two small constants $\si_1,\si_2>0$,  where $\de$ is the standard metric on $\R^2$. We extend this metric to all of $\R^2$ through symmetry, $g_0(x) = g_0(x+p)$ for all $p \in \Z^2 = \{ (z_1,z_2) \ | \ z_1, z_2 \in \Z\}$, and in doing so obtain a smooth Riemannian metric $g_0$ on $\R^2$ satisfying 
 $ (1-\ep)\de \leq g_0 \leq (1+\ep)\de,$ with $\psi_p^*g_0 =g_0$ for any $p\in \Z^2$, where
 $\psi_p(x) = x+p$.
 Let $T^2_{i}$ refers to the standard $2$-torus whose circles have radius $i \in \N$.
 $T^2_{i} := \R^2 / \Gamma(i),$ where $\Gamma(i):= \{T_y: \R^2 \to \R^2 \ |  \  y \in \Z^2, $ where $T_y(x) = x + iy,$ for all $ x \in \R^2 \}.$ That is
 $T^2_{i}  = \{ [x] \ | \ x \in \R^2\}$ where $[x] =[z]$ if and only if $T(x)=z$ for some $T \in \Gamma(i)$. 
We give $T^2_{i}$ the unique metric $g_0(i)$ such that  $\pi^*(g_0(i)) = g_0$ 
where $\pi: \R^2 \to T^2_{i}$ is the standard projection, $\pi(x) = [x]$.    

 From the work of Shi, \cite{Shi}, there exists a unique smooth 
 solution $(T^2_{i}, g(i)(t))_{t\in [0, T]}$ to \eqref{Meq} with $h=g_0(i),$ $g(i)(0) = g_0(i),$   
 $\sup_{t \in [0,T]} |\gradh^j g(i)(t)|^2 < c_j(g_0)$  for all $j\in \N,$ and all $t\in [0, T]$ and  $ (1-2\ep)g_0(i) \leq g(i)(t) \leq (1+2\ep)g_0(i)$ for all $t\in [0, T].$
 Defining $\phi_p : T^2_i \to T^2_i$ by 
 $\phi_{p}([x]) = [p+x],$  where $p \in \Z^2$, we see that
 $\phi_p^{*}(g_0(i)) = g_0(i)$ by construction, and hence it is an isometry with respect to $g_0(i)$. Setting $\ti g(i)(t):= (\phi_p)^*(g(i)(t)),$ we see that it is also a solution to \eqref{Meq}, with $\ti g(i)(0) = g_0(i)$ and $h=g_0(i)$. Uniqueness of such solutions, which can be seen by applying the maximum principle to the function $f(\cdot,t):= |g_1(\cdot,t)-g_2(\cdot,t)|^2,$ shows us that  $(\phi_p)^*g(i)(t) = g(i)(t)$ for all $t \in [0,T]$.

Taking a subsequence and then a limit $i\to \infty,$ we obtain a solution $g(t)$ to \eqref{Meq} with $g(0) = g_0,$  $\sup_{t \in [0,T]} |\gradh^j g(t)|^2 < c_j(g_0)< \infty$  for all $j\in \N,$ and all $t\in [0, T],$
and such that $(\psi_p)^{*}g(t) = g(t)$, $\psi_p^* h = h$ where $\psi_p(x) = x+p$ and $p \in \Z^2.$
 Furthermore  $|\gradh^j g_0| =0$ for all $j\in \N,$ since $g_0 =h$.
 
Using a Taylor expansion in time for $g(t)$, and the fact that $g(t)$ is a smooth solution to  \eqref{Meq} with $g(0)=h$, we see that $g(x,t)  = h(x) + \partt g(x,0) \cdot t + (\partt \partt)  g(x,s) \cdot t^2$ 
for some $s \in (0,t)$. Notice that $\partt g(x,0) = 0$ for $x \in C_{\frac 4 5}(0) \backslash C_{\frac 3 4}(0)$
because there, $g(x,0) = h(x) = \de$ and $\de$ has zero curvature. 
Hence  $g(x,t) - \de = O(t^2)$ for small $t$ for $x \in 
C_{\frac 4 5}(0) \backslash C_{\frac 3 4}(0)$ and $g_{ii}(y,t) -h_{ii}(y) = -    2\si_{ii}  t  + O(t^2)$  for $y=0,$ for where $\si_{ii}=    c(n)\si_1>0.$ 
If $\grad g(t) = 0$ holds for all $t \in [0,P)$ then, for all $t \in (0,P)$ with $P>0$ small enough, we would have
$g(t) = h $ on $C_{\frac 4 5}(0) \backslash C_{\frac 3 4}(0)$
and $g(y,t) -h(y) \neq 0$ at $y=0$.
But taking any two vectors $v,w$ at a point $p$ in $C_{\frac 4 5}(0) \backslash C_{\frac 3 4}(0)$
we have $g(p,t)(v,w) = h(p)(v,w)$.
Parallel transporting the  vectors along a geodesic (with respect  to $h$) from $p$ to $0$,
we would have $\partr g(\ga(r),t)(v(r),w(r)) = 0 = \partr h(\ga(r))(v(r),w(r))$ and hence
$g(0,t)(v(r),w(r)) = h(0)(v(r),w(r))$. Since the vectors were arbitrary, we see $g(0,t) =h(0)$ for all $t \in [0,P)$ : a contradiction.
By smoothness of the solution,  there exists an 
$s_t>0$ such that $\int_{C_1(0)} (|\gradh g(t)|^2 + |\gradh^2 g(t)|^2 )dh\geq s_t>0$ for some $t>0$. 
Using the isometry $\psi_p,$ we see that this means 
  $\int_{C_1(p)} (|\gradh g(t)|^2 + |\gradh^2 g(t)|^2)dh \geq s_t>0,$  for all $p\in \Z^2$, 
hence $\int_M (|\gradh g(t)|^2 + |\gradh^2 g(t)|^2)dh = \infty$ for this  $t$.  
To obtain an example in $\R^n$  with $n \in \N, n > 2$   we simply take  $\hat g_0  = \hat h = g_0 \oplus \de_{\R^k}$ on $\R^{k+2}=\R^n$  where $\de_{\R^k}$  is the standard metric on $\R^k$.

 \end{ex}

\section{Ricci flow estimates}\label{ricciflowestimates} 
The results of the previous sections are in the setting of the Ricci DeTurck flow.
As mentioned at the beginning of the paper, in certain settings, for example the closed smooth setting, there is a Ricci flow related solution, which can be written as 
$\ell(t) := (\phi_t)^*{g}(t)$ where $\phi_t:M \to M$, $t\in [0,T]$, or possibly only $t\in (0,T]$, is a smooth family of diffeomorphisms, and $g(t)_{t\in [0,T)}$ is a solution to Ricci DeTurck flow. We say in this case, that the Ricci flow solution $\ell(t)$ {\it comes from} the  Ricci DeTurck  solution $g(t)$, and we call  $\ell(t)$  a {\it Ricci Flow related solution} to $g(t)$. 
In the next section we construct and analyse the behaviour of  solutions $\ell(t)$ {\it coming  from} a Ricci DeTurck  solution $g(t)$ constructed in the previous sections. 
In order to do this, we  require various estimates for solutions to local Ricci flows. The statements and proofs thereof are contained in this section.  
 The results of this section  are written in a local setting assuming various geometric bounds, which we know will hold {\it if} 
 the local Ricci flow solution comes from a  Ricci DeTurck flow solution   constructed in  the previous  sections of this paper.  However, it is not  necessary to assume that the local Ricci flow solutions we consider are constructed in this manner.

\begin{thm}\label{Lpboundsriccithm}
For all $p\in [2,\infty)$ there exists a $ \al_0=  \al_0(p,n)>0$ such that the following holds. 
Let $\Omega$ be a smooth $n$-dimensional manifold 
 and $(\Omega^n,\ell(t))_{t\in (0,T]}$      be a smooth solution to Ricci flow   satisfying 
\begin{align}
&\int_{\Omega} |\Rc(\ell(t))| d\ell(t)  \leq \ep \label{LpRicciconditions}\\
&|\Rc (\ell(t))| \leq \frac{\ep}{t} \ \mbox{ on } \Omega   \nonumber
\end{align}
for all $t\in (0,T],$ where $\ep \leq \al_0$.
 Then there exist  $\de(n,p,\ep),c(n,p)>0$, with the property that $\de(n,p,\ep) \to 0$ as $\ep \to 0$, such that  
\begin{eqnarray}
&&   \int_{\Omega}  |\ell(t)-\ell(s)|^p_{\ell(t)} d\ell(t) \leq  \de(n,p,\ep)  |t-s| \label{Lpest1}\\
&& \int_{\Omega}  |(\ell(t))^{-1}-(\ell(s))^{-1}|^p_{\ell(t)} d\ell(t) \leq  \de(n,p,\ep)|t-s| \label{Lpest2}  
\end{eqnarray} for all $t,s\in (0,T]$ with $s<t$.
Furthermore,  
\begin{eqnarray}
&&   \int_{\Omega}  |\ell(s)|^p_{\ell(t)} d\ell(t) \leq c(n,p)(\Vol(\Omega,\ell(t))   + |t-s| )\label{Lpest21}\\
&&   \int_{\Omega}  |\ell^{-1}(s)|^p_{\ell(t)} d\ell(t) \leq c(n,p)(\Vol(\Omega,\ell(t))   + |t-s| )\label{Lpest22}\\
&&   \int_{\Omega}  |\ell(r)-\ell(s)|^p_{\ell(t)} d\ell(t) \leq c(n,p)( \Vol(\Omega,\ell(t)) +t )^{\frac 3 4}|r-s|^{\frac 1 4} \label{Lpest23}\\
&& \int_{\Omega}  |(\ell(r))^{-1}-(\ell(s))^{-1}|^p_{\ell(t)} d\ell(t) \leq c(n,p)( \Vol(\Omega,\ell(t)) +t )^{\frac 3 4}|r-s |^{\frac 1 4} \label{Lpest24} 
\end{eqnarray} for all   $r,s,t \in (0,T]$ with $r,s<t$.
\end{thm}
\rmk{This Theorem  is true for {\it any  smooth Ricci flow} satisfying \eqref{LpRicciconditions}: Completeness, compactness, volume bounds,  Sobolev inequalities, are not   assumed}.

\begin{proof}

We write $v := \ell(s)$, and $\ell$ for $\ell(t)$,
 for $t\in (s,T]$,   and $\Rc$ for $\Rc(\ell(t))$
and we calculate for \\
$h(t):= \int_{\Omega}  |\ell(t)-\ell(s)|^p_{\ell(t)} d\ell(t) ,$
\begin{align*}
 \partt h(t)  & =  \partt \int_{\Omega} |\ell(t)- v|^p_{\ell(t)} d\ell(t) \cr
 & =  \partt  \int_{\Omega}  (\ell^{ik}\ell^{jm}(\ell_{ij}    -v_{ij}  ) ( \ell_{k m}    - v_{km} ))^{\frac p 2}  d\ell  \cr
  & =    \int_{\Omega}   -R(\ell) |\ell-v|_\ell^p    +    \frac{p}{2} |\ell- v|^{p-2}_{\ell}  \Big( 2\ell^{is}\ell^{kz} \Rc(\ell)_{sz} \ell^{jm}(\ell_{ij}  -  v_{ij}) ( \ell_{k m}    -v_{km} ) \cr 
  &  \ \  \ \ \ \ \ \ \ \ 
   +2 \ell^{ik} \ell^{js}\ell^{mz} \Rc(\ell)_{sz}  (\ell_{ij}     -v_{ij}  ) ( \ell_{k m}   -v_{k m}  )
    \cr  
   &  \ \  \ \ \ \ \ \ \ \  -  2  \ell^{ik}\ell^{jm} \Rc(\ell)_{ij}    ( \ell_{k m}    - v_{km} )  -  2  \ell^{ik}\ell^{jm} ( (\ell_{ij}     -v_{ij}  )       \Rc_{k m}(\ell)  \Big)   \cr
 & \leq   c(n)p  \int_{\Omega}(|\Rc |_{\ell} |\ell- v|^{p}_{\ell}  
   + |\Rc|_{\ell} |\ell- v|^{p-1}_{\ell}) d\ell\cr
  & \leq  \frac{c(n)p\ep }{t}  h(t) +  c(n)p  \int_{\Omega} |\Rc|_{\ell}^{\frac 1 p} (|\Rc|_{\ell}^{\frac{p-1}{p}} |\ell- v|^{p-1}_{\ell}) d\ell\cr 
 & \leq \frac{c(n)p\ep }{t} h(t)  +  c(n)p \int_{\Omega}  (|\Rc|_\ell + |\Rc|_\ell |\ell- v|^{p}_{\ell})d\ell \cr 
 & \leq \frac{c(n)p\ep }{t} h(t)  +  c(n)p \ep 
 \end{align*}
for some $c(n)\geq 1$ depending only on $n$. Hence the function $f(t)= h(s+t)$ also satisfies
\begin{align}
  \partt f(t) &  \leq  \frac{\be_0}{s+t} f(t)   + {\be_0}\cr
  & \leq  \frac{\be_0}{t} f(t)   + {\be_0}\label{ODEin}
\end{align} 
for all $t \in (0,T-s]$,
where $\be_0:=  c(n) p \ep.$ The assumptions on $\ep$ guarantee that $\be_0=  c(n) p \ep  \leq 
 \frac{1}{2},$ and hence,  
 using the ODE Lemma  \ref{ODECor}, we see that
 $f(t) \leq 2\be_0 t$  for all $t\in (0,T-s]$. In particular, we obtain
 $h(t) =f(t-s) \leq 2\be_0 (t-s) =  c(n) p \ep (t-s) $ for all $t\in (s,T]$, that is 
  \eqref{Lpest1} holds.

The estimate \eqref{Lpest2} is proved in an alomost identical way.
We calculate for  $y(t) := \int_{\Omega}  |(\ell^{-1}-{v}^{-1})|^p_{\ell(t)} d\ell(t) ,$
\begin{align*}
\partt y(t)  
 & =  \partt \int_{\Omega} | \ell^{-1}- {v}^{-1}|^p_{\ell} d\ell \cr
  & =     \int_{\Omega}   -R(\ell)| \ell^{-1}- {v}^{-1}|^p_{\ell} +  \frac{p}{2} |\ell^{-1}- {v}^{-1}|^{p-2}_{\ell}  \Big( -2\Rc(\ell)_{ik}\ell_{jm}(\ell^{ij}    -v^{ij}  ) ( \ell^{k m}    - v^{km} )   \cr 
  &  \ \  \ \ \ \ \ \ \ \  -2\ell_{ik}\Rc(\ell)_{jm}(\ell^{ij}    -v^{ij}  ) ( \ell^{k m}    - v^{km} )    \cr
 &    \ \  \ \ \ \ \ \ \ \  +  2  \ell_{ik}\ell_{jm} \ell^{ip}\ell^{jq} \Rc(\ell)_{pq}
     ( \ell^{k m}    -  v^{km} )
     \cr
   & \ \  \ \ \ \ \ \ \ \  +  2  \ell_{ik}\ell_{jm} (\ell^{ij}    -v^{ij}  )\Rc(\ell)_{pq}   \ell^{kp} \ell^{mq}   \Big) d\ell \cr
 & \leq   c(n)p(  \int_{\Omega}|\Rc(\ell)|_{\ell} |\ell^{-1}- v^{-1}|^{p}_{\ell}    + |\Rc(\ell)|_{\ell} |\ell^{-1}- v^{-1}|_{\ell}^{p-1}d\ell)  \cr
 &  \leq  \frac{c(n)p\al_0 }{t}  \int_{\Omega} |\ell^{-1}- v^{-1}|^{p}_{\ell} d\ell +  
 c(n)p  \int_{\Omega} |\Rc(\ell)|_{\ell} |\ell^{-1}- v^{-1}|^{p-1}_{\ell} d\ell \cr 
 & \leq \frac{2c(n)p\al_0 }{t} \ell(t)  +  c(n)p \al_0\cr
 & \leq  \frac{\be_0 y(t)}{t} + \be_0. 
 \end{align*}
The inequality \eqref{Lpest2}  now follows as before from Lemma \ref{ODECor}. 
The inequalities \eqref{Lpest21}, \eqref{Lpest22}, \eqref{Lpest23} and \eqref{Lpest24} follow   from the H\"older
   and   triangle inequalities, as we now show.  First we show \eqref{Lpest21}:
\begin{eqnarray}
\int_{\Omega} |\ell(s)|^p_{\ell(t)}d\ell(t) 
&& = \int_{\Omega} |\ell(s) -\ell(t) +\ell(t)|^p_{\ell(t) }d\ell(t)\cr
&& \leq c(n,p) \int_{\Omega} |\ell(s) -\ell(t)|^p_{\ell(t) }d\ell(t) + c(n,p)\int_{\Omega}|\ell(t)|_{\ell(t)}^pd\ell(t)\cr
&& \leq  c(n,p)|t-s|+ c(n,p)\Vol(\Omega,\ell(t)).\label{inbetweenin1}
\end{eqnarray}
Similarly,
\begin{eqnarray}
\int_{\Omega} |\ell^{-1}(s)|^p_{\ell(t)}d\ell(t) 
&& = \int_{\Omega} |\ell^{-1}(s) -\ell^{-1}(t) +\ell^{-1}(t)|^p_{\ell(t) }d\ell(t)\cr
&& \leq  c(n,p)\int_{\Omega} |\ell^{-1}(s) -\ell^{-1}(t)|^p_{\ell(t) }d\ell(t) +c(n,p) \int_{\Omega}|\ell^{-1}(t)|_{\ell(t)}^p d\ell(t)
\cr
&& \leq  c(n,p)|t-s|+ c(n,p)\Vol(\Omega,\ell(t). \label{inbetweenin2}
\end{eqnarray}
Thus we see that \eqref{Lpest21}, \eqref{Lpest22} hold.
To show \eqref{Lpest23} and \eqref{Lpest24} hold, we will use the estimates of Appendix {D}, which 
show that certain  general inequalities hold which relate the $L^p$  norms of a tensor taken  with respect to different metrics.

For any two tensor $T= T_{ij}$ we have, using Corollary \ref{lpmetriccor},   
\begin{eqnarray}
  \int_{\Omega} |T|^{p}_{\ell(t) } d\ell(t)  && \leq   c(n,p)  (\int_{\Omega} |\ell(s)|^{2p}_{\ell(t)} d\ell(t))^{\frac 1 2} 
(\int_{\Omega} |T|^{4p}_{\ell(s)}  d\ell(s))^{\frac 1 4} (\int_{\Omega} |\ell(t)|_{\ell(s)}^{\frac n 2}  d\ell(t) )^{\frac 1 4}   
\cr
&& = c(n,p)(\int_{\Omega} |\ell(s)|^{2p}_{\ell(t)} d\ell(t))^{\frac 1 2} 
(\int_{\Omega} |T|^{4p}_{\ell(s)}  d\ell(s))^{\frac 1 4} (\int_{\Omega} |\ell^{-1}(s)|_{\ell(t)}^{\frac n 2}  d\ell(t) )^{\frac 1 4}    \label{Testimate1}  
\end{eqnarray} 

For $T = (\ell(r) -\ell(s))$  in \eqref{Testimate1} we obtain
\begin{eqnarray*}
&&  \int_{\Omega} |\ell(r) -\ell(s)|_{\ell(t)}^p d\ell(t) \cr
  && \leq  c(n,p)(\int_{\Omega}   |\ell(s)|^{2p}_{\ell(t)} d\ell(t) )^{\frac 1 2}  \cdot  (\int_{\Omega}  |\ell(s) -\ell(r)|^{4p}_{\ell(s)} d\ell(s) )^{\frac 1 4} \cdot (\int_{\Omega}  |\ell^{-1}(s)|^{\frac n 2}_{\ell(t)}   d\ell(t) ) ^{\frac 1 4} \cr
&& \leq c(n,p)(\Vol(\Omega,\ell(t)) +t)^{\frac 3 4} \cdot  |r-s|^{\frac 1 4} 
\end{eqnarray*}
in view of the estimates  \eqref{Lpest1},  \eqref{Lpest21} and \eqref{Lpest22}.
For any two tensor $N= N^{ij}$ we have, using Corollary \ref{lpmetriccor}, 
\begin{eqnarray}
  \int_{\Omega} |N|^{p}_{\ell(t) } d\ell(t)  && \leq  c(n,p)  (\int_{\Omega} |\ell(t)|^{2p}_{\ell(s)} d\ell(t))^{\frac 1 2} 
(\int_{\Omega} |N|^{4p}_{\ell(s)}  d\ell(s))^{\frac 1 4} (\int_{\Omega} |\ell(t)|_{\ell(s)}^{\frac n 2}  d\ell(t) )^{\frac 1 4}   \cr
&& =   c(n,p)(\int_{\Omega} |\ell^{-1}(s)|^{2p}_{\ell(t)} d\ell(t))^{\frac 1 2} 
(\int_{\Omega} |N|^{4p}_{\ell(s)}  d\ell(s))^{\frac 1 4} (\int_{\Omega} |\ell^{-1}(s)|_{\ell(t)}^{\frac n 2}  d\ell(t) )^{\frac 1 4} 
 \label{Testimate2}  
\end{eqnarray}

For $N = (\ell^{-1}(r) -\ell^{-1}(s))$ in \eqref{Testimate2} we obtain
\begin{align*}
   \int_{\Omega} |\ell^{-1}(r) -\ell^{-1}(s)|_{\ell(t)}^p d\ell(t)    
 \leq&  c(n,p)(\int_{\Omega} |\ell^{-1}(s)|^{2p}_{\ell(t)} d\ell(t))^{\frac 1 2}   \cdot  (\int_{\Omega}  |\ell^{-1}(s) -\ell^{-1}(r)|^{4p}_{\ell    (s)} d\ell(s) )^{\frac 1 4} \cr
& \cdot  (\int_{\Omega} |\ell^{-1}(s)|_{\ell(t)}^{\frac n 2}  d\ell(t) )^{\frac 1 4}  \cr
 \leq& c(n,p)(\Vol(\Omega,\ell(t)) +t)^{\frac 3 4} \cdot  |r-s|^{\frac 1 4} 
\end{align*}
in view of the estimates  \eqref{Lpest2},  \eqref{Lpest21} and \eqref{Lpest22}.
That is, the inequalities \eqref{Lpest23} and \eqref{Lpest24} hold.
\end{proof}

The previous Theorems shows that for $p\in [2,\infty)$ and $n\in \N,$ a solution $(\Omega,\ell(t))_{t\in (0,T]}$ which satisfies  the conditions of the Lemma,  
that is $\int_{\Omega} |\Rc(\ell(t))| d\ell(t)  \leq \ep$ and 
$|\Rc (\ell(t))| \leq \frac{\ep}{t} \ \mbox{ on } \Omega $ 
for all $t\in (0,T],$ where $\ep \leq \al_0 = \al_0(n,p),$
must have a  uniquely well defined  starting value $\ell_0 \in L^p(\Omega)$ which is a symmetric two tensor, whose inverse exists almost everywhere: 

\begin{cor}\label{Lpcor}
For all $p\in [2,\infty)$ and $n\in \N$ there exists an $\al_0(n,p) >0$ such that the following holds. 
Let $\Omega$ be a smooth $n$-dimensional manifold 
 and $(\Omega^n,\ell(t))_{t\in (0,T]}$      be a smooth solution to Ricci flow   satisfying 
\begin{align*}
&\int_{\Omega} |\Rc(\ell(t))| d\ell(t)  \leq \ep \cr
&|\Rc (\ell(t))| \leq \frac{\ep}{t} \ \mbox{ on } \Omega  
\end{align*}
for all $t\in (0,T],$ where $\ep \leq \al_0$.
Then there exists a unique two tensor $\ell_0 \in L^p$ such that
$ \ell(s) \to \ell_0$ in $L^p(\Omega)$ as $s \downto 0$ where $\ell_0$ is positive definite (except for a measure zero set), and $ \ell^{-1}(s) \to (\ell_0)^{-1}$ in $L^p(\Omega)$ as $s \downto 0$.
\end{cor}
\begin{proof}
From \eqref{Lpest23} and \eqref{Lpest24}  we see that $\ell(s)_{s \in (0,T]}$ and $\ell^{-1}(s)_{s \in (0,T]}$ are  Cauchy w.r.t to $s$ in $L^{2p}(\Omega,\ell(t))$  for fixed $t>0$, if $\ep \leq \al_0(2p,n)$ is chosen small enough, where
$\ell_0$ is as in the statement of Theorem \ref{Lpboundsriccithm},  so there exists $\ell_0 , r_0 \in L^{2p}(\Omega,\ell(t))$ such that $\ell(s) \to \ell_0$ and $\ell^{-1}(s) \to r_0$ as $s \downto 0$ in the $L^{2p}$ norm.
Furthermore
$\de^i_{j}  = \ell^{ik}(s)\ell_{jk}(s)$ and so we have, for $\Vert\cdot \Vert_{L^q} = \Vert\cdot \Vert_{L^q(\Omega,\ell(t))}$ for some fixed $t>0$, 
\begin{eqnarray*}
 &&\Vert {\de^i}_j  -  (\ell_0)_{jk}r_0^{ik} \Vert_{L^p}    \cr
  && = \Vert \ell^{ik}(s)\ell_{jk}(s)  -  (\ell_0)_{jk}r_0^{ik} \Vert_{L^p}   \cr 
&& = \Vert (\ell^{ik}(s) -  r_0^{ik}) \ell_{jk}(s)  -r_0^{ik} ( (\ell_0)_{jk} - \ell_{jk}(s)) \Vert_{L^p}   \cr
 && \leq  \Vert (\ell^{ik}(s) - r_0^{ik})\ell_{jk}(s)\Vert_{L^p}  
+ \Vert r_0^{ik} ( (\ell_0)_{jk} - \ell_{jk}(s)) \Vert_{L^p}   \cr
&& \leq  \Vert \ell^{ik}(s) - r_0^{ik} \Vert_{L^{2p}} 
\Vert \ell_{jk}(s)\Vert_{L^{2p}}  + 
\Vert r_0^{ik} \Vert_{L^{2p}} \Vert    (\ell_0)_{jk} - \ell_{jk}(s)  \Vert_{L^{2p}}   \cr
&& \to 0
\end{eqnarray*}
as  $s \downto 0$ in view of \eqref{Lpest21}, \eqref{Lpest22}, \eqref{Lpest23} and
\eqref{Lpest24}. Hence $r_0= (\ell_0)^{-1}$ almost everywhere. 
At points $x$  in the  set of measure zero, where $\ell(0)(x)$ is degenerate, we replace $\ell(0)(x)$  by $\ell(t)(x)$ for a fixed $t>0$.  The convergence result still holds,  but now $\ell(0)$ is positive definite everywhere.
\end{proof}

\begin{thm}\label{W22pboundsriccithm}
For any $A  >0$  there exist  $\al_1,\be,S>0$ such that the following holds. 
Let $(M^4,\ell(t))_{t\in (0,T]}$    be a smooth four dimensional solution to Ricci flow, with $B_{\ell(t)}(x_0,10) \Subset  M$  $T \leq 1$,  satisfying a uniform  Sobolev inequality  for all $t \in (0,T]$ : 
\begin{eqnarray*}
\Big( \int_{B_{\ell(t)}(x_0,2) } |f|^4 d\ell(t) \Big)^{\frac 1 2} \leq 
A  \Big( \int_{B_{\ell(t)}(x_0,2) } |\nabla f|^2 d\ell(t)  + \int_{B_{\ell(t)}(x_0,2) } |f|^2 d\ell(t)  \Big)
\end{eqnarray*}
for any $f$ compactly contained in $ B_{\ell(t)}(x_0,2)$ for any $t\in (0,T]$,  where $\gradg$ refers to the covariant derivative with respect to $\ell(t)$.
We further assume 
\begin{align}
&\int_{B_{\ell(t)}(x_0,2) } |\Rm(\ell(t))|^2 d\ell(t)  +  \int_{B_{\ell(t)}(x_0,2) } |\Rm(\ell(t))| d\ell(t)\leq \al_1 \label{contcond}\\
&|\Rc(\ell(t))| + |\nabla \Rc(\ell(t))|^{\frac 2 3}\leq \frac{\al_1}{t}   \ \mbox{ on } B_{\ell(s)}(x_0,2) \nonumber
\end{align}
for all $t,s\in (0,T].$ 
 Then  we have
\begin{eqnarray}
&& \int_{B_{\ell(t)}(x_0,\frac 1 2) }|\nabla ( \ell(t)-\ell(s))|^2_{\ell(t)}   d\ell(t) \leq   |t-s|^{\be}  \label{L2gradientest} 
\end{eqnarray} for all $t,s \in (0,T] \cap (0,S],$  with $s<t$,  
where $\gradg$ refers to the covariant derivative with respect to $\ell(t)$.
\end{thm}

\begin{proof}

We first  prove that the space time integral of $|\nabla \Rc|^2$ can be locally, uniformly bounded in the setting we
are considering. This estimate shall in turn be used to prove the   $L^2$   gradient estimate \eqref{L2gradientest}.
In the following $c$ refers to a universal constant independent of the solution.
Let $\eta: M \to \R^+$ be a Perelman cutoff-function, with $\eta(\cdot,t)  = e^{- t}$ on  
$B_{\ell(t)}(x_0,1)$ and $\eta(\cdot,t) = 0$ on $M - B_{\ell(t)}(x_0, \frac 5 4),$ $\partt \eta \leq 
\lap_{\ell(t)} \eta,$  $|\nabla \eta|^2 \leq c  \eta$ with  (see \cite{SimonTopping1} section $7$ for details of the construction).
Then, using the Sobolev inequality, H\"older's inequality, and the fact that 
$\int_{B_{\ell(t)}(x_0,2)} \eta^2(t) |\Rm|^2(t) d\ell(t)  \leq \al_1,$ we see that
\begin{eqnarray*}
&& \partt \int_M  \eta^2(\cdot,t) |\Rm|^2(\cdot,t)d\ell(t) \cr
&& =  \int_M  (  |\Rm|^2 (\partt \eta^2)  + \eta^2 \partt( |\Rm|^2 ) - \eta^2 \Sc |\Rm|^2 ) d\ell(t)\cr
&& \leq  \int_M (   |\Rm|^2 \lap (\eta^2)  + \eta^2 \lap (|\Rm|^2) - 2\eta^2 |\gradg \Rm|^2 +c\eta^2 |\Rm|^3)  d\ell(t)  \cr
&& = \int_M( \eta \langle  \Rm * \gradg \Rm , \gradg \eta\rangle_{\ell} - 2\eta^2 |\gradg \Rm|^2 + c \eta^2 |\Rm|^3  ) d\ell(t) \cr
&& \leq  \int_{B_{\ell(t)}(x_0,2)} ( - \eta^2 |\gradg \Rm|^2   + c  |\Rm|^2 + c \eta^2 |\Rm|^3 )d\ell(t) \cr 
&& =  \int_{B_{\ell(t)}(x_0,2)} (-\frac{1}{2}\eta^2 |\gradg \Rm|^2  - \frac{1}{2}[   |\gradg (\eta \Rm)|^2  - |\gradg \eta|^2 |\Rm|^2  - 2\eta \langle \Rm \gradg \eta , \gradg \Rm\rangle]   \cr
&&  \ \ \ \ \ \ \ \ \   \ \ \ \ \  + c  |\Rm|^2 + c \eta^2 |\Rm|^3 )d\ell(t) \cr 
&&  \leq  \al_1 c  + \int_{B_{\ell(t)}(x_0,2)} \big( - \frac 1 4 |\gradg (\eta \Rm)|^2    + c  (\eta |\Rm|)^2 |\Rm|  \big)d\ell(t)\cr
&&  \leq \al_1c   - \int_{B_{\ell(t)}(x_0,2)}  \frac 1 4 |\gradg (\eta \Rm)|^2 d\ell(t)    \cr
&& 
 \ \ \ \   \ \ \ \ \  +   c  \Big( \int_{B_{\ell(t)}(x_0,2)} |\eta \Rm|^4 d\ell(t)\Big)^{\frac 1 2} \Big( \int_{B_{\ell(t)}(x_0,2)} |\Rm|^2 d\ell(t)\Big)^{\frac 1 2}\cr
&& \leq  \al_1 c    +(c  A\sqrt{\al_1}-\frac14 ) \int_{B_{\ell(t)}(x_0,2)} |\gradg (\eta \Rm)|^2d\ell(t) \cr
&& \leq \al_1 c    - \frac18 \int_{B_{\ell(t)}(x_0,2)}  |\gradg (\eta \Rm)|^2d\ell(t)
\end{eqnarray*}
if $\al_1$ is small enough.
Hence, integrating form $s$ to $t$ we see that 
 \begin{align}\label{thisone}
 \int_s^t \int_{B_{\ell(r)}(x_0,1)} |\gradg   \Rm|^2(r) dx dr  
 & \leq c  \int_M  \eta^2 |\Rm|^2(\cdot,s) + \al_1ct \cr
 & \leq \al_1 c 
\end{align} 
with $ \al_1 c \leq 1$ without loss of generality.
We now turn to the proof of the integral gradient  estimate, \eqref{L2gradientest}.
This is similar to the $L^p$ estimate obtained for $\ell(t)$ in \eqref{Lpboundsriccithm}, but uses the space-time $L^2$ bound on the gradient of the Ricci curvature \eqref{thisone} that we just derived, 
instead of the  bound on the Riemannian curvature. 
In the following $|\cdot |$ refers to $|\cdot |_{\ell(t)}$, and $\Rc$ to $\Rc(\ell(t))$. 
Defining $\Omega:= B_{\ell(t_0)}(x_0,\frac 1 2),$ we see that
$\Omega \subseteq B_{\ell(r)}(x_0,1)$ for all $r \in (0,t_0)$ if $t_0 \leq 1$ in view of Corollary 3.3 of \cite{SimonTopping1} and the fact that the condition \eqref{contcond} holds.
Differentiating the function $f(t):= \int_{\Omega} |\gradg ( \ell(t)-\ell(s))|^2_{\ell(t)} d\ell(t),$
for $s<t   \leq t_0$, and using Young's, we get
\begin{align*} 
& \partt f(t)= \partt \int_{\Omega} |\gradg  (\ell(t)-\ell(s))|^2_{\ell(t)} d\ell(t)\cr
 & \leq  \int_{\Omega} ( c  |\Rc | |\gradg  (\ell(t) -\ell(s))|^2 + 2\langle \gradg  \Rc , \gradg  (\ell(t)-\ell(s))\rangle_{\ell(t)}) d\ell(t)\cr
 & \, + c\int_{\Omega} |\ell(t)-\ell(s)| |\gradg  \Rc ||\gradg  (\ell(t) -\ell(s))|  d\ell(t)\cr
  & \leq \frac{c \al_1}{ t} f(t) +  
 c \int_{\Omega} |\gradg  \Rc |  |\gradg  (\ell(t) -\ell(s))| d\ell(t) \cr
 &\,  + c\int_{\Omega} |\ell(t)-\ell(s)| |\gradg  \Rc ||\gradg  (\ell(t) -\ell(s))| d\ell(t)\cr 
   &\leq \frac{c \al_1}{ t} f(t) +   \frac{c\al_1}{ t-s} f(t)  +    
  \frac {c} {\al_1} (t-s)\int_{\Omega} |\gradg  \Rc |^2d\ell(t)  \cr
  & + c (\int_{\Omega} |\ell(t)-\ell(s)|^2 |\gradg  \Rc|^{\frac43} d\ell(t) )^{\frac 1 2} (\al_1t^{-1}f(t))^{\frac 1 2} \cr    
 & \leq \frac{c \al_1}{ t-s} f(t) +    
  \frac {c} {\al_1} (t-s)\int_{\Omega} |\gradg  \Rc |^2 d\ell(t) + c \int_{\Omega} |\ell(t)-\ell(s)|^2 |\gradg  \Rc|^{\frac43}\, d\ell(t)    \cr
   & \leq      \frac{c  \al_1}{ t-s} f(t) +   
   \frac {c} {\al_1} (t-s)\int_{\Omega} |\gradg  \Rc |^2d\ell(t) \cr
   & + c(\int_{\Omega} |\ell(t)-\ell(s)|^{6}d\ell(t))^{\frac {1}{3}} (\int   |\gradg  \Rc |^{2}d\ell(t))^{\frac{2}{3}} \cr  
   & \leq      \frac{c  \al_1}{t-s} f(t) +   
   \frac {c} {\al_1} (t-s)\int_{\Omega} |\gradg  \Rc |^2d\ell(t)  + c (t-s)^{\frac {1}{3}} (\int   |\gradg  \Rc|^{2}d\ell(t))^{\frac{2}{3}}\cr
   & \leq      \frac{c  \al_1}{ t-s} f(t) +  \frac{c}{\al_1} (t-s)^{\frac {1}{3}} (\int_{\Omega} |\gradg  \Rc|^2d\ell(t) +1) ,
\end{align*}
since $\int_{\Omega} |\ell(t)-\ell(s)|^6  \leq c(t-s)$ for sufficiently small $\al_1$, in view of Theorem \ref{Lpboundsriccithm},  and $(t-s)\leq (t-s)^{\frac 1 3},$ for $t<s\leq T \leq 1.$ 
Hence 
$$ \partt f(t) \leq \frac{\al_2}{t-s}f(t) + Z(t) $$ for $\al_2 :=c\al_1 $ 
and $Z(t) :=  c \al_2^{-1} (t-s)^{\frac13}(  \int_{\Omega} |\gradg  \Rc(t)|^{2}d\ell(t) + 1)$ for $t\leq S(n, \al_1) \leq 1.$  Hence, 
 $F(t) := f(t+s)$ for $t\in (0,S-s)$ then satisfies
 $$  \partt F(t) \leq \frac{\al_2}{t}F(t) + \ti Z(t) $$
 where $ \ti Z(r) = c \al_2^{-1} r^{\frac13} (   \int_{\Omega} |\gradg  \Rc(\ell(s+r))|^{2}d\ell(s+r) +1),$
 for $r \in (0,T-s)$. 
Thus, for $\al_2 \leq \frac16,$ we obtain 
 \begin{eqnarray*}
 && F(t) \leq t^{\al_2} \int_0^{t} \frac{\ti Z (r)}{r^{\al_2}} dr\cr
 &&  =  c\al_2^{-1}t^{\al_2}\int_0^{t}    r^{\frac13-\al_2} (   \int_{\Omega} |\gradg  \Rc(\ell(r+s))|^{2}d\ell(r+s) +1) dr \cr
 && \leq c\al_2^{-1} t^{\al_2}\int_0^{t} ( \int_{\Omega} |\gradg  \Rc(\ell(r+s))|^{2}d\ell(r+s) dr +1)\cr
 && =   c\al_2^{-1}t^{\al_2} (\int_s^{s+t} \int_{\Omega} |\gradg  \Rc(\ell(r))|^2 d\ell(r) dr +1)\cr
 && \leq  ct^{\al_2} (\al_2^{-1}\int_s^{s+t} \int_{B_{\ell(r)}(x_0,1)} |\gradg  \Rc(\ell(r))|^2 d\ell(r) dr +1)\cr
 && \leq   c t^{\al_2} 
 \end{eqnarray*}
 for $t\in (0,S-s)$ 
 in view of Lemma  \ref{ODELem}, and the fact that  the inequality \eqref{thisone} hold. That is 
 $f(t)  \leq c   (t-s)^{  \al_2}$ for $t \in (s,S)$.
 By choosing $\be = \frac{\al_2}{2},$ we obtain 
$f(t) = \int_{\Omega} |\gradg ( \ell(t)-\ell(s))|^2_{\ell(t)} d\ell(t)  \leq   (t-s)^{  \be}$ for $t \in (s,S)$ as required.
\end{proof}

\section{The Ricci flow related solution}\label{Ricciflowrelatedsolution}
In the  sections before Section \ref{ricciflowestimates}, we constructed a solution $g(t)_{t\in (0,T]}$ to the Ricci-DeTurck flow coming out of $W^{2,2}$ initial data $g_0$ on a four dimensional manifold, and we proved estimates for such solutions. 
In this section we construct a Ricci flow related solution $\ell(t)_{t\in (0,T]}$ coming from the Ricci DeTurck flow solution constructed in the sections preceding Section \ref{ricciflowestimates}. 
We show in the setting we are considering,  that the Ricci flow related solution
$\ell(t)$ converges back to some starting value $\ell_0$ locally in the $W^{1,2}$ norm, as $t\to 0$. 
 We shall see that the tensor  
$\ell_0$ is non-negative definite, up to a set of measure zero.
 Note that since, $g_0$ and $\ell_0$ are  only defined up to a measure zero, we can arbitrarily change distance induced by   $g_0$ respectively by  $\ell_0$ by changing $g_0$ respectively $\ell_0$  on a set of measure zero, if we try and use the usual definition of distance with respect to a Riemannian metric, as the following example shows.
\begin{ex}
Let $g,h$ be smooth Riemannian metrics on   $M= \B_1(0) \subseteq \R^n$, $r>0$ small so that $B_h(0,r) \Subset \B_1(0) $  and $ x \neq y, x,y \in B_h(0,r/4)$ and let  $\ga:[0,1] \to \B_{1}(0)$ be a smooth length minimising geodesic with respect to   $h$ from $x$ to $y$. We define a new metric $\ti g$, which is the same as $g$ except on the line $\ga$. On $\ga$ we define  $\ti g(\ga(s))  =  b^2  h(\ga(s))$ for all $s  \in [0,1],$ for some $b \in \R, b>0.$  $\ti g$ is still a well defined Riemannian metric, with  $\frac{1}{N}\de \leq  \ti g \leq N \de $  for some $N >0,$ $N \in \R,$  in view of the  smoothness of $g$ and the definition of $\ti g.$  
This ensures then that $\ti g_{ij}$ is  a Borel-measurable function, since $ g_{ij}$ is smooth and $g = \ti g$ almost everywhere.
Using the fact  that   $\frac{1}{N} \de \leq \ti g \leq N \de $ for some $N>0,$    we see then that $\ti g$ is in $L^p$ for any $p\in (0,\infty].$  We also have,  for any piecewise smooth
 $\si:[0,1] \to \B_1(0)$ with $\si(0) \neq \si(1)$  that 
  $\ti g_{ij} \of \si : I \to \R$ is Borel measurable, since both $\ti g_{ij}$ and $\si$ are, and 
  $\frac{1}{N^2} \de_{ij} \leq \ti g_{ij} \of \si   \leq N^2 \de_{ij} $, since this is true for   $\ti g_{ij}.$
  This means 
  $ \ell:[0,1] \to \R$, 
  $\ell(s):=  \sqrt{\ti g_{ij}(\si(s)) \partial_s \si^i(s) \partial_s \si^j(s)}$ is a well defined $L^1$ function and 
  $ L_{\ti g} (\si):= \int_0^1 \sqrt{\ti g_{ij}(\si(s)) \partial_s \si^i(s) \partial_s \si^j(s)}ds $ satisfies
  $    0< \frac 1 N L_{\de}(\si) \leq L_{\ti g} (\si) < N L_{\de}(\si) < \infty$ for all such $\si$.
If we define $d(\ti g)(p,q) :=   \inf_{\si \in B_{p,q} } L_{\ti g}(\si)$  for all $p,q \in M$ 
where $B_{p,q}$ refers to the space of continuous, piecewise smooth curves between $p$ and $q$ in $M=  \B_1(0),$ 
then we see that $(\B_1(0),  d(\ti g)) $ is a well defined metric space, that is $d(\ti g)$ is symmetric, satisfies the triangle inequality, and $d(\ti g)(p,q) \geq 0$ for all $p,q \in M$ with equality if and only if $p =q.$
Furthermore  $ d(\ti g)(x,y) \leq L_{\ti g}(\ga) (x,y) = L_{bh}(x,y) = bd(h)(x,y) < d(g)(x,y)$ if $b>0$ is chosen small enough,  and hence $d(\ti g)(x,y) <  d(g)(x,y)$ if $b>0$ is chosen small enough.

 That is, {\it if we use the usual definition for distance with respect to a Riemannian metric, distance can change   if we change the Riemannian metric on a set of measure zero}.
\end{ex}
In particular, this example shows that we cannot be sure that $d(g(t))(x,y) \to d(g_0)(x,y))$ everywhere, as $t\downto 0,$
in the case that we have a family of smooth metrics $g(t)$ which convergences in the $L^1$ sense (or another weak sense)  to a $g_0 \in L^1,$  
 if we  define $d(g_0)$ in the usual way,  $d(g_0)(x,y) = \inf_{\si \in B_{x,y}} L_{g_0}(\ga),$ where $B_{x,y}$ is the set of smooth curves going from $x$ to $y$ : if $g_0$ 
 is bounded from above and below by a smooth metric, we can change $g_0$ on a smooth curve between two given points $x$ and $y$ (as in the example above), so that $d(g(t))(x,y)$ doesn't converge to 
 $d(g_0)(x,y)$, but we still have 
 $g(t) \to g_0$ in $L^1$ as $t\downto 0.$  

Nevertheless,  we will see  for solutions $g(t)$   to the Ricci DeTurck flow constructed in the previous sections,   that 
$d(g(t))(x,y) $ does converge to some metric $d_0(x,y)$ as $t \downto 0,$  
where   $d_0$ is  defined in a similar fashion to the usual definition of $d(g_0)$, but it is necessary to restrict further the class of admissible curves $B_{x,y}$ between $x$ and $y$ to 
the class $C_{\ep,x,y}$ of so called {\it $\ep$-approximative Lebesgue curves} between $x$ and $y$, and then to take a limit inferior  as $\ep\to 0$ of the lengths.
\begin{defn}\label{lebesgueparlines}
i) For $p\in [1,\infty)$  we say $g$ is an  $L^p$ metric,  if the following holds. $g$ is a Riemannian metric, that is     
$g(x):T_xM \times T_x M \to \R$  is defined, symmetric, positive-definite, for all $x \in M$ and locally, writing 
$\ti g_{ij}(\ti x ) := g(x)(\frac{\partial}{\partial_i}(x),  \frac{\partial}{\partial_j}(x))$  for any smooth coordinates $\phi: U \to \phi(U)= \ti U \subseteq \R^n,$  
$ \ti g_{vv} : \ti U \to \R$ is in $L^p(\ti U)$ for all $v \in \R^n,$ where $v$ is any fixed length one (w.r.t to $\de$) vector in $\R^n$, and $\ti g_{vv}(x) := \ti g_{ij}(x)v^i v^j.$ \\
ii)  
For $x,y \in M$ we define  the set $C_{\ep,x,y}(g)$  of {\it $\ep$-approximative Lebesgue curves} with respect to $g$ from $x$ to $y$ in $M$  to be the set of paths $\ga: [a,b] \to M$ which can be written as the union of finitely many so  called {\it parameterised Lebesgue lines } $\ga_i: [a_{i-1},a_{i}] \to M,$ $i\in \{1, \ldots,N\},$ $a_0 =a$, $a_N =b,$  $\ga = \ga_1 \cup \ga_2 \cup  \ga_3 \cup \ldots \cup \ga_N,$ that is   $\ga(s) := \ga_i(s)$ if $s \in [a_{i-1},a_i]$, and \\
\begin{eqnarray*}
&& d_h(x,\ga_1(a)) + d_h(\ga_1(a_1),\ga_2(a_1))  +  d_h(\ga_2(a_2),\ga_3(a_2)) \\
&& + \ldots +  d_h(\ga_{N-1}(a_{N-1}), \ga_N(a_{N-1} )) +  d_h(\ga_N(b),y) \leq \ep,
\end{eqnarray*} and a
{\it parameterised Lebesgue line} is defined as follows:    
 A {\it parameterised Lebesgue line} for an $L^1$ metric $g$ on $M,$  is a smooth curve $\ell:[b,c] \to M$ with $|b-c|\leq \frac 1 4,$ such that there exist  smooth  coordinates 
  $\phi: \Omega  \to  \phi(\Omega) = \B_1(0)$ for some $p\in M,$ 
such that $ \ell([b,c]) \subseteq \Omega$ and the curve $\si:= \phi \of \ell :[b,c] \to \B_1(0)$  in these coordinates is a line  in the direction $e_1,$   with speed one,
$\si(s) = -\frac{(b+c)}{2}e_1 +  se_1$, $\si(b) = -k,$ $\si(c) = k,$ where $k = \frac{(c-b)}{2}$ and for $f(s) :=  \sqrt{\ti g_{11}( \si(s))} ,$ we have
 $f \in L^1([b,c])$ and 
 $\int_{b}^c f(s) ds = \lim_{\al \to 0} 
 \frac{\int_{T_{\al}(\si)} \sqrt{\ti g_{11}(x)} dx    }{ \omega_{n-1} \al^{n-1}},$ 
  where here  $T_{\al} (\si)$ refers to an $\al$ tubular neighbourhood of $\si$ with respect to $\de$, 
  $T_{\al}(\si) = \{  se_1 + \al v \ | \ |v|=1 , \langle v,e_1\rangle =0, s \in ( -\frac{(c-b)}{2},\frac{(c-b)}{2} )\}.$  Note that if $g$ is smooth then
$ \int_{b}^c f(s) ds = \lim_{\al \to 0} \frac{\int_{T_{\al}(\si)} \sqrt{\ti g_{11}(x)} dx}{  \omega_{n-1} \al^{n-1}  }$ always holds. Also, an {\it $\ep$-approximative Lebesgue curve} $\ga$ is the union of finitely smooth curves, but may itself be discontinuous, and hence non-smooth.
\end{defn}

In the setting that a Riemannian metric is $L^{\infty}$ (or weaker) there are various notions of distance  and convergence of distance which may be defined, and there are many papers in this  area investigating the properties thereof, their relation  to one another and to the underlying measures. For one overview, as well as independet results/proofs thereof,  we refer to the paper  \cite{CreutzSpultanis}.
Further notions and convergence results may be found in the papers \cite{BrianAllen},\cite{ConghanDong}, \cite{Aldana}, \cite{DavidSemmes}, 
 \cite{KorteKansanen},  \cite{LeeNaberNeumeyer}, as well as the papers cited in these papers. Earlier works can be found in
 \cite{DeCeccoPalmieri}.
 In our setting it is sufficient to define distance by considering  the class of {\it $\ep$-approximative Lebesgue curves} instead of the class of piecewise smooth curves or continuous curves,  and then to take the $\liminf$ as $\ep\to 0$ of the lengths: see (iv) of \eqref{mainthm} below.

%


\begin{thm}\label{mainthm}
Let $(M,h)$ be a smooth four dimensional Riemannian manifold satisfying \eqref{hassumptionsscaled} $1<a< \infty$ and $g_0$ satsify the assumptions of  Theorem \ref{main2}, and 
let $(M,g(t))_{t\in (0,T]}$, $T \leq 1$  be the   smooth solution to  \eqref{Meq}  appearing in the 
conclusions of    Theorem \ref{main2}.
Then \begin{itemize}
\item[(i)] 
there exists a constant $c(a)$ and a smooth solution  $\Phi :M \times (0,T] \to M$ to   \eqref{ODEDe} with 
$\Phi(T/2) = Id$ such that  $\Phi(t):= \Phi(\cdot,t) : M \to M$ is a diffeomorphism and 
 $d_h(\Phi(t)(x), \Phi(s)(x)) \leq c(a,n) \sqrt {|t - s|}$ for all $x \in M.$  The metrics $\ell(t):= (\Phi(t))^*g(t), t\in (0,T]$ solve the  Ricci flow equation.
 Furthermore there are   well defined limit maps $\Phi(0): M \to M,$  $\Phi(0): = \lim_{t\downto 0} \Phi(t),$
 and $W(0): M \to M,$  $W(0): = \lim_{t\downto 0} W(t),$ where  $W(t)$ is the inverse of $\Phi(t)$ and  
  these limits  are obtained uniformly  on compact subsets, and  $\Phi(0), W(0)$  are  homeomorphisms inverse to another. 
 
 \item[(ii)] For the Ricci flow solution $\ell(t)$ from (i), there is a value $\ell_0(\cdot) = \lim_{t\downto 0} \ell (\cdot,t) $ well defined up to a set of measure zero, where the limit exists in the $L^p$ sense, for any $p \in [1,\infty)$, such that, $\ell_0$ is positive definite, and  in $W_{loc}^{1,2}$ and for any $y_0 \in M$ and $0<s<t$  we have 
 \begin{eqnarray}
 && \int_{B_{1}(y_0)}  |\ell(s)-\ell_0|^p_{\ell(t)} d\ell(t) \leq  c(g_0,h,p,y_0) s \cr
 &&  \int_{B_{1}(y_0)}  |(\ell(0))^{-1}-(\ell(s))^{-1}|^p_{\ell(t)} d\ell(t) \leq c(g_0,h,p,y_0) |s |^{\frac 1 4}  \cr
&& \int_{B_{1}(y_0) )}|\gradg  \ell_0 |^2_{\ell(t)}   d\ell(t) \leq   c(g_0,h,p,y_0) t^{\si} \cr
 && \int_{B_{1}(y_0)} |\Rm(\ell)|^2(x,t) d\ell(x,t)   + \int_0^t \int_{B_{\ell(s)}(y_0,1)}  |\gradg \Rm(\ell)|^2 (x,s) d\ell(x,s) ds   \leq  c(g_0,h,p,y_0) \cr   
 && \sup_{B_{1}(y_0)}  |\gradg^j \Rc(\ell(t))|^2 t^{j+2}     \to 0  \mbox{ as } t \downto 0   \mbox{ for all } j\in \N_0 \cr
\end{eqnarray}
where $\si>0$ is a universal constant, $c(g_0,h,p,y_0)$ is a constant depending on $g_0,h,p,y_0$ but not on $t$, and  $\gradg$ refers to the gradient with respect to $\ell(t)$.
 \item[(iii)]
 
The  limit maps $\Phi(0): M \to M,$  $\Phi(0): = \lim_{t\downto 0} \Phi(t),$
 and $W(0): M \to M,$  $W(0): = \lim_{t\downto 0} W(t),$ from (i) 
   are also  obtained in the $W_{loc}^{1,p}$ sense for $p\in [1,\infty)$.
 Furthermore,  for any smooth coordinates $\phi :U \to \R^n$, and $\psi: V \to \R^n$  and open sets $\ti U \subsub U$ and $\ti V \subsub V$ with 
 $ W(s)(\ti V)  \subsub U$  and
 $ \Phi(s)(\ti U)  \subsub V$ for all $s \in [0,\hat T],$ for some $0 <\hat T < T$  
the functions   $(\ell_0)_{i j} \of W(0):\ti V \to \R $ are in $L_{loc}^p$ for all $p \in [1,\infty)$ and  
$(g_0)_{\al \be}: \ti V \to \R $ and $(\ell_0)_{ij}: W(0)(\ti V)  \to \R$ are related by the identity 
$$(g_0)_{\al \be}  = D_{\al} (W(0))^i    D_{\be} (W(0))^j  ( ({\ell}_0)_{i j } \of W(0))$$
on $\ti V$. 
In particular: $\ell_0$ is isometric to $g_0$ almost everywhere through the map $W(0)$ which is in $W_{loc}^{1,p},$
for all $p\in [1,\infty)$.

 \item[(iv)]  We define
 $d_t    = d(g(t)) $ and $\ti d_t = d(\ell(t)) =  \Phi(t)^*d_t.$  
 There  are   well defined limit metrics $ d_0,\ti d_0: M \times M \to \R_0^+,$  $d_0(x,y) = \lim_{t\downto 0}  d_t(x,y) $, 
 and $\ti d_0 :=  M \times M \to \R_0^+,$  $\ti d_0(p,q) = \lim_{t\downto 0} \ti d_t(p,q),$
 and they satisfy $\ti d_0(x,y) = d_0(\Phi(0)(x),\Phi(0)(y)).$ That is, $ (M, \ti d_0)$ and $(M,d_0)$ are isometric to one another through the map $\Phi(0)$.\\ 
The  metric $d_0$ satisfies $d_0(x,y):= \liminf_{\ep \downto 0} \inf_{\ga \in C_{\ep,x,y}} L_{g_0}(\ga),$ where 
$ C_{\ep, x,y} $ is the space of  Lebesgue curves between $x$ and $y$ with respect to $g_0$.


\end{itemize}

 \end{thm}
\begin{proof}

(i):\\
For $r\in (0,T)$ we define $\psi_r:M \times (0,T] \to M$ to be the solution to 
\begin{eqnarray*}
&& \partt \psi_r(y,t)  = V( \psi_r(y,t),t) ,   \cr
&&     \psi_r(y,r) = y,  \cr
 \mbox{ where } &&  V(y,t):= g^{ij}(\Gamma^k_{ij}(g) -\Gamma(h)^k_{ij})(y,t) \frac{\partial}{\partial x^k}(y)\cr
&&   \forall y \in M , t\in (0,T)  
 \end{eqnarray*}
 The Fundamental Theorem of Time Dependent Flows (see \cite{Lee}, Thm. 9.4.8)
 tells us that the $\psi_r(\cdot,s):M \to M$ are smooth diffeomorphisms for all $r,s \in (0,T]$ and that  $\psi_{t_1}(\psi_{t_0}( p, t_1),s) = \psi_{t_0}(p,s)$ for all $t_0,t_1,s  \in (0,T]$ and in particular that  $\psi_{t_1}(\psi_{t_0}( \cdot, t_1),t_0) = Id(\cdot)$ for all  $t_0,t_1  \in (0,T].$ We shall use these facts freely in the following. 
The maps $\Phi(s), W(s) : M \to M$ are defined by $\Phi(s)(x):=  \psi_{T/2}(x,s),$ and 
$ W(s)(x):= \psi_{s}(x,T/2)$ for    $s \in (0,T]$  and   $\Phi,W: M \times (0,T] \to M$ are defined by
$ \Phi(x,s) = \Phi(s)(x)$, and  $W(x,s) = W(s)(x)$ for    $s \in (0,T]$.
From the  fact that 
 $\psi_{t_1}(\psi_{t_0}( \cdot, t_1),t_0) = Id(\cdot), $
 we have 
$W(s) \of \Phi(s)(\cdot) =  \psi_{s}(\Phi(s)(\cdot),T/2) =   \psi_{s}(\psi_{T/2}(\cdot,s),T/2) = 
Id(\cdot)$ for all $s \in (0,T]$.
Defining $$\ell(s):=   (\Phi(s))^{*} g(s),$$ for $s \in (0,T]$ we obtain a smooth solution to Ricci flow
with $\ell(T/2) = g(T/2)$.


We define $\Phi(0)$ by 
$\Phi(0)(x):= \lim_{s_i \downto 0} \Phi(s_i)(x) = \lim_{s_i \downto 0} \psi_{T/2}(x,s_i)$ and 
$W(0):M \to M$ by $W(0)(x):= \lim_{t_i \downto 0} W(t_i)(x)= \lim_{t_i \downto 0} \psi_{t_i}(x,T/2) $: In the following we show that these limits exist, and are   independent of the  sequences $t_i \downto 0$ and $s_i \downto 0$ chosen.
\\
We have $|\partt \psi_{T/2}(x,t)|_h =  |V(\psi_{T/2}(x,t),t)|_h \leq \frac{\ep}{\sqrt t},$ due to the fact that $|\grad g|^2 \leq \frac{\ep}{t}$. \\Hence
$ d_h(\Phi(s)(x),\Phi(t)(x)) = d_h(\psi_{T/2}(x,s) , \psi_{T/2}(x,t))
$ $ \leq 2 \ep|\sqrt{t}-\sqrt{s}| \leq 2\ep\sqrt{|t-s|}$ for all $t,s \in (0,T]$ which shows that $\Phi(0):M \to M$ is obtained uniformly and is well defined:
$d_h(\Phi(0)(x) ,\psi_{T/2}(x,t)) = d_h(\Phi(0)(x),\Phi(t)(x)) \leq c\sqrt{t}$  for all $x \in M$.

We now turn to the construction and properties of $W$.
We can estimate  $d_h(\psi_{t_i}(x,s),x)\leq \ep \sqrt{s}$ for all $s \in (t_i,T/2],$ for all $x \in M$  in view of the fact that
$|\parts ( \psi_{t_i}(x,s) ) |_h\leq \frac{\ep}{\sqrt{s}}.$
In particular 
 writing everything w.r.t to fixed coordinates,$\phi$  where
 $ \frac{99}{100} \de \leq h \leq \frac{101}{100} \de$ and 
$|D h|^2  +|D^2 h|^2 \leq \ep$, on a large ball of radius
$1000$ and centre point $x$, we have $\psi_{t_i}(x,s), \psi_{t_j}(x,s)$ stays in this ball : We write $x$, $h,$ ... for $\phi(x),$  $\phi_*(h),$ ...\\
For $s \geq  s_0:= \max(t_j,t_i)$  
\begin{eqnarray}\label{startingvaluef}
 |\psi_{t_i}(x,s)-\psi_{t_j}(x,s)| && \leq |\psi_{t_i}(x,s)-x|  + |x - \psi_{t_j}(x,s)|  \cr
 &&  \leq  4 \ep \sqrt{s}.
\end{eqnarray}
We also have for such $s,$ 
\begin{eqnarray*}
 \parts |\psi_{t_i}(x,s) - \psi_{t_j}(x,s)|^2 
  && =|\psi_{t_i}(x,s) - \psi_{t_j}(x,s))||V(\psi_{t_i}(x,s),s) -  V(\psi_{t_j}(x,s),s)| \cr
 && \leq 2  |DV(m,s)| |\psi_{t_i}(x,s) -\psi_{t_j}(x,s)|^2\cr
 && \leq c(a)\sup_{y \in \B_{1000}(x)}  ( |\grad  ^2 g| +  |\grad g|^2 + \ep) |\psi_{t_i}(x,s) -\psi_{t_j}(x,s)|^2\cr
 && \leq \frac{\ep}{s}  |\psi_{t_i}(x,s) -\psi_{t_j}(x,s)|^2
\end{eqnarray*}
where $m$ is some value  lying on the line between $\psi_{t_i}(x,s)$ and $\psi_{t_j}(x,s)$ in the euclidean ball  $\B_{1000}(x)$.
Here we used $|\grad g^j|^2t^j \leq \ep$ when $t$ is sufficiently small. 
 Hence, writing $f(s):=  |\psi_{t_i}(x,s) -\psi_{t_j}(x,s)|^2$ we have   $ \parts f(s)   \leq \frac{\ep}{s} f(s)$ for $s \geq \max(t_i,t_j) =s_0,$ which implies
  $ \parts ( s^{-\ep} f(s))    \leq  0,$ and hence
  $f(s) \leq s^{\ep}  ( (s_0)^{-\ep} f(s_0) )$ for all $s \in [s_0=\max(t_i,t_j),T/2].$  But $f^2(s_0) =  |\psi_{t_i}(x,s_0)-\psi_{t_j}(x,s_0)|^2  \leq   16  \ep^2 s_0$ from the above estimate, and so we get  
  \begin{eqnarray*}
    (s_0)^{-\ep} f(s_0) 
  && \leq (s_0)^{-\ep}  2 \ep {s_0}  \cr
  &&  \leq  2 \ep ( s_0 )^{ 1  -\ep}\cr
&& =   2 \ep( \max(t_i,t_j) )^{1  -\ep}
   \to 0
  \end{eqnarray*}
  as $\max(t_j,t_i) \to 0,$ and hence
  \begin{align}
   f(s)  & =   |\psi_{t_i}(x,s) -\psi_{t_j}(x,s)|^2 \cr
   & \leq   2s^{\ep} \ep ( \max(t_i,t_j) )^{1 -\ep}\cr
   & \leq 2T^{\ep} \ep ( \max(t_i,t_j) )^{ 1-\ep} \to 0\label{fsestimatey}
   \end{align}
 as $t_i,t_j \to 0$, for all $s \in  (\max(t_i,t_j),T/2],$ for all $x \in M$. 
  This shows,
 $(\psi_{t_i}(x,s))_{i\in \N}$ with $t_i \downto 0$ is Cauchy  and hence 
  $    \lim_{t_i \downto 0} \psi_{t_i}(x,s) $ exists 
  for all $s \in (0,T], $ and in particular,
   $W(0)(x) :=  \lim_{t \downto 0} \psi_{t}(x,T/2) = \lim_{t \downto 0}  W(t)(x)$ is well defined, and achieved uniformly,
  \begin{align}
 d_h(W(0)(x), W(t)(x)) =  \lim_{s \to 0} d_h(\psi_s(x,T/2), \psi_{t}(x,T/2)) 
   \leq   \sqrt{2}T^{\frac{\ep}{2}} \ep (t)^{\frac{1  -\ep}{2} } \to 0
  \end{align}
 for $t\downto 0$, in view of \eqref{fsestimatey}.

We show now that $\Phi(0)$ is the inverse  of $W(0)$.
$\Phi(0)$ and $W(0)$ are  continuous, by construction, and are  the uniform limits of continuous functions, 
$\sup_{z \in M} d_h( \Phi(0)(z), \psi_{T/2}(z,t_i)) \to 0$ as $ i\to \infty$, and
$\sup_{x \in M } d_h( W(0)(x), \psi_{t_i}(x,T/2)) \to 0$ as $ i\to \infty$.
For $x\in M$, for any $\si>0$ if $i$ is large enough, we have
$ d_h( \Phi(0)(W(0)(x)),\Phi(0)(\psi_{t_i}(x,T/2))) \leq  \si,$ 
and $d_h( \Phi(0)(z),\psi_{T/2}(z,t_i) )\leq  \si$ for all $z \in M$.
This implies:
\begin{align*}
d_h( \Phi(0)(W(0)(x)), x) & \leq  d_h(\Phi(0)(W(0)(x)) ,\Phi(0)(\psi_{t_i}(x,T/2)) )  +  d_h( \Phi(0)(\psi_{t_i}(x,T/2)),x) \cr
& =  \ \ d_h(\Phi(0)(W(0)(x)) ,\Phi(0)(\psi_{t_i}(x,T/2)) ) \\
& \ \  \   +  d_h( \Phi(0)(\psi_{t_i}(x,T/2)),   \psi_{T/2}(\psi_{t_i}(x,T/2),t_i)  ) \cr
& \leq  2\sigma \cr
\end{align*}

Hence $ \Phi(0)(W(0)(x)) = x,$ as $x \in M$ and $\si>0$ were arbitrary.
Similarly, for $z\in M$, for any $\si>0$ if $i$ is large enough, we have  
 $  d_h(W(0)(\Phi(0)(z)), W(0)(\psi_{T/2}(z,t_i))) \leq   \si$ 
 and  $d_h( W(0)(x),\psi_{t_i}(x,T/2) )\leq  \si$ for all $x \in M,$ and hence
\begin{align*}
d_h(W(0)(\Phi(0)(z)),z)  & \leq  d_h(W(0)(\Phi(0)(z)),W(0)( \psi_{T/2}(z,t_i) ) )  + d_h( W(0)( \psi_{T/2}(z,t_i) ),  z) \cr
& =  d_h(W(0)(\Phi(0)(z)),W(0)( \psi_{T/2}(z,t_i) ) )  \\
& \ \  \  + d_h( W(0)( \psi_{T/2}(z,t_i) ), \psi_{t_i}( \psi_{T/2}(z,t_i),T/2) )) \cr
& \leq 2\si
\end{align*}

Hence $W(0)(\Phi(0)(z)) = z$ for all $z \in M,$  as $z \in M$ and $\si>0$ were arbitrary.

That is, $W(0)$ is the inverse of $\Phi(0)$.\\

We further have that $\Phi(s)(B_{1-\ep}(x_0)) \subseteq \Phi(0)(B_{1-\ep/4}(\ti x_0))$ for all $s= s(\ep)>0$ small enough, as we now explain,  where $\ti x_0 =  W(0) \of \Phi(s)(x_0):$
$W(0) \of \Phi(s) \to \Id$ uniformly as $s \downto 0$, as was shown above. 
This means $ W(0) \of \Phi(s) (B_{1-\ep}(x_0)) \subseteq B_{1-\frac{3\ep}{4}}(x_0) \subseteq   B_{1-\frac{\ep}{4}}(\ti x_0)$ for $s $ small enough, and hence
taking $\Phi(0)$ of both sides, the claim follows.
 
This is (i). 
(ii) Let  $y_0\in M$   and  $p\in [0,\infty)$ and $\ep \leq \al_0(p,n=4)$ be the constant from  Theorem \ref{Lpboundsriccithm}, Corollary \ref{Lpcor}, and assume also  that $\ep \leq \al_1(4,C_S(4))$ where $\al_1(n,A)$ is the constant from Theorem   \ref{W22pboundsriccithm} with $n=4$ and $A= C_S(4)$ , and $C_S(n)$ is the Sobolev constant from Theorem \ref{balllemma}.  
The construction of our solution, see Theorem \ref{main1}, Theorem \ref{main2} and the Tensor Sobolev inequality, Theorem \ref{balllemma}  (v) guarantee,  that, without loss of generality, 
$\int_{B_{2}(x)}( |\gradh g(t)|^4 +|\gradh^2 g(t)|^2 )dh  \leq  \de^4(a)  $    for all $x\in M$    and  $\de(a)$ is the constant from Corollary \ref{smallenergycor}. We also have , without loss of generality, 
\begin{eqnarray*}
|\gradh^3 g|^2 t^3 + |\gradh^2 g|^2t^2 + |\gradh g|^2t \leq \ep^2   
\end{eqnarray*}
on  $\overline{B_{200}(y_0)}$ for $t\in (0,T)$   in view of \eqref{eee_t}, and hence 
\begin{eqnarray}
| \Rm(g(t))| + |{{}^{g(t)} \nabla \Riem}|^{\frac 2 3}  \leq \frac{\ep}{t} \label{curvatureest2}
\end{eqnarray}
on  $\overline{B_{200}(y_0)}$   for $t\in (0,T),$ after reducing the time interval if necessary.  
 
By choosing $R_1 = R_1(y_0,g_0) >0$ small enough, we can guarantee that 
$\int_{B_{R_1}(x_0)}( |\gradh g_0|^4 +|\gradh^2 g_0|^2 )dh  \leq \frac{\ep}{2}$  for all $x_0$ in the compact set $\overline{B_{100}(y_0)}$  in view of Lemma \ref{smalllocallemma}.
By scaling once, we have for all such   $x_0$   that 
$\int_{B_{8}(x_0)} (|\gradh g_0|^2 +|\gradh^2 g_0|^2) dh  \leq \frac{\ep}{2} ,$
and   $\int_{B_{20}(x)} (|\gradh g(t)|^4 +|\gradh^2 g(t)|^2) dh  \leq  \de^4(a)  $ for all $x \in M$ and all $x_0$ in  $\overline{B_{100}(y_0)}$ which implies  $\int_{B_{20}(x)} (|\gradh g(t)|^2 +|\gradh^2 g(t)|^2 )dh  \leq \de(a)$ for all $x \in M,$  in view of H\"older's inequality  and the fact that $\de(a)$ is without loss of generality  less than  $\vol_h(B_{20}(x))$ for all $x \in M.$ 
Using Corollary \ref{smallenergycor},  we see  $\int_{B_{4}(x_0)}  ( |\gradh^2 g(t)|^2 + |\gradh g(t)|^2 )dh \leq  \ep$ for all $x_0$ in  $\overline{B_{100}(y_0)}$ and hence, using the fact that without loss of generality $|\Riem(h)|\leq \ep$, 
\begin{eqnarray*}
&& \int_{B_{4}(x_0)} |\Riem(g(t))|^2_{g(t)} dg(t)  \leq    2\ep,\cr
\end{eqnarray*} 
  for all $t \leq T,$ after reducing the time interval if necessary. This estimate with \eqref{curvatureest2} show that the Ricci flow related solution $\ell$  restricted to  $\Omega = B_{4}(x_0)$ for any such $x_0$ 
  satisfies all the conditions of Theorem \ref{Lpboundsriccithm}, Corollary \ref{Lpcor}    and  Theorem \ref{W22pboundsriccithm} (after scaling once more by a factor 5), and hence the estimates obtained there hold. These estimates 
change at most  by  a factor when we scale the solution back to the original solution, the constant depending   on the scaling factor, $h$ and $x_0,$ $g_0$  and $p$. These scaled estimates are  (ii) for the given $p$. As $p\in [0,\infty)$ was arbitrary, (ii) holds.

(iii) From the definition of $\ell$,  in local coordinates  we have 
\begin{align}\label{Metriceqstart} 
 & \ell_{i j }(s)(x) = D_{i}\Phi^{\al}(s)(x)    D_j\Phi^{\be}(s)(x)  g_{\al \be}(s)( \Phi(s) (x)) \cr
 &  \ \ \ \ \  \ \ \ \ \  \ \ \ \ \ \ \ \ \ \ \mbox { and }   \cr 
  &  g_{\al \be}(s)( \Phi(s) (x)) = D_{\al} W^i(s)(\Phi(s)(x))   D_{\be} W^j(s)(\Phi(s)(x))  \ell_{i j }(x,s) \cr
   &  \ \ \ \ \  \ \ \ \ \  \ \ \ \ \ \ \ \ \ \ \mbox { and }   \cr 
    &  g_{\al \be}(s)( y) = D_{\al} W^i(s)(y)   D_{\be} W^j(s)(y)  \ell_{i j }(W(s)(y),s), \cr 
\end{align}\\
where we have chosen   smooth coordinates as in the statement of the claim of (iii) of  this Theorem , $y \in \ti V$,$x \in \ti U.$
Notice that $g(t) \to g(0)$ and $g(t)^{-1} \to g(0)^{-1}$  in the $L^p_{loc}$ sense for all 
$p\in [1,\infty),$ in view of Corollary \ref{L2continuityCor}. 
Hence, we may apply  Theorem \ref{ConvMetricSob}, and we see that (iii) holds.

(iv): 
For $x,y \in M$,   
we define 
$d_0(x,y):= \liminf_{ \ep \to 0} \inf_{\ga \in C_{\ep,x,y}} L_{g_0}(\ga),$  
 where 
$ C_{\ep,x,y} $    is the space of  {\it $\ep$-approximative Lebesgue  curves} with respect to $g_0$ 
 joining $x$ and $y,$ defined in Definition \ref{lebesgueparlines}.
 
Let $x,y \in B_{\frac{R}{c(a)}}(x_0)$  where  $B_{\frac{R}{c(a)}}(x_0)$ is the ball  with respect to $h,$ for some  fixed $x_0 \in M,$  where $c(a)$ is a large constant to be determined in the proof. 
Since  $g_0 \in W^{2,2}(B_{2R}(x_0))$  we know,  from Lemma \ref{smalllocallemma} of Appendix B  that for any $\si>0$ there exists a  $r(\si,R,a)>0$ such that $\int_{ B_{2r}(z) } (|\gradh^2 g_0|^2 + |\gradh g_0|^4 )dh \leq \si^4,$  for all $z \in B_{R}(x_0). $

Scaling $g_0$ and $h$ once, by the same large constant $K$, and still calling the new scaled metrics  $g_0$,  $h$, and   the new radius  $  \sqrt{K} R$ will still be denoted by $R$, we have  
$\int_{ B_{2}(z) } (|\gradh^2 g_0|^2 + |\gradh g_0|^4)dh  \leq \si^4$ for all $ z\in B_{ R}(x_0)$ 
and without loss of generality $\sup_M \sum_{i=1}^4 |\gradh^i \Riem(h)| \leq \si^4$. 
Hence, using Corollary  \ref{smallenergycor},    H\"older's inequality and Lemma \ref{balllemma}  (v), 
\begin{eqnarray}
&& \int_{B_{2}(z)}  ( |\gradh^2 g(t)|^2 + |\gradh g(t)|^4 )dh \leq  \si ,
\cr
&& \sum_{i=1}^4 |\grad^{i} \Riem(h)| \leq \si \cr  
 && \frac{1}{400   a} h  \leq g(t) \leq 400   a h, \cr 
 && |\gradh^j  g(t)|  \leq \frac{c(j,a)}{t^j}, \label{startuptwo}
\end{eqnarray} 
 for all $j\in \N$  for all $ z\in B_{R}(x_0)$ 
 and for all $0<t \leq S_2(400a,\si),$ and after scaling once more by $\frac{1}{2 S_2}$, for all $t\in (0,2]$.

\begin{eqnarray}
\mbox{ We first show that  } &&
d_0(x,y)  \leq \liminf_{t \to  0} d_t(x,y)  \label{claim1}
\end{eqnarray}

Let  $\ep>0$ given. Taking any $0<t\leq \ep^4$ and scaling by $\hat g(s) = \frac{1}{t} g(st),$ and denoting the new radius by $\hat R$, that is  $\hat R= \frac{1}{\sqrt{t}} R,$ and $\hat h = \frac{1}{t} h,$ $\hat g_0 = \frac{1}{t} g_0,$ we see, in view of
\eqref{startuptwo},  that we obtain a new solution $\hat g(s),$ $s \in [0,2], $  such that
  \begin{eqnarray*}
&& \int_{\hat B_{2}(z)} (  |\grad^2 \hat g(s)|^2 + |\grad \hat g(s)|^4  )dh\leq 2\si,
\cr
 && \frac{1}{400 a} \hat h  \leq \hat g(s) \leq 400 a \hat h, \cr 
 && |\grad^j  \hat g(s)|  \leq \frac{c(j,a)}{s^j},\cr 
 && \sum_{i=1}^4 |\grad^{i} \Riem(\hat h)| \leq \si \cr 
 && |\grad^j  \hat g(1)| \leq c(j, a,\si)
\end{eqnarray*}  
for all $j\in \N$  for all $ z\in \hat B_{\hat R}(x_0)$ 
 and for all $s \leq 1,$ 
 where $c(j ,a,\si) \to 0$ for $\si \to 0,$ where $\hat B_s(m)$ refers to a ball of radius $m$ with respect to $\hat h$.
 For later use not that $\hat R \geq \frac{1}{\ep^2}.$
Let $\ga$ be a length minimising geodesic between $x$ and $y$ with respect to $\hat g(1)$. 
Writing $\hat g(1)$ in geodesic coordinates at any $z \in \hat B_{\hat R}(x_0)$ on a ball of radius one,  we have
\begin{align*}
 &(1 - |\al(\si,a)|)) \de \leq  \hat g(1) \leq (1+ |\al(\si,a)|) \de  \mbox{ on } \B_{10}(0) \cr
  & \frac{1}{400 a} \hat h \leq  \hat g(s) \leq 400  a \hat h  \mbox{ for } s \in [0,2] 
 \end{align*}
  where $\al(\si,a ) \to 0$ as $\si \to 0.$ 

 In the following, any constant $c(\si,a)$  with  $c(\si,a) \to 0$ as $\si \to 0,$ shall be denoted by $\al(\si,a)$ although it may differ from the one just defined.  $\al(\si,a)$ is not necessarily larger than zero. 

We can break $\ga $ up into $N$ pieces, $\ga_1 = \ga|_{[0,1]}$,
$\ga_2 =  \ga|_{[1,2]},$ $\ga_{N-1}=  \ga|_{[N-2,N-1]},$ $\ga_N = \ga|_{[N-1,B]}$ each with length one with respect to $\hat g(1)$, except for the last piece which has length less than or equal to one. Due to  $ \frac{1}{400 a} \hat h \leq   \hat g(1) \leq 400 a \hat h,$ we have 
$N \leq c(a) \hat R$.
  After rotating once, we may assume that any length one   piece of $\ga,$ going from $\ga(i)$ to $\ga(i+1),$  $i\in \{0, 2, \ldots N-2\},$  in geodesic coordinates, with respect to $\hat g(1)$  centred at $z= \ga(i +\frac{1}{2})$ 
 lies in  $\overline{ \B_{2}(0)},$ and  is   (in these coordinates) the  line segment $v:[-\frac 1 2 , \frac 1 2] \to \overline{ \B_{2}(0) }$, $v(s) = se_1.$ We ignore  the last piece of $\ga$ for the moment.  

Using  Corollary \ref{L2continuityCor}  and $|\Rm(\hat h) | \leq \sigma$ we have 
  $\int_{\B_{1}(0)}   |\hat g(t)-\hat g_0|^2 dh \leq 2\sigma t$ for all $t\leq 1$
and hence  \begin{align}\label{gzerotogt}
\int_{\B_{1}(0)}   |\de-\hat g_0|^2  \leq \al(\sigma,a). 
\end{align}
Let $\ep>0$ be given.
Using Lemma  \ref{distboundabove} we see, by choosing $\si= \si(\ep)>0$ small enough, that 
there exists an $x \in \B^{n-1}_{\ep}(0)$ such that 
$\sqrt{ \hat g_{11}(0) }(\cdot,x): [-\frac 1 2,\frac 1 2] \to \R^n$ is measurable, 
$\ell:[-\frac 1 2,\frac 1 2] \to \R^n,$ $\ell(t) = (t,x) $ is a Lebesgue line between $(-\frac 1 2,x)$ and $(\frac 1 2,x)$ and 
\begin{eqnarray*}
\int_{-\frac 1 2}^{\frac 1 2} \sqrt{\hat g_{11}(0)(s,x)}ds  \leq 1+ \ep &&  = ( 1+ \ep ) d_{\de}( ( -\frac 1 2,x),(\frac 1 2,x)) \cr
&& \leq ( 1+ 2\ep )d_{\hat g(1)}(( -\frac 1 2,x),(\frac 1 2,x))
\end{eqnarray*}

which  tells us for the original curve $\ga|_{[0,N-1]}$ that there exists   Lebesgue curves $v_i:[i-1,i] \to \hat B_{\hat R}(x_0)$ for all $i \in \{1, \ldots,N-1\}$   w.r.t.  to $\hat g(0)$ such that $d_{\hat h}(v_i(i-1),\ga(i-1)) \leq c(a) \ep$ and $ d_{\hat h}( v_i(i),\ga(i)) \leq c(a) \ep$ 
and $L_{\hat g_0}(v_i) \leq (1+\ep) d_{\hat g(1)}( \ga(i-1),\ga(i))$ for all $i \in \{1, \ldots,N-1\}:$ The curves $v_i$
are the curves $\ell$ constructed above. 
  Adding up all the curve segments we have 
\begin{align*}
\sum_{i=1}^{N-2} d_{\hat h}(v_i(i),v_{i+1}(i))    \leq 
N \ep c(a) \leq \hat R  \ep c(a)
\end{align*}
and  hence
\begin{align*} 
\sum_{i=1}^{N-2} d_{\hat h}(v_i(i),v_{i+1}(i))   + d_h(x,v_1(0)) + d_h(y,v_{N-1}(N-1))     \leq \hat R  \ep c(a) + 2
\end{align*}
and also 
\begin{eqnarray*}
  \sum_{i=1}^{N-1} L_{{\hat g}_0}(v_i) \leq   (1+ \ep)d_{\hat g(1)}(x,\ga(N-1))  
  \leq   (1+ \ep)d_{\hat g(1)}(x,y)\cr
 \end{eqnarray*}
 as $\ga$ was a length minimising, with respect to $\hat g(1)$, geodesic between $x$ and $y.$ 
 Scaling back the solution we had at the beginning of this  proof of this claim,  \eqref{claim1}, 
that is defining $  g(s) =t g(s \frac{1}{t}),$  for the $t$ we chose there, we see at time $t$ that 
 \begin{eqnarray*}
 && \sum_{i=1}^{N-1} L_{ g_0}(v_i) \leq   (1+ \ep)d_{ g(t)}(x,y)
 \end{eqnarray*}
 and
 \begin{align*}
&\sum_{i=1}^{N-2} d_h(v_i(i),v_{i+1}(i)) + d_h(x,v_1(0)) + d_h(y,v_{N-1}(N-1))\cr
 &  \leq  R  \ep c(a) + 2\sqrt{t} \cr
& \leq  R  \ep c(a) + 2\ep
\end{align*}
in view of the choice of $t\leq \ep^4$. 
That is $v= \cup_{i=1}^{N-1} v_i$  is an $Rc(a)  \ep$-approximative Lebesgue curve
and $L_{g_0}(v)  \leq   (1+  \ep)d_{ t}(x,y).$
Hence $\inf_{\ga \in C_{R c(a) \ep,x,y}}  L_{g_0}(\ga)  \leq (1+ \ep)d_{ t}(x,y)$ for all $t\leq T(\ep,a,g_0,x,y,R,h),$
and this shows 
$d_0(x,y)= \liminf_{ \ep \to 0} \inf_{\ga \in C_{\ep,x,y}} L_{g_0}(\ga) \leq \liminf_{t \to  0} d_t(x,y)$
for all $x, y\in M$.

Now we show that 

 $ d_0(x,y)  := \liminf_{ \ep \to 0} \inf_{\ga \in C_{\ep,x,y}} L_{g_0}(\ga)  \geq \limsup_{t\downto 0}  d_t(x,y)$
 for all $x,y \in M$.

From the definition of  $C_{\ep,x,y},$ 
$\ga \in C_{\ep, x,y},$   may be written as $\ga = \cup_{i=1}^N \ga_i$ where each $\ga_i:[a_i,b_i]\to M$ is a parametrised Lebesgue line. 
Let $\si:[c_i,d_i] \to \B_2(x)$ be one of the segments  $\ga_i$ written in smooth coordinates, so that   $\si(t) =     te_1.$
Since the coordinates are smooth, and $\frac{1}{400a}   h \leq   g(t) \leq 400 a  h$,  we see that there exists a constant $C$ depending possibly on the coordinates  and $a$, such that
$      \frac{1}{C}\de  \leq   h \leq C \de$   and  
$      \frac{1}{C}\de  \leq   g(t) \leq C  \de$ in these coordinates.

Using Corollary \ref{L2continuityCor} and $|\Rm(  h) | \leq \sigma \leq 1,$ we see  that we have 
  $$\int_{\B_{2}(0)}   |  g(x,t)- g_0(x)|_{\de}^2 dx \leq c t$$ with respect to these coordinates for all $t\leq 1$
  for some constant $c=c(C).$

Using Lemma \ref{distanceboundbelow}, 
we see that : For all $\ep >0$ there exists a $t_0 >0$ such that 
 \begin{eqnarray*}
L_{g_0}(\si):= \int_{c_i}^{d_i} \sqrt{g(0)_{11}(s,0)}  \geq (1-\ep)d(g(t))(\si(c_i),\si(d_i)) 
\end{eqnarray*}
for all $t\in (0,t_0)$ in these coordinates. 

That is $L_{g_0}(\ga_i) \geq  (1-\ep)d(g(t))(\ga_i(a_i),\ga_i(b_i))).$

Hence, estimating on each  Lebesgue line $\ga_i$ in this way, we see (setting $\ga_{N+1}(a_{N+1}) := y$  that, there exists a $s_0$ such that  
\begin{align*}
L_{g_0}(\ga) 
& \geq \sum_{i=1}^N (1-\ep)d_t(\ga_i(a_i), \ga_{i}(b_i))  \cr
& \geq  \sum_{i=1}^N (1-\ep)d_t(\ga_i(a_{i}), \ga_{i+1}(a_{i+1}) ) 
-  \sum_{i=1}^N d_t(\ga_i(b_i), \ga_{i+1}(a_{i+1})) \cr
& \geq  \sum_{i=1}^N (1-\ep)d_t(\ga_i(a_{i}), \ga_{i+1}(a_{i+1})) 
-  c\sum_{i=1}^N d_h(\ga_i(b_i), \ga_{i+1}(a_{i+1}) ) \cr
& \geq    (1-\ep)d_t(\ga_1(a_1), \ga_{N+1}(a_{N+1})= y) -c\ep \cr
& \geq (1-\ep)d_t(x,y) -2c\ep
\end{align*}
for   $t\in (0,s_0)$.
That is, for fixed $x,y \in M$ we have  
$
 d_0(x,y) \geq \limsup_{t\downto 0}  d_t(x,y),
$
in view of the definition of $d_0$ and the fact that $\ep>0$ was arbitrarily chosen in the argument above.
Combining the lower and upper bound proved for 
$d_0(x,y),$ we have
$$
 d_0(x,y) :=\liminf_{ \ep \to 0} \inf_{\ga \in C_{\ep,x,y}} L_{g_0}(\ga)  =   \lim_{t\downto 0}  d_t(x,y)
$$
as claimed.

The property $\ti d_0 = \lim_{t\downto 0} \ti d_t$ now follows easily from the definitions and the fact that $\Phi(t)$ converges uniformly to $\Phi(0)$ as $t\downto 0:$
\begin{eqnarray*}
   && \ti d_0(x,y) \cr 
    &&:=  \ \     \ \ \  d_0(\Phi(0)(x),\Phi(0)(y)) \cr
    &&  = \ \  \ \ \ \  d_t(\Phi(0)(x),\Phi(0)(y)) + \ep( x, y ,t) \cr
    &&  \leq (\geq )   \ \  d_t(\Phi(0)(x),\Phi(t)(y)) \cr
    &&   \ \ \ \ \ \  \ \ \ \  +  (-) \ \  d_t(\Phi(t)(y),\Phi(0)(y))+ \ep( x, y ,t) \cr
    &&  \leq (\geq )    \ \ d_t(\Phi(t)(x),\Phi(t)(y)) + (-)  d_t(\Phi(t)(x),\Phi_0(x))   \cr 
    &&  \ \ \ \    \ \ \ \ \ \ 
     +  (-) \ \ d_t(\Phi(t)(y),\Phi(0)(y))+ \ep( x, y ,t) \cr
    &&    \leq (\geq )  \ \ d_t(\Phi(t)(x),\Phi(t)(y)) + (-)  \ \ c(n,a)d_h(\Phi(t)(x),\Phi(0)(x))   \cr 
   &&   \ \ \ \   \ \ \ \   \ \  + (-) \ \ c(n,a)d_h(\Phi(t)(y),\Phi(0)(y))+ \ep( x, y ,t) \cr
    &&    =   \ \ \ \  \ \   \ti d_t(x,y)  + (-)  \ \ c(n,a)d_h(\Phi(t)(x),\Phi(0)(x)) \cr
    &&   \ \ \ \   \ \ \ \   \ \
     +(-) \ \ c(n,a) d_h(\Phi(t)(y),\Phi(0)(y)) +   \ep( x, y ,t) 
  \end{eqnarray*}
  where $\ep(x,y,t) \to 0$ for $t\downto 0$. 
d.h. $\ti d_t(x,y) \to \ti d_0(x,y)$ for $t\downto 0$.

 \end{proof}

\section{Metric convergence   in Sobolev spaces}\label{metricconvergencesobolevspaces}

\begin{thm}\label{ConvMetricSob}
Let $(M,h)$ be a four dimensional Riemannian manifold satisfying \eqref{hassumptionsscaled}. 
We assume that  there are continuous maps
$W(0),\Phi(0): M \to M$ inverse to one another such that for all compact sets $K\subseteq M,$ 
$ \sup_{x\in K} d_h(\Phi(r)(x),\Phi(0)(x))  \to 0 \mbox { as } r\to 0,$
$ \sup_{y\in K} d_h(W(r)(y),W(0)(y))  \to 0 \mbox { as }r \to 0,$
where $\Phi, W :M  \times (0,T)  \to M,$  are smooth maps such that 
$W(s): M \to M, \Phi(s)$ are smooth diffeomorphisms   inverse to one another,  for all $s\in (0,T).$ 
Let $\ell(s)_{s \in (0,T)}, g(s)_{s\in(0,T)}$ be  smooth families of Riemannian metrics
isometric to one another through the smooth maps
$\Phi(s)$ and $W(s):$ $\ell(s) = \Phi(s)^{*}(g(s)),$ 
$g(s)= W(s)^{*}\ell(s).$
 Assume that we have chosen  smooth coordinates
$\phi :U \to \R^n$, and  $\psi: V \to \R^n$ 
 and open sets $\ti U \subsub U$ and $\ti V \subsub V$ with 
 $ W(s)(\ti V)  \subsub U$  and
 $ \Phi(s)(\ti U)  \subsub V$ for all $s \in [0,S],$
 $ W(s)(V)  \subsub U$ for all $s\in[0,S]$  for some $0<S<T$. That is, in these coordinates we  have
\begin{align}
 & \ell_{i j }(s)(x) = D_{i}\Phi^{\al}(s)(x)    D_j\Phi^{\be}(s)(x)  g_{\al \be}(s)( \Phi(s) (x)) \label{Metriceq} \\
 &  \ \ \ \ \  \ \ \ \ \  \ \ \ \ \ \ \ \ \ \ \mbox { and }   \cr 
  &  g_{\al \be}(s)( \Phi(s) (x)) = D_{\al} W^i(s)(\Phi(s)(x))   D_{\be} W^j(s)(\Phi(s)(x))  \ell_{i j }(x,s) \cr
   &  \ \ \ \ \  \ \ \ \ \  \ \ \ \ \ \ \ \ \ \ \mbox { and }   \cr 
    &  g_{\al \be}(s)( y) = D_{\al} W^i(s)(y)   D_{\be} W^j(s)(y)  \ell_{i j }(W(s)(y),s)  \nonumber
\end{align}\\
for $x \in \ti U,$ $y \in \ti V.$
Assume further, that there exist   Riemannian metrics  $\ell(0)$ and $g(0)$  whose inverse exists almost everywhere, so that   $g(0), \ell(0),g^{-1}(0),  \ell^{-1}(0)  \in L^p_{loc}$ for all $p \in [1,\infty)$ such that
\begin{itemize}
\item[(i)]  $\ell(s) \to \ell(0),  \ell^{-1}(s) \to \ell^{-1}(0)  , g(s) \to g(0), g^{-1}(s) \to g^{-1}(0)  $  as $s \to 0$ locally in $L^p$ for all $p\in [1,\infty)$,
\item[(ii)] For any compact set $K \subseteq M,$
for all $ s \in (0,T), p\in [1,\infty) ,$ there exists a constant $c(K,h,p)$ such that  
$\| \ell(s)\|_{W^{1,2}(K)} + \|g(s)\|_{W^{1,4}(K)}  \leq c(K,h,p) ,$
\item[(iii)] There is a constant $a \geq 1 $ such that $\frac{1}{a} h  \leq g(s) \leq a h$ for all  $ s \in (0,T).$

\end{itemize}

Then:  $ D\Phi(s)$ is bounded in $L^p(\ti U)$ and $D W(s)$ is   bounded  in $L^p(\ti V)$ uniformly independent of $s\in (0,T).$ 
Furthermore, $\Phi(s) \to \Phi(0)$ and $W(s) \to W(0)$ locally in $W^{1,p}$ for any 
$p\in [1, \infty),$  $\ell(0) \of W(0)$ is in $L^p_{loc}$ for all $ p\in [1,\infty),$
and $\ell(0)$ and $g_0$ are isometric to one another through the  map $W(0),$ which is in $W^{1,p}_{loc},$
for all $p\in [1,\infty):$
$$(g_0)_{\al \be}  = D_{\al} (W(0))^i    D_{\be} (W(0))^j  ( ({\ell}_0)_{i j } \of W(0))$$ on $\ti V$ in the $L^p$ sense for all $p\in [1,\infty)$.
\end{thm}
\begin{proof}
We consider in the following only $s \in (0,S).$ 
 The first identity of \eqref{Metriceq} implies:
\begin{eqnarray*}
  h^{ij}(x)  \ell_{i j }(s)(x) && = 
 h^{ij}(x) D_{i}\Phi^{\al}(s)(x)    D_j\Phi^{\be}(s)(x)  g_{\al \be}(s)( \Phi(s) (x))\cr
   && \geq \frac{1}{a} h^{ij}(x) D_{i}\Phi^{\al}(s)(x)    D_j\Phi^{\be}(s)(x)  h_{\al \be}( \Phi(s) (x)),
\end{eqnarray*}
and hence using the fact that $h^{ij}\ell_{ij}(s)$ is locally uniformly bounded in $L^p$  independent of $s$,  we see that
$D\Phi(s)$ is locally uniformly bounded in $L^p$ , that is independent of $s$, for any $p\in[1,\infty)$:
\begin{eqnarray*}
&& \ \ \ \int_{\ti U}  | h^{ij}(x) D_{i}(\Phi(s))^{\al}(x)    D_j(\Phi(s))^{\be}(x)  h_{\al \be}( \Phi(s)(x))|^p d_h(x) \leq c(p,\ldots) < \infty,
\end{eqnarray*}
where $c(p,\ldots)$ is a constant depending on $p, $ $\ti U,$ $a,$ $\ell,$ $h$. Constants  which only depend on 
$p,$ $\ti U, \ti V , \phi,\psi,\Phi,W,$ $a,$  $\ell,$ $h$  , $g$ $\ell^{-1}$ $h^{-1}$  , $g^{-1}$ and importantly {\it don't depend on $s$}   shall be denoted by  $c(p,\ldots),$   although the value  may differ from line to line.
In view of the uniform convergence of $\Phi(s)$ and $W(s)$ we may  assume, by choosing geodesic coordinates with respect to $h$ around $m$ and 
$ \Phi(0)(m)$ without loss of generality, that 
$\ti U= B_r(m)=: B$  and $\ti V= B_v(\Phi(0)(m)) =:\hat B$ so that $\frac{1}{2} \de_{\al \be}  \leq h_{\al \be} \leq 2 \de_{\al \be}$ on $\ti V$
and $\frac{1}{2} \de_{i j}  \leq h_{i j} \leq 2 \de_{i j}$ on $\ti U,$ and $\ti B = B_{\ti r}(m) \subseteq W(s)(\hat B) \subseteq    B = B_{r}(m)$ for $s$ sufficiently small.   With respect to these coordinates we have 
\begin{eqnarray}
 \int_{B}  \sum_{i,\al =1}^{n} |D_{i} \Phi^{\al}(s)(x)|^{2p}  dx\leq c(p,\ldots)
 \end{eqnarray} 
for $s$ sufficiently small. 
Using the third identity of \eqref{Metriceq}  and these coordinates we see that
\begin{eqnarray*} 
&& \int_{\hat B } (\sum_{\al,i=1}^n  (D_{\al} W^i(s)(y))^{2} )^p      dy \cr
&& = \int_{\hat B} |\de^{\al \be} D_{\al} W^i(s)(y)   D_{\be} W^j(s)(y)  \de_{i j }|^p dy\cr
&& \leq c(p,\ldots) \int_{\hat B } |h^{\al \be}(y) D_{\al} W^i(s)(y)   D_{\be} W^j(s)(y)  h_{i j }(W(s)(y),s)|^p dy  \cr
&& = c(p,\ldots) \int_{\hat B } |DW(s)|^{2p}_{ h,h(W(s))}(y) dy \cr
&& \leq c(p,\ldots) \int_{\hat B} |DW(s)|^{2p}_{ h, \ell(W(s),s)}(y) (1 + |h \of W(s)|^{2p}_{\ell (W(s),s) }(y))dy\cr
&&[ \mbox{ in view of } \eqref{firstD}\mbox{ of Theorem } \ref{lpmetricthm}]\cr
&& = c(p,\ldots) \int_{\hat B} |h^{\al \be} g_{\al \be}(s)|^{p}(y)  (1 + |h|^{2p}_{\ell (s)}\of W(s))(y) dy \cr 
&& \leq c(p,\ldots)  ( \int_{\hat B} |h^{\al \be}g_{\al \be}(s)(y)|^{2p}dy )^{\frac 1 2} (\int_{\hat B}   (1 + |h|^{4p}_{\ell (s)}\of W(s)(y) ) dy )^{\frac 1 2}\cr 
&&  \leq c(p,\ldots)  (1+(\int_{\hat B}   |h|^{4p}_{\ell (s)}\of W(s)(y) dy )^{\frac 1 2})\cr 
&&  =  c(p,\ldots)  (1+(\int_{\hat B}   |h|^{4p}_{\ell (s)}\of W(s)(y) \norm {DW(s)}(y) \norm{ D\Phi(s) \of W(s)}(y)  dy  )^{\frac 1 2}) \cr 
&& [ \mbox{ since } DW(s) \cdot D\Phi(s) \of W(s) = ID] \cr 
&& \leq   c(p,\ldots)  (1+(\int_{W(s)(\hat B)\subseteq B_r(m)}   |h|^{4p}_{\ell (s)}(x) \norm{ D\Phi(s) (x)}  dx )^{\frac 1 2})\cr 
&& \leq c(p,\ldots) 
\end{eqnarray*}
since $(\ell)^{-1},$ $ g $ and $D\Phi(s)$  are  bounded locally in $L^p$ independently of $s$, for $s$ sufficiently small,  and we have used the transformation formula, and the notation $\norm{A}$ to represent    $\det(A)$.

We also have

\begin{eqnarray*}
 &&\int_{  \ti B } |D(g_s \of \Phi(s))|^{4-\ep}(x)dx \cr
&& \leq  \int_{W(s)(\hat B)} |D(g_s \of \Phi(s))|^{4-\ep}(x)dx  \cr
&&  =  \int_{W(s)(\hat B)} |((D g_s) \of  \Phi(s) )(x) \cdot D\Phi(s)(x)|^{4-\ep}dx \cr
&&  \leq (\int_{W(s)(\hat B)} (|D g_s| \of \Phi(s) )^{4-\frac \ep 2}  (x) dx)^{\frac 1 q}(\int_{W(s)(\hat B)} |D\Phi(s)|^{r(4-\ep)}(x) dx)^{\frac 1 r} \cr
&& [ \mbox {where } q= \frac{4-(\ep/2)}{4-\ep}, r  = \frac{1}{1-(1/q)} ] \cr 
&&  \leq c(p(\ep),\ldots)( \int_{W(s)(\hat B)} |D g_s|^{4-\frac \ep 2}  \of \Phi(s)(x)dx)^{\frac 1 q} \cr
&&  = c(p(\ep),\ldots) (\int_{W(s)(\hat B)} |D g_s|^{4-\frac \ep 2 }  \of \Phi(s) (x)  \cdot  \norm {DW(s) \of  \Phi(s)(x)}  \norm{D\Phi(s)(x)}dx  )^{\frac 1 q}\cr
&& [\mbox{ since } I = D(W(s) \of \Phi(s))  = (DW(s)  \of \Phi(s))  \cdot D \Phi(s) ]\cr
&&   = c(p(\ep),\ldots) ( \int_{\hat B}  |Dg_s|^{4-\frac \ep 2}(y)   \norm{DW(s)}(y)dy )^{\frac 1 q}\cr
&&  \leq c(p(\ep),\ldots)(\int_{\hat B}  |D g_s|^{4}(y)dy )^{ {\frac 1 {qv}} }  (\int_{\hat B} \norm{DW(s)}^udy)^{ {\frac 1 {qu}} } \cr
&& [ \mbox {where } v = \frac{4}{4-\frac \ep 2},  u = \frac{1}{1- \frac{1}{v} }]\cr
&& \leq c(p(\ep),\ldots) 
\end{eqnarray*} 
for sufficiently small $s>0,$
where $c(p(\ep),\ldots )$ is independent of $s$, and $2>\ep>0$   is arbitrary, since
$Dg$ is bounded in $L^4$ due to the assumptions, and $DW(s)$ is bounded uniformly  in every $L^p$ for every fixed $p$. 
, and we used $(\int_{W(s)(\hat B)} |D\Phi(s)|^{r(4-\ep)}(x) dx)^{\frac 1 r} \leq c(p,\ep,\ldots)$.
Hence for any sequence $t_i>0$ with $t_i \to 0$ as $i\to \infty,$ we can find a subsequence $s_i$ with $s_i \downto 0$ such that   $(g_{s_i})_{\al \be} \of \Phi(s_i)$ converges strongly in $L^p$ to some $Z_{\al \be}$ on $\ti B$  for all $p$ as $s_i$ goes to zero,  in view of the Sobolev-Embedding theorems (see for example Theorem 7.26 in \cite{GilbargTrudinger}). 
$Z$ satisfies  $ \frac{1}{ C(a)} \de_{\al \be} \leq Z_{\al \be} \leq C(a) \de_{\al \be}$  on  $\ti B$ since  

$\frac{1}{2} \de_{\al \be}  \leq h_{\al \be} (\cdot) \leq 2 \de_{\al \be}$ on $\phi(s)(\ti B)  \subseteq \ti V$ and $\frac{1}{a} h \leq g(s) \leq a h$.

For a sequence $0<s_i \downto 0,$ 
 we write $g(i) = g(s_i)$ and $\Phi(i) = \Phi(s_i)$ and $\ell(i) := \ell(s_i)$. 

Using \eqref{Metriceq} we see 
$$\de^k_s  =  \ell(i)^{rk}(x) D_{s}  \Phi(i)^{\al}(x)   D_{r}\Phi(i)^{\be}(x) g(i)_{\al\be} \of \Phi(i)(x) $$ and hence

\begin{eqnarray*}
0 && = \ell(i)^{rs}(x) D_{s}  \Phi(i)^{\al}(x)   D_{r}\Phi(i)^{\be}(x) g(i)_{\al\be} \of \Phi(i)(x)\\
&& 
\ \ \ - \ell(j)^{rs}(x) D_{s}  \Phi(j)^{\al}(x)   D_{r}\Phi(j)^{\be}(x) g(j)_{\al\be} \of \Phi(j)(x)\cr
&& = \ell_0^{rs}(x)  Z_{\al \be}(x) (D_{s}  \Phi(i)^{\al}(x)   D_{r}\Phi(i)^{\be}(x)\cr
&&  \ \ \ -
D_{s}  \Phi(j)^{\al}(x)   D_{r}\Phi(j)^{\be}(x) )  + err(i,j)(x)\cr
&& = |D\Phi(i) -D\Phi(j)|_{\ell_0,Z }^2(x)+ err(i,j)(x)
\end{eqnarray*}
where $\err(i,j)$ is an error term which goes to 0  in the $L^p$ sense, on $\ti B$  for any $p\in [2,\infty)$ as $i,j\to \infty,$ 
and hence, using \eqref{secondD} of Theorem \ref{lpmetricthm}, we see
\begin{eqnarray*}
&&\int_{\ti B}  |D\Phi(i) -D\Phi(j)|^p_{\de,\de}(x)  dx\cr
&&  \leq c(p, \ldots) \int_{\ti B}  |D\Phi(i) -D\Phi(j)|_{h,Z }^p(x) dx\cr
&&\leq c(p,\ldots)  \int_{\ti B}  (1+|\ell_0|^2_{h} )^{p/2}(x)  |D\Phi(i) -D\Phi(j)|_{ \ell_0,Z}^p(x) dx  \cr
&& \leq c(p,\ldots)  (\int_{\ti B}  (1+|\ell_0|^2_{h} )^{p}(x) dx )^{\frac 1 2} 
(\int_{\ti B}  |D\Phi(i) -D\Phi(j)|_{ \ell_0,Z}^{2p}(x) dx)^{\frac 1 2} \cr
&& =  c(p,\ldots)(  \int_{\ti B}(err(i,j))^{2p}(x)dx)^{\frac 1 2}
\end{eqnarray*}
where $ \int_{\ti B} |err(i,j)|^{2p} \to 0$ as $ i,j \to \infty.$ 
That is $D\Phi(i)|_{\ti B}$ is Cauchy in $L^p({\ti B})$.
Since $\Phi(i) \to \Phi(0)$ in locally, uniformly, and hence  locally in  $L^p$ for any $p<\infty$, we see  $\Phi(i) \to \Phi(0)$ in $W^{1,p}(\ti B)$ as $i\to \infty$.
In fact $\Phi(s) \to \Phi(0)$ in $W^{1,p}({\ti B})$  as $s \downto 0$. If this were not the case, then we could find a sequence of times $t_i \to 0$ such that 
  $|\Phi(t_i) -\Phi(0)|_{W^{1,p}({\ti B})} \geq \si>0$. Repeating the above argument, we see that a subsequence $\Phi(s_i)$ of $\Phi(t_i)$  converges to $\Phi(0)$ in $W^{1,p}({\ti B})$, which contradicts  $|\Phi(t_i) -\Phi(0)|_{W^{1,p}({\ti B})} \geq \si>0.$

We now show that a subsequence of  $\ell_{t_i} \of W(t_i)$ converges in $L^p$ locally for any sequence $t_i\downto 0$. 
For   $0<4\de <3$ we have (where here the norm, $|\cdot|$ refers to the    euclidean norm) 
  
\begin{eqnarray*}
 \int_{\hat B} |D(\ell_s \of W(s))|^{1+\de}dy  &&  =  \int |(D\ell \of  W(s) ) DW(s)|^{1+\de}dy  \cr
&&  \leq   \int_{\hat B} |D\ell|^{1+2\de}  \of W(s)dy  + \int_{\hat B} |DW(s)|^v dy\cr
&&[\mbox{ with } v = \frac{1}{1-(1/q)}, q = \frac{1+2\de}{1+ \de} ]\cr
&&  \leq   \int_{\hat B} |D\ell|^{1+2\de}  \of W(s)dy + c(v(\de),\ldots)   \cr
&&  =    \int_{\hat B} (|D\ell|^{1+2\de}  \of W(s)) \norm {D\Phi(s) \of W(s) }  \norm{DW(s)}dy  + c(v(\de),\ldots)  \cr
&&   =    \int_{W(s)(\hat B)}   |D\ell|^{1+2\de}dx   \norm{D\Phi(s)} +c(v(\de),\ldots)  \cr
&&  \leq  \int_{W(s)(\hat B)}  |D\ell|^{1+4\de}dx  +   \int_{W(s)(\hat B)} \norm{D\Phi(s)}^{\hat p}dx  + c(v(\de),\ldots)   \cr
&&[\mbox{ with } \hat p = \frac{1}{1-(1/\hat q)}, \hat q = \frac{1+4\de}{1+ 2\de} ]\cr
&& \leq c(v(\de),\ldots) 
\end{eqnarray*}
where $c(v(\de),\ldots) $ is independent of $s$.
Hence a subsequence $(s_i)_{i\in \N} $ of $(t_i)_{i\in \N}$ satisfies 
 $\ell_{jk}(s_i) \of W(s_i)  \to R_{jk}$ in $L^{\al(\de)}(\hat B)$ for some $R \in L^{\al}(\hat B)$ from the Sobolev embedding Theorems (see for example Theorem 7.26 in \cite{GilbargTrudinger}).  In particular 
$\ell_{jk}(s_i)\of W(s_i) \to R_{jk}$ almost everywhere on $\hat B$.  
On the other hand, the transformation  formula for smooth diffeomorphisms shows us that
\begin{eqnarray*}
\int_{\hat B} |\ell_s\of W(s)|^p  dy && = \int_{\hat B} |\ell_s\of W(s)|^p  \norm{D\Phi(s)}\of W(s)  \norm{DW(s)}  dy\cr
&& = \int_{W(s)(\hat B)} |\ell_s|^p   \norm{D\Phi(s)}  dx \cr
&&  \leq c(p,\ldots)
\end{eqnarray*}
in view of H\"older's Theorem,  since  $\ell_s$ and $D\Phi(s)$ are  uniformly bounded in $ L^p(B)$ for all $ p \in [1,\infty)$.
This  shows us  that $ \ell(s_i)\of W(s_i) \to R$ in $L^p(\hat B)$ for all $p\in [1,\infty)$ and that
$R \in L^p(\hat B).$
 Similarly    $ \ell(s_i)^{-1} \of W(s_i) \to R^{-1}$ with $R^{-1}$ in $L^p$ for all $p\in [1,\infty)$ after taking a subsequence :
$   | D (\ell^{ij}(s) \of W(s))|^{1+\de}    =  |\ell^{ ik}(s) \ell^{jl}(s) D \ell_{kl}(s) \of W(s)|^{1 +\de} $
and   hence a subsequence of $\ell^{ij}(s_k) \of W(s_k)$  converges in $L^{\al(\de)}$ to some
$H^{ij}$ in $L^{\al(\de)}$, and  we also have 
\begin{eqnarray*}
\int_{{\hat B}} |\ell_s^{-1}\of W(s)|^p dy    && = \int_{\hat B} |\ell_s^{-1}\of W(s)|^p  \norm{D\Phi(s)}\of W(s)  \norm{DW(s)}   dy\cr
&& = \int_{W(s)(\hat B)} |\ell_s^{-1}|^p   \norm{D\Phi(s)} dx \cr
&&  \leq c(p,\ldots)
\end{eqnarray*}
and so,  $\ell^{ij}(s_{k_r}) \of W(s_{k_r}) \to H^{ij}$ in $L^p(\hat B)$ for all $p\in [1,\infty),$   where $(s_{k_r})_{r \in \N}$ is a subsequence of  
  $(s_k)_{k\in \N}.$  We also have
$  \de^{j}_m =  \ell^{jk}(s_{i_r})\of W(s_{i_r})  \ell_{km}\ (s_{i_r} )  \of W(s_{i_r})  \to H^{jk}R_{km}$ in $L^p(\hat B)$ and hence almost everywhere,  
and hence $H$ is the inverse of $R$ almost everywhere, and after changing the function  $H$ on a set of measure zero, $H$ is the inverse of $R$ everywhere.

Using \eqref{Metriceq} we see 
$$\de^{\ga}_{\be}  =  g^{\ga \al}(y)(s)  D_{\al} W^i(s)(y)   D_{\be} W^j(s)(y)  \ell_{i j }(W(s)(y),s),$$
and writing $W(k) := W(s_{i_k}), g(k):= g(s_{i_k}),...$, we have
\begin{eqnarray*}
0 && = g(k)^{\be \al}(y)  D_{\al} W^i(k)(y)   D_{\be} W^j(k)(y)  \ell(k)_{i j }(W(k)(y))
\cr 
&& \ \ \ - g(l)^{\be \al}(y)  D_{\al} W^i(l)(y)   D_{\be} W^j(l)(y)  \ell(l)_{i j }(W(l)(y))\cr
&& = g(0)^{\be \al}(y) (  D_{\al} W^i(k)(y)   D_{\be} W^j(k)(y) - D_{\al} W^i(l)(y)   D_{\be} W^j(l)(y))R_{ij}(y)\\
&& \ \ \  + \err(k,l)(y)\cr
&& = |DW(k)-DW(l)|^2_{g(0),R}+ \err(k,l),
\end{eqnarray*}
where $\err(k,l) \to 0$ in $L^p(\hat B)$ for all $p\in [1,\infty)$.
Hence, using \eqref{firstD} of Theorem \ref{lpmetricthm} we have
\begin{eqnarray*} 
&& \int_{\hat B} |DW(k) -DW(l)|^{2p}_{\de,\de}  dy \cr
&& \leq c(p,\ldots)  \int_{\hat B}  |DW(k) -DW(l)|^{2p}_{g(0),h}  dy \cr
&& \leq c(p,\ldots)  \int_{\hat B}  (|DW(k) -DW(l)|^{2p}_{g(0),R} )(1 + |h|^{2p}_{R}) dy\cr
&& \leq  c(p,\ldots)  (\int_{\hat B} |DW(k) -DW(l)|^{4p}_{g(0),R}  dy)^{\frac 1 2}\cr
&& =   c(p,\ldots) (\int_{\hat B} |\err(k,l)|^{4p}  dy)^{\frac 1 2}
\end{eqnarray*}
and hence $W(k)$ is Cauchy in $W^{1,p}({\hat B})$ and hence converges. 
Here we used that $\int_{\hat B} |h|^{4p}_R dy$ is bounded which follows from the fact that  $
R^{-1}= H \in L^q$ for all $q \in [1,\infty)$.
Using a similar argument to the one we used for $\Phi,$ we see that $W(s) \to W(0)$ in $W^{1,p}({\hat B})$ 
as $s \downto 0,$ ie. that the convergence  $W(t_i) \to W(0)$ in $W^{1,p}({\hat B})$ holds for all sequences $0<t_i \to 0.$
 

We saw that  $\ell(s) \of W(s)$ converges in $L^p(\hat B)$ for all $p\in [1,\infty)$ as $s \downto 0$. We would like to further show that 
the limit function is $\ell(0) \of W(0)$.
 
Using the change of variable formula for smooth diffeomorphisms, we see for the same coordinates from above,  for any $B_r(y_0) \subsub \hat B $ and any cut off smooth non-negative function $\eta$  with support 
in $B_{r-2\ep}(y_0)$ that
\begin{eqnarray*}
\int_{B_r(y_0)} \ell_s\of W(s) \cdot \eta  dy && = \int_{B_{r-2\ep}(y_0)} \ell_s\of W(s) \cdot \eta \of \Phi(s) \of W(s) \cdot \norm{D\Phi(s)}\of W(s)  \norm{DW(s)} dy \cr
&& = \int_{W(s)(B_{r-2\ep}(y_0))} \ell_s \cdot \eta \of \Phi(s)\cdot \norm{D\Phi(s)}  dx\cr
&& =  \int_{W(\de)(B_{r-\ep/4}(y_0))}   \ell_{0} \cdot   \eta \of \Phi(0)\cdot  \norm{D\Phi(0)}dx + \err(s)  
\end{eqnarray*}
for any $\de \leq s,$ 
  where  $\err(s) \to 0$ in $L^p_{loc}$ as $ s \downto 0$ since 
 $ D\Phi(s) \to D\Phi(0) $  in $L^p_{loc}$ and   $\Phi(s) \to \Phi(0)$ uniformly as $s \downto 0$ and $\ell_s \to \ell_0$ in $L^p_{loc}$ as $s \downto 0$, and 
 $ \eta \of \Phi(s)$ has compact support in $W(s)(B_{r-2\ep}(y_0)) \subseteq W(\de)(B_{r-\ep/4}( y_0))$ for $s,\de$ sufficiently small. 
 $\Phi(0),W(0)$ are    homeomorphisms which are continuous representatives of  $W^{1,p}_{loc}$ functions with $p>n$ and so they both 
  satisfy the {\it  Lusin $N$-Property}  (see  Corollary B  , \cite{MalyMartio}) and hence 
  the  change of area formula is valid for  $\Phi(0)$ and $W(0)$  (see Proposition 1.1 of \cite{Maly}) : 
\begin{eqnarray*}
&&  \int_{W(\de)(B_{r-\ep/4}(y_0))}   \ell_{0}\cdot   \eta \of \Phi(0)\cdot  \norm{D\Phi(0)} dx + err (s) \cr
&&= \int_{\Phi(0)(W(\de)(B_{r-\ep/4}(y_0))}   \ell_{0} \of W(0) \cdot   \eta dy + err(s) \cr
&& \to \int_{B_r}  \ell_{0} \of W(0) \cdot   \eta dy
\end{eqnarray*}
as $s \downto 0$.
As this is true for any continuous $\eta$  and ball $B_r(y_0)$  of this type, 
we see that $ \ell(s) \of W(s) \to R=  \ell_0 \of W(0) $ almost everywhere and in $L^{p}_{loc},$ since $\ell(s) \of W(s) $ converges in $L^{p}_{loc}$
for $s \downto 0,$ for all $p \in [1,\infty):$ $\int  \eta \ell(s) \of W(s) dy  \to \int \eta R dy  = \int \eta \ell(0) \of W(0) dy,$ and hence 
$ \int \eta(R-\ell(0) \of W(0))dy   = 0$ for all non-negative cut off functions of this type, and hence, using the {\it Fundamental lemma of the calculus of variations}, we have ,
$R-\ell(0)\of W(0) = 0$ in $L^1(\hat B)$ and hence $R = \ell(0)\of W(0)$ almost everywhere in $ B_r(y_0)$.
Returning to the last identity in \eqref{Metriceq}, we see that
$g(0)_{\al \be}   = D_{\al} W^i(0)   D_{\be} W^j(0)  \ell_{ij}(0) \of W(0) $  almost everywhere and in the $L^p$ sense.
An almost identical argument shows that $ 
g_{\al \be}(s) \of \Phi(s)$ converges to $g_{\al \be}(0) \of \Phi(0)$
in $L^p,$ as we now explain.
Let $C  = W(0)(B_r(y_0))$ and $\eta $ a cut off function with compact support in $W(0)(B_{r-4\ep}(y_0)).$
Then $\eta = 0$ outside of  $W(s)(B_{r-2\ep}(y_0))$  and $W(0)(B_{r-4\ep}(y_0)) \subseteq W(s)(B_{r-2\ep}(y_0))$ for sufficiently small $s>0$. Hence 
\begin{eqnarray*}
\int_{C}  g(s)\of \Phi(s)  \eta dx && =  \int_{W(0)(B_{r-4\ep}(y_0)) }  g(s)\of \Phi(s)  \eta dx\cr
&&=\int_{W(s)(B_{r-2\ep}(y_0)) }  g(s)\of \Phi(s)  \eta dx\cr
&& = \int_{W(s)(B_{r-2\ep}(y_0))}  g(s) \of \Phi(s) \cdot \eta \of W(s)  \of \Phi(s)  \cdot \norm{DW(s)}\of \Phi(s)  \norm{D\Phi(s)} dx \cr
&& = \int_{B_{r-2\ep}(y_0)}  g(s) \cdot \eta \of W(s) \cdot \norm{DW(s)} dy  \cr
&&  = \int_{B_{r-2\ep}(y_0)}    g(0) \cdot   \eta \of W(0)\cdot  \norm{DW(0)}dy  + \err(s)\cr
&& =  \int_{W(0) (B_{r-2\ep }(y_0))}    g(0) \of \Phi(0)  \cdot  \eta  dx  + \err(s) \cr
&& \to \int_{C}  g(0) \of \Phi(0) \eta  dx \cr
\end{eqnarray*}
as $s \downto 0$ 
where we have once again used that the change of variables formula is valid for $\Phi(0)$ and $W(0)$.
Hence, since $ g(s) \of  \Phi(s) \to  Z$ in $L^p_{loc}$ we have $ g(s) \of  \Phi(s) \to  Z= g(0) \of \Phi(0)$   in  $L^{p}(C)$ as $s \downto 0$ for all $p\in [1,\infty),$  in view of the {\it Fundamental lemma of the calculus of variations}.
Hence $ g(s) \of  \Phi(s) \to g(0) \of \Phi(0)$  in  $L^{p}_{loc}$ as $s \downto 0,$ for all $p\in [1,\infty).$  
Returning to the first  identity in \eqref{Metriceq}, we see that this implies 
\begin{eqnarray*}
\ell_{i j }(0)(x) = D_{i}\Phi^{\al}(0)(x)    D_j\Phi^{\be}(0)(x)  g_{\al \be}(0)( \Phi(0) (x)) 
\end{eqnarray*}
almost everywhere and in the $L^p$ sense.

\end{proof}

\section{Distance  convergence in Sobolev spaces}\label{distanceconvergencesobolev}
\begin{lemma}\label{distanceboundbelow}
In the following $\B^k_r(0)$ is a $k$-dimensional ball of radius $r>0$ in $\R^k$ and middlepoint $0$.   Let $c>1$ and $g: \B^n_2(0)\times [0,1]  \to \R^{n\times n}$ be a family of non-negative
two-tensors, such that 
\begin{eqnarray}
\int_{\B^n_2(0)} |g(t)(z)-g(0)(z)|_{\de}^2 dz \leq c t\\
\frac{1}{c} \de \leq g(t) \leq c \de 
\end{eqnarray}
for all $t\in [0,1),$ where $g(t)$ are smooth for all $t>0$ and $g(0)$ is in $L^2$. 

Let $\si:[-1,1] \to \B^n_2(0)$,  $\si(s) =  se_1$ be a {\it Lebesgue line} with respect to $g(0)$, that is 
the function $\sqrt{g(0)_{11}(\cdot,\cdot )}:[-1,1]\times \B^{n-1}_1(0) \subseteq \B^n_{2}(0) \to \R^+_0$ is measurable, $\sqrt{g(0)_{11}(s,0 )}: \B^{n-1}_1(0) \to \R^+_0$ is measurable,  $\sqrt{g(0)_{11}(s,\cdot )}: \B^{n-1}_1(0) \to \R^+_0$ is measurable for almost all 
$ s \in [-1,1]$, $\sqrt{g(0)_{11}(\cdot,x )}: [-1,1]  \to \R^+_0$ is measurable for almost all 
$x \in \B^{n-1}_1(0)$ and 
\begin{eqnarray*}
\int_{-1}^1 \sqrt{g(0)_{11}(s,0)} ds&& = \lim_{\al \to 0} \frac{1}{ \omega_{n-1} \al^{n-1}}  \int_{-1}^1 \int_{\B^{n-1}_{\al}(0)} \sqrt{g(0)_{11}(s,x) }dx ds  \\
&& =  \lim_{\al \to 0}\frac{1}{ \omega_{n-1} \al^{n-1}}  \int_{T_{\al}(\si)} \sqrt{g(0)_{11} (z)} dz,
\end{eqnarray*}
where $dx$ is $n-1$ dimensional Lebesgue measure, and
$dz$ is  $n$ dimensional Lebesgue measure, $T_{\al}(\si): = \{ \si(s)+\be (0,v)  \ | \  s \in [-1,1], v \in \R^{n-1},  |v|=1, \be \in \R, |\be| \leq \al   \}$ is an $\alpha$ tubular neighbourhood of $\si$, $  \omega_{n-1} \al^{n-1}$ is the (Lebesgue) $n-1$ dimensional volume of $T_{\al}(\si(s)) :=  \{ \si(s)+\be (0,v)  \ | \   v \in \R^{n-1},  |v|=1, \be \in \R, |\be| \leq \al   \}.$ 
Then: For all $\ep >0$ there exists a $t_0 >0$ such that 
 \begin{eqnarray}
L_{g_0}(\si):= \int_{-1}^1 \sqrt{g(0)_{11}(s,0)}  \geq (1-\ep)d(g(t))(\si(-1),\si(1)) 
\end{eqnarray}
for all $t\in (0,t_0)$.

\end{lemma}

\begin{proof}
We calculate
\begin{align}
L_{g_0}(\si) & =  \int_{-1}^1  \sqrt{g(0)_{11}(s,0)} ds \cr
& \geq  \frac{ \int_{ y \in T_{\al}(\si) }   \sqrt{g(0)_{11}(y)} dy}{ \vol(\B^{n-1}_{\al}(0) ) }    -R(\al)  \cr
& =  \frac{ \int_{ y \in T_{\al}(\si) }  ( \sqrt{g(0)_{11}(y)}  -\sqrt{g(t)_{11}(y)}) + \sqrt{g(t)_{11}(y)}  dy}{ \vol(\B^{n-1}_{\al}(0) ) }    
  -R(\al) \cr
  & \geq  \frac{ \int_{ y \in T_{\al}(\si) }  -| \sqrt{g(0)_{11}(y) } -\sqrt{g(t)_{11}(y)} |  + \sqrt{g(t)_{11}(y)}  dy }{ \vol(\B^{n-1}_{\al}(0) ) }      -R(\al) \cr
  & [\mbox{ using } |\sqrt{a} - \sqrt{b}| \leq \sqrt{|a-b|} \mbox{ for } a,b \in \R^+ ] \cr
   & \geq  \frac{ \int_{ y \in T_{\al}(\si) }  -\sqrt{ | g(0)_{11}(y)  - g(t)_{11}(y)| }    + \sqrt{g(t)_{11}(y)} dy } { \vol(\B^{n-1}_{\al}(0) ) }    
  -R(\al) \cr
   & \geq  - \frac{ (\int_{ y \in T_{\al}(\si) }   | g(0)_{11}(y)  - g(t)_{11}(y)|^2 dy)^{\frac 1 4} }{\vol(\B^{n-1}_{\al}(0) )} (2\vol(\B^{n-1}_{\al}(0) ))^{\frac 3 4} \\
   & \ \  \ 
       +  \frac{ \int_{ y \in T_{\al}(\si)} \sqrt{g(t)_{11}(y)} dy } { \vol(\B^{n-1}_{\al}(0) ) }    
  -R(\al) \cr
& \geq \frac{- 2c   t^{\frac 1 4} (\vol(\B^{n-1}_{\al}(0) ))^{\frac 3 4} }{\vol(\B^{n-1}_{\al}(0) )}  +   \frac{\int_{x \in \B_{\al}^{n-1}(0)} \int_{-1}^1 \sqrt{g(t)_{11}(s,x)} ds  dx}{\vol(\B^{n-1}_{\al}(0) )}  - R(\al) \cr
& \geq \frac{-  ct^{\frac 1 4}}{(\vol(\B^{n-1}_{\al}(0) ))^{\frac 1 4}}  +  \frac{\int_{x \in \B_{\al}^{n-1}(0)} L_{g(t)}(\si_x) dx }{\vol(\B^{n-1}_{\al}(0) )} - R(\al)\cr
	& \geq \frac{-ct^{\frac 1 4}}{(\vol(\B^{n-1}_{\al}(0) ))^{\frac 1 4}}  +   \inf_{x \in \B_{\al}^{n-1}(0)}  L_{g(t)}(\si_x) -  R(\al) \cr 
& \geq \frac{-  ct^{\frac 1 4}  }{(\vol(\B^{n-1}_{\al}(0) ))^{\frac 1 4} }  +   \inf \{  d_t( (-1,x), (1,x))  \ | \ x \in \overline{\B_{\al}^{n-1}(0)} \} -  R(\al),  \label{maindistest}
\end{align}
where $R(\al) $ is independent of $t$ and $R(\al)  \to 0$  as  $\al \downto 0,$ and $\si_x:[-a,a] \to \R^n$ is $ \si_x(s) = (s,x)$. Now using the equivalence of  $g(t)$  to $\de$ with the constant $c$ we see that 
\begin{align*}
& \inf \{  d_t( (-1,x), (1,x) ) \ | \ x \in \overline{\B_{\al}^{n-1}(0)} \} \cr
& \geq  d_t( (-1,0),(1,0))  - \sup_{x \in \overline{\B_{\al}^{n-1}(0)}} d_t( (-1,0), (-1,x))  -
\sup_{x \in \overline{\B_{\al}^{n-1}(0)}}d_t( (1,0), (1,x)) \cr
& \geq d_t( (-1,0),(1,0)) -c|\al|
\end{align*}
 independently of $t$, and hence choosing $t= \al^{200n}$ and $\al = \al(\ep)$ small enough, we see that
\begin{align}
L_{g_0}(\si) \geq (1-\ep) d_t( (-1,0),(1,0))
\end{align}
in view of \eqref{maindistest},  and the fact that  $ d_t( (-1,0),(1,0)) \geq \frac{1}{\sqrt{c}}>0$ (since $g(t) \geq  \frac{1}{c} \de$ ).
\end{proof}

\begin{lemma}\label{distboundabove}
For all $\ep>0$ there exists a $\al >0$ such that
if   $g$  is an $L^2$ metric on $\B_1(0)\subseteq \R^n$, the standard ball of radius one and middlepoint $0$ in $\R^n,$ with 
$\int_{\B_1(0)} |g - \de|^2 \leq \al,$ then there exits an $x \in \B^{n-1}_{\ep}(0)$ such that 
$\sqrt{ g_{11} }(\cdot,x): [-\frac 1 2,\frac 1 2] \to \R^n$ is measurable, 
$\si:[-\frac 1 2,\frac 1 2] \to \R^n,$ $\si(t) = x+e_1t, $ is a Lebesgue line between $(-\frac 1 2,x)$ and $(\frac 1 2,x)$ and 
\begin{eqnarray*}
\int_{-\frac 1 2}^{\frac 1 2} \sqrt{g_{11}(s,x)}ds  \leq 1+ \ep = d_{\de}( ( -\frac 1 2,x),(\frac 1 2,x))+ \ep
\end{eqnarray*}
\end{lemma}
\begin{proof} 
Fubini's Theorem tells us that 
the function $f: \B^{n-1}_{\frac 1 4}(0) \to \R^+_0,$ 
$f(x):= \int_{-\frac 1 2}^{\frac 1 2} \sqrt{g_{11}(s,x)} ds$  is   well defined  for almost all $x \in \B^{n-1}_{\frac 1 4}(0)$ and defines an  $L^1$ function, also denoted by $f$,  and the function  $\hat f: [-1,1] \to \R^+_0$ 
$\hat f(s):= \int_{\B^{n-1}_{\frac 1 4}(0)}  \sqrt{g_{11}(s,x)} dx$
 is   well defined     for almost all $s \in [-\frac 1 2, \frac 1 2]$  and defines an  $L^1$ function, and
$\int_{\B^{n-1}_{\frac 1 4}(0)} f(x) dx  =  \int_{-\frac 1 2}^{\frac 1 2} \int_{\B^{n-1}_{\frac 1 4}(0)}  \sqrt{g_{11}(s,x)} dx ds =    \int_{\B^{n-1}_{\frac 1 4}(0)}  \int_{-\frac 1 2}^{\frac 1 2} \sqrt{g_{11}(s,x)} ds dx 
=  \int_{\B^{n-1}_{\frac 1 4}(0) \times  [-\frac 1 2,\frac 1 2]    }\sqrt{g_{11} (z)} dz.
$
This also implies that  almost every $x \in \B^{n-1}_{\frac 1 4}(0)$ is a Lebesgue point of $f$, that is 
$$
\frac{ \int_{\B^{n-1}_{r}(x) } |f(y)-f(x)| dy}{\omega_{n-1} r^{n-1}}   \to 0
$$
as $r\downto 0$, for almost every $x \in \B^{n-1}_{\frac 1 4}(0)$ (see Corollary 1 of Section 1.7 of \cite{EvansGariepy}), and as a consequence
$$
\frac{ \int_{\B^{n-1}_{r}(x) } f(y)dy}{\omega_{n-1} r^{n-1}}   \to f(x)
$$
as $r\downto 0$, for almost every $x \in \B^{n-1}_{\frac 1 4}(0)$

That is, almost every curve $ v_x:[-\frac 1 2, \frac 1 2] \to \B_1(0),$ $ v_x(s):= (s,x)$ for $x \in \B^{n-1}_{\frac 1 4}(0)$  is a  parametrised  Lebesgue line.

We wish to estimate the measure $m$ of the set $Z  \subseteq \B^{n-1}_{\frac 1 4}(0) $ of  $ x \in \B^{n-1}_{\frac 1 4}(0)$ such that 
$\int_{-\frac 1 2}^{\frac 1 2} \sqrt{g_{11}(s,x)}ds  $ is well defined and  for which 
$\int_{-\frac 1 2}^{\frac 1 2} \sqrt{g_{11}(s,x)}ds  \geq 1 +   \hat \al $ where $\hat \al := (\al)^{\frac 1 8} $ ($\to  0$ as $\al \to 0$). 
We will see that $m \leq  a(n) \al^{\frac 1 8}.$ 
Using $|\sqrt{a}-\sqrt{b} |\leq \sqrt{|a-b|}$ for $a,b \in \R^+,$ we see that

\begin{align*}
\al  & \geq  \int_{\B^{n-1}_{\frac 1 4}(0)}  \int_{-\frac 1 2}^{\frac 1 2} |g_{11}(s,x) -1|^2 ds dx 
 \cr
 & = \int_{\B^{n-1}_{\frac 1 4 }(0)}  \int_{-\frac 1 2}^{\frac 1 2} \Big(\sqrt{|g_{11}(s,x) -1|}\Big)^4 ds dx \cr
& \geq \int_{\B^{n-1}_{\frac 1 4 }(0)}  \int_{-\frac 1 2}^{\frac 1 2} |\sqrt{g_{11}(s,x)} -1|^4 ds dx \cr
& \geq \frac{(\int_{\B^{n-1}_{\frac 1 4 }(0)}  \int_{-\frac 1 2}^{\frac 1 2} |\sqrt{g_{11}(s,x)} -1| ds dx)^4 }{ (\vol(\B^{n-1}_{\frac 1 4 }(0) \times [-\frac 1 2, \frac 1 2] ))^3} \cr
& = \frac{1}{c(n)} (\int_{\B^{n-1}_{\frac 1 4 }(0)}  \int_{-\frac 1 2}^{\frac 1 2} |\sqrt{g_{11}(s,x)} -1| ds dx)^4  \cr
& \geq  \frac{1}{c(n)} (\int_{\B^{n-1}_{\frac 1 4 }(0)}  |\int_{-\frac 1 2}^{\frac 1 2} (\sqrt{g_{11}(s,x)} -1)  ds |dx)^4  \cr
 & \geq  \frac{1}{c(n)} (\int_{Z}  |\int_{-\frac 1 2}^{\frac 1 2} (\sqrt{g_{11}(s,x)} -1)  ds |dx)^4  \cr
 & =    \frac{1}{c(n)} (\int_{Z}  |\int_{-\frac 1 2}^{\frac 1 2} \sqrt{g_{11}(s,x)}ds  -1 |dx)^4  \cr
  & =   \frac{1}{c(n)} (\int_{Z}  \int_{-\frac 1 2}^{\frac 1 2} \sqrt{g_{11}(s,x)}ds  -1 dx)^4  \cr
 &\geq \frac{1}{c(n)}  ( \int_Z \hat \al dx)^4 \cr
 & \geq \frac{1}{c(n)}  m^4  (\hat \al)^4
\end{align*}
which implies $m^4 \leq \frac{c(n) \al }{\hat \al^4} = c(n) \al^{\frac 1 2},$ that is, 
$m= \curlL^{n-1}( Z ) \leq  (c(n))^{\frac 1 4}  \al^{\frac 1 8}\leq \al^{\frac 1 {20}}$ for $\al \leq \frac{1}{(c(n))^{4} }.$
For $\ep>0$ given, we now choose  $\al= \ep^{100n},$
so that $m\leq \ep^{5n}$.
But then 
$\curlL^{n-1}( Z^c \cap \B^{n-1}_{\ep}(0)) >0$  
  : Otherwise, $\curlL^{n-1}( Z^c \cap \B^{n-1}_{\ep}(0))=0$ 
 that is $ \curlL^{n-1}(Z \cap \B^{n-1}_{\ep}(0)) = \omega_{n-1} \ep^n $ and as a consequence
 $\ep^{5n} \geq m = \curlL^{n-1}(Z) \geq \curlL^{n-1}(Z \cap \B^{n-1}_{\ep}(0) ) = \omega_{n-1} \ep^n$ which is a contradiction. 
 
Using this, with the fact that for almost every $x \in \B^{n-1}_{\ep}(0)$ the curve $ v_x:[-\frac 1 2, \frac 1 2] \to \B_1(0)$ $ v_x(s):= (s,x)$ is  a parametrised  Lebesgue line, we see that it is possible to choose an $x \in \B^{n-1}_{ \ep}(0)$ such that $\int_{-\frac 1 2}^{\frac 1 2} \sqrt{g(0)_{11}(s,x)} ds \leq 1 +  \hat \al = 1 + \ep^{\frac{100n}{8}} \leq 1 +\ep$
and so that $v_x$ is a parametrised Lebesgue line, as claimed.
\end{proof}

\section{Uniqueness}\label{uniquenesssection}

\begin{lemma}[$L^2$-Lemma]\label{L2Lemma}  (cf. Lemma 6.1 in proof of \cite{DeruelleSchulzeSimon}).
 Let $M$ be $n$-dimensional and $g_{1},g_2$ be two smooth solutions on $M \times [0,T]$  to the $h$-Ricci-DeTurck flow,
and  let $\ell:= g_1- g_2,$  $\tilde{\ell}^{ab}:= \frac 12 \left(g_1^{ab}+g_2^{ab}\right)$ and $\hat{\ell}^{ab}:= \frac 12 \left(g_1^{ab}-g_2^{ab}\right).$
Then the quantity $|\ell|_{h}^2$ satisfies the evolution equation:

\begin{equation}
\begin{split}
\partt |\ell|^2 &= \tilde{\ell}^{ab}\gradh_a \gradh_b |\ell|^2 - 2|\gradh \ell|^2_{\ti{\ell},h}  + \ell   * \hat{\ell} *\gradh^2 (g_1+g_2)\\
&\ \ \ +  * \ell*\hat{\ell}* g_1^{-1}   *\gradh g_1*\gradh g_1 + \ell*g_2^{-1}* \hat{\ell}* \gradh g_1*\gradh g_1 \\
&\ \ \ + \ell*g_2^{-1}* g_2^{-1} * \gradh \ell *\gradh g_1 + \ell* g_2^{-1}* g_2^{-1} * \gradh g_2 *\gradh \ell \\
& \ \ \ 
+  \ell * \hat{\ell} * g_1 *  \Rm(h) +   \ell * \ell * (g_2)^{-1} *Rm(h) 
\,,
\end{split}\label{evoldiff}
\end{equation}
where $T*S$ refers to contractions of the tensors $T$ and $S$ involving $h^{-1}$
and $ |Z|^2_{g,h} =    g^{ij}h^{ks}h^{vr} Z_{ikv} Z_{jsr}$ for a zero-three tensor $Z$.
\end{lemma}
\begin{proof}
The formula was proved in \cite{DeruelleSchulzeSimon} for the case that 
$h =\de$ is the standard metric on a euclidean ball $B_1(0)$, and hence the curvature $\Rm(h) = 0$.
We carry out a similar argument to the one given there, explaining 
 why the term arising from the curvature $Rm(h)$ in the evolution equation of $|\ell|^2$ in this setting can be written as 
$\ell * \hat{\ell} * g_1 *  \Rm(h) +   \ell * \ell * (g_2)^{-1} *Rm(h).$

\begin{equation*}
\begin{split}
\partt \ell &=  g_1^{ab} \gradh_a \gradh_b g_1 + g_1^{-1} * g_1^{-1}*\gradh g_1 *\gradh g_1 +\\
& \ \ \ 
-(g_1)^{kl}(g_1)_{ip}h^{pq}R_{jkql}(h)
  -(g_1)^{kl}(g_1)_{jp}h^{pq}R_{ikql}(h) \\
  & \ \ \ 
 - g_2^{ab} \gradh_a \gradh_b g_2- g_2^{-1} * g_2^{-1}*\gradh g_2 *\gradh g_2 \\
 & \ \ \ +(g_2)^{kl}(g_2)_{ip}h^{pq}R_{jkql}(h)
  +(g_2)^{kl}(g_2)_{jp}h^{pq}R_{ikql}(h) \\
&= \frac 12 \left(g_1^{ab}+g_2^{ab}\right)\gradh_a \gradh_b \ell + \frac 12 \left(g_1^{ab}-g_2^{ab}\right)\gradh_a \gradh_b (g_1+g_2)\\
&\ \ \ + (g_1^{-1} - g_2^{-1})* g_1^{-1} * \gradh g_1*\gradh g_1 + g_2^{-1}* (g_1^{-1}-g_2^{-1}) * \gradh g_1*\gradh g_1 \\
&\ \ \ + g_2^{-1}* g_2^{-1} * \gradh \ell *\gradh g_1 +  g_2^{-1}* g_2^{-1} * \gradh g_2 *\gradh \ell, \\
& \ \ \ -{\hat \ell}^{kl}h^{pq}\ (g_1)_{ip}R_{ikql}(h)  - {\hat \ell}^{kl}h^{pq} \ (g_1)_{jp}R_{jkql}(h)  \\
& \ \ \ -(g_2)^{kl}\ell_{ip}h^{pq}R_{jkql}(h)
  +(g_2)^{kl}\ell_{jp}h^{pq}R_{ikql}(h)\\
\end{split}
\end{equation*}
which we can write as 
\begin{equation*}
\begin{split}
\partt \ell &= \tilde{\ell}^{ab}\ \gradh_a \gradh_b \ell + \hat{\ell}^{ab}\gradh_a \gradh_b (g_1+g_2)\\
&\ \ \ + \hat{\ell}* g_1^{-1} * \gradh g_1*\gradh g_1 + g_2^{-1}* \hat{\ell}* \gradh g_1*\gradh g_1 \\
&\ \ \ + g_2^{-1}* g_2^{-1} * \gradh \ell *\gradh g_1 +  g_2^{-1}* g_2^{-1} * \gradh g_2 *\gradh \ell\\
& +\hat{\ell} * g_1 *  \Rm(h) +  \ell * (g_2)^{-1}  *Rm(h) .
\end{split} 
\end{equation*}
The formula  now  follows from this equality, combined with the  facts that
$\partt |\ell|_h^2 = 2( \ell, \partt \ell )_h$ and $2(\ell, {\ti \ell}^{ab} \ \gradh_{a} \gradh_b \ell)_h
=   {\ti \ell}^{ab} \ \gradh_{a} \gradh_b |\ell|^2_h - 2|\gradh \ell|^2_{\ti \ell,h}.$
\end{proof}
Using the previous evolution equation for the difference between two solutions of the $h$-Ricci-DeTurck flow we are now able to show that the solution constructed in Theorem \ref{main1} is unique among all solutions satisfying $(a_t)$, $(b_t)$ and $(c_t)$ with $\varepsilon$ sufficiently small. The proof below slightly resembles the argument used  by Struwe  to prove a uniqeuness result for the harmonic map flow in two dimensions, see  the argument given in the proof of uniqueness in the proof of Theorem 6.6, Chapter III in  \cite{Struwe2}.
\begin{thm}\label{uniqueness}
Let $M$ be four dimensional and  $g(t)$, $0\leq t\leq S$, be a solution of the $h$-Ricci-DeTurck flow with initial condition 
$g_0\in W_{loc}^{2,2}\cap W_{loc}^{1,\infty}$ satisfying $\eqref{aaa}$ and $\eqref{bbb}$. Assume additionally that $g(t)$  satisfies the estimates $\eqref{aaa_t}$, $\eqref{bbb_t}$ $\eqref{ccc_t}$ and $\eqref{ddd_t}$  from Theorem \ref{main2} for all $0\leq t\leq S$. Then there exists a time $T_{\max}=T_{\max}(n,a)\in(0,S)$ so that the solution is unique for all $0\leq t\leq T_{\max}$.
\end{thm}
\begin{proof}
We let $g_1$, $g_2$ be two solutions and as above we let $l=g_1-g_2$.
Next we multiply \eqref{evoldiff} with $\eta^4$, where $\eta$ is a cut-off function which is equal to one on $B_{\frac 1 2}(x)$ and zero outside of $B_{\frac  3 4}(x)$, and integrating by parts, with respect to $\grad$, and using Youngs and H\"older inequality,  we obtain the estimate
\begin{align}
\partial_t \int_M \eta^4 |l|^2+2\int_M \eta^4|\gradh l|^2_{\ti \ell ,h} \leq& C\int_M \eta^3(\eta |l| |\gradh \tilde{l}||\gradh l|+ |\gradh \eta||\tilde{l}| |l||\gradh l|)+\int_M \eta^4|\gradh l|^2_{\ti \ell,h} \nonumber \\
&+C(\int_M \eta^4 (|l|^4+|\hat{l}|^4))^{\frac12} (\int_{B_1(x)}  |\gradh ^2(g_1+g_2)|^2)^{\frac12} \nonumber \\
&+C(\int_M \eta^4 (|l|^4+|\hat{l}|^4))^{\frac12} (\int_M \eta^4 |\gradh (g_1+g_2)|^4)^{\frac12} \nonumber \\
&+C\int_M \eta^4 (|l|^2+|\hat{l}|^2),\label{unique1}
\end{align}
where the term $2\int_M \eta^4|\gradh l|^2_{\ti \ell ,h} $ is the second term , up to a change of sign,   appearing on the right hand side of  equation \eqref{evoldiff}. 
Using the Sobolev embedding Theorem, see Theorem \ref{balllemma},  and the assumption $(b_t)$ it follows that
\[
\int_{B_1(x)}  |\gradh ^2(g_1+g_2)|^2+(\int_M \eta^4 |\gradh (g_1+g_2)|^4)^{\frac12} \leq C\varepsilon.
\]
Using  the estimate $|\hat{l}|\leq C|l|$ and again the Sobolev inequality we  conclude
\begin{align}
(\int_M \eta^4 (|\hat l|^4 +  |l|^4))^{\frac12}\leq C \int_{B_1(x)}( |\gradh l|^2+ |l|^2)\label{uniqueinbet},
\end{align}
and hence
\begin{align}\label{start1}
\partial_t \int_M \eta^4 |l|^2+ \int_M \eta^4|\gradh l|^2_{\ti \ell ,h} \leq& C\int_M (\eta^4 |l| |\gradh \tilde{l}||\gradh l|+ |\gradh \eta|\eta^3 |\tilde{l}| |l||\gradh l|)  \nonumber \\
&+C \sqrt{\ep} \int_{B_1(x)}( |\gradh l|^2+ |l|^2) +C\int_M \eta^4 (|l|^2+|\hat{l}|^2). 
\end{align}
We will estimate the first two terms appearing on the right hand side of \eqref{start1}. In preparation thereof, we first note that 
from $(b_t)$ we also have 
\[
(\int_M \eta^4 (|\gradh l|^4+|\gradh \tilde{l}|^4))^{\frac12} \leq C\varepsilon.
\]
The first term $\int_M \eta^4(\eta |l| |\gradh \tilde{l}||\gradh l|) $ on the right hand side of
\eqref{start1} we estimate as follows, 
\begin{eqnarray}\label{start2}
\int_M \eta^4 |l| |\gradh \tilde{l}||\gradh l| && \leq \frac{1}{4}\int_M \eta^4 |\gradh l|^2_{\ti \ell ,h} 
+ C \int_M \eta^4 |l|^2  |\gradh \tilde{l}|^2 \\
&& \leq \frac{1}{4}\int_M \eta^4 |\gradh l|^2_{\ti \ell ,h} 
+ C (\int_M \eta^4 |l|^4)^{\frac 1 2}( \int_M   \eta^4|\gradh \tilde{l}|^4)^{\frac 1 2} \cr
&& \leq \frac{1}{4}\int_M \eta^4 |\gradh l|^2_{\ti \ell ,h}  +C (\int_{B_1(x)} |\grad l|^2 + |l|^2) C\epsilon.\nonumber
\end{eqnarray}
The second  term $\int_M |\gradh \eta|\eta^3 |\tilde{l}| |l||\gradh l|)$ of \eqref{start1}
is estimated as follows
\begin{eqnarray}\label{start3}
\int_M |\gradh \eta|\eta^3 |\tilde{l}| |l||\gradh l|) && \leq  (\int_M \eta^4  |\tilde{l}|^2 |l|^2)^{\frac 1 2} ) 
(\int_M |\grad \eta| \eta^2 |\gradh l|^2)^{\frac 1 2} \\
&& \leq C (\int_M  \eta^4(|\tilde{l}|^4 +  |l|^4) )^{\frac 1 2} (\int_M \eta^4|\gradh l|^4)^{\frac 1 4}\cr
&& \leq C (\int_{B_1(x)} |l|^2 + |\grad l|^2) (\ep)^{\frac 1 4}.\nonumber
\end{eqnarray} 
Using \eqref{start2} and \eqref{start3}, we can estimate the left  hand side of \eqref{start1} by
\begin{eqnarray}
\partial_t \int_M \eta^4 |l|^2+ \int_M \eta^4|\gradh l|^2_{\ti \ell ,h} \leq &&
C (\int_M |l|^2 + |\grad l|^2) (\ep)^{\frac 1 4} + C\int_M\eta^4 |l|^2
\end{eqnarray}

and hence
\[
\partial_t \int_M \eta^4 |l|^2+\frac{1} {c a} \int_M \eta^4|\gradh l|^2 \leq C\int_{B_1(x)} |l|^2+C\sqrt{\varepsilon}\int_{B_1(x)}|\gradh l|^2.
\]
After integrating in time we obtain for every $x\in M$
\begin{align}
\int_{B_{\frac12}(x)} |l|^2(t)+\frac{1} {c a}\int_0^t\int_{B_{\frac12}(x)} |\gradh l|^2\, ds \leq C \int_0^t \int_{B_1(x)} |l|^2 (s)\, ds+C(\varepsilon)^{\frac 1 4} \int_0^t\int_{B_1(x)} |\gradh l|^2\, ds .\label{gronwall}
\end{align}
Next we let $1>\sigma>0$   be arbitrary and we conclude from Corollary \ref{L2continuityCor} that 
\[
\sup_{x\in M} \int_{B_1(x)}|l|^2(t) <\sigma 
\]
for every $0\leq t \leq C\sigma$ where $C$ is a constant only depending on $n$ and $a$. In the following we let $T_{\max}$ be the smallest time so that
\begin{align*}
\sup_{x\in M}\int_{B_1(x)} |l|^2(T_{\max})&= \sigma.
\end{align*}

For any $x\in M$ we can cover the ball $B_1(x)$ by finitely many balls $B_{\frac12}(x_i)$, $1\le i\le N=N(h)$ (see Appendix B). We conclude from \eqref{gronwall} that for $t\leq T_{\max}$
\begin{align*}
\int_{B_1(x)} |l|^2(t)+\frac{1} {c a} \int_0^t\int_{B_{1}(x)} |\gradh l|^2\leq&  N  \sup_i (C \int_0^t\int_{B_{\frac12}(x_i)} |l|^2(t) +C \sqrt{\ep} \int_0^t\int_{B_{\frac12}(x_i)} |\gradh l|^2) \\
\leq& CNt\sigma +CN\ep^{\frac 1 4}  \sup_i \int_0^t\int_{B_{1}(x_i)} |\gradh l|^2
\end{align*}
and hence
\begin{align*}
&\sup_{x\in M} (\int_{B_1(x)} |l|^2(t) +\frac{1} {c a} \int_0^t\int_{B_{1}(x)} |\gradh l|^2)\\
&\leq CNt\sigma +CN\ep^{\frac 1 4}  \sup_{x\in M} \int_0^t\int_{B_{1}(x)} |\gradh l|^2\cr
& \leq CNt\sigma +  \frac{1}{2}\sup_{x\in M} (\int_{B_1(x)} |l|^2(t) +\frac{1} {c a} \int_0^t\int_{B_{1}(x)} |\gradh l|^2)
\end{align*}
if  $\varepsilon>0$ is  sufficiently small. Hence
\begin{align*}
&\sup_{x\in M} (\int_{B_1(x)} |l|^2(t) +\frac{1} {c a} \int_0^t\int_{B_{1}(x)} |\gradh l|^2) \leq 2CNt\sigma  
 \leq \frac{\si}{2}
\end{align*}
for all $0\le t \le \frac{1}{4CN}$ which implies that $T_{\max}\geq \frac{1}{4CN}$ and since $\sigma>0$ was arbitrary this finishes the proof of the theorem. 
\end{proof}

\section{An application}\label{anapplication}
Here we present an application   for  $W^{2,2}\cap L^{\infty}$ metrics on four dimensional manifolds in the setting that scalar curvature is weakly bounded from below. 
For the case that the metric is  $C^0$ we 
refer to the paper of \cite{B-G} for related results.

\begin{defn}\label{weakscalarlater}
Let $M$ be a four dimensional smooth closed manifold and
$g$ be a $W^{2,2} \cap L^{\infty}$ be  Riemannian metric (positive definite everywhere), such that $g, g^{-1} 
\in  L^{\infty}$  and let $k\in \R.$ 
Locally the scalar curvature may be written  
\begin{eqnarray*}
\Sc(g) && = g^{jk} (\partial_i \Gamma(g)_{jk}^{i} - \partial_j \Gamma_{ik}^{i} \cr
 && \ \ + \Gamma_{ip}^{i}\Gamma_{jk}^p - \Gamma_{jp}^{i}\Gamma_{ik}^{p})
 \end{eqnarray*}
where   $\Gamma(g)_{ij}^m = \frac{1}{2} g^{mk} ( \partial_i g_{jk} + \partial_j g_{ik} - \partial_k g_{ij})$, and hence  
$\Sc(g)$ is well defined in the $L^2$ sense  for a $W^{2,2}$ Riemannian metric.  Let $k\in \R.$ We say the  scalar curvature $\Sc(g)$ is  weakly  bounded from below by $k$, $\Sc(g) \geq k$, if this is true almost everywhere, for all local smooth coordinates.
\end{defn}

\begin{thm}\label{sctheorem}
Let $(M,h)$ be four dimensional closed and satisfy \eqref{hassumptionsscaled}.
Assume that $(M,g_0)$ is a $W^{2,2}$ metric such that
$\frac{1}{a} h \leq g_0\leq a h$  for some $ \infty>a>1$ and $\Sc(g_0) \geq k$ in the weak sense of Definition \ref{weakscalarlater}.  
Then the solution  $g(t)_{t\in (0,T)}$ to Ricci DeTurck flow  respectively  $\ell(t)_{t\in (0,T)}$ to Ricci Flow constructed in Theorem \ref{mainthm}, with initial value $g(0)=g_0$,  has $\Sc(g(t)) \geq k$ and $\Sc(\ell(t)) \geq k$ for all $t \in (0,T)$.
\end{thm}
\begin{proof}
The solution $g(t)$ to Ricci DeTurck flow constructed in the main theorem is smooth for all $t>0$ and satisfies
$g(t) \to g_0$ in the $W^{2,2}$ sense and 
$\frac{1}{400 a} h \leq g(t)\leq 400 a h$  for all $t \in (0,T)$.
Hence $\Sc(g(t)) \to \Sc(g_0)$ as $t\downto 0$  in the $L^2_{loc}$ sense, and in the pointwise sense  almost everywhere,   where $\Sc(g_0)$ is the $L^2$ quantity defined above ( convergence of a sequence of functions in  the $L^2_{loc}$ sense to an $L^2_{loc}$ function implies convergence of the sequence almost everywhere ).
This means $ (\Sc(g(t)) + k)_{-} \to 0$ in the $L^2$ sense as $t \downto 0$,
and hence  $\varphi(t):= \int_M (\Sc(g(t)) + k)_{-} ^2 dg(t)  =  \int_M (\Sc(\ell(t)) + k)_{-} ^2 d\ell(t)  \to 0$ as $t \downto 0.$

The Integrand 
 $V(t):= (\Sc(\ell(t)) + k)_{-} ^2$ is differentiable  in space and time for all $t>0$ and this yields that  $\varphi$ is differentiable in time for all $t>0$. The derivative of $V$ is zero for all $(x,t) \in M \times (0,T)$ with 
  $ \Sc(\ell)(x,t) +k \geq 0$.

 By Sard's theorem (see Section 2 of \cite{Milnor}), we know, for almost all $k$,    
that the sets  $ \{ x \in M \ | \ \Sc(x,g(t)) + k <0\}$ have  smooth boundary for almost every $t>0$: Sard applied once to $\Sc$ yields that $W_k:= \{ (x,t) \in M \times (0,T) \ | \ \Sc(x,t) = -k \}$ is smooth for almost all $k\in \R$ and then Sard applied to $\Psi_k: W_k \to \R,$ $\Psi_k(x,t) = t$ (for such $k$) yields that
$ \{ x \in M   \ | \ \Sc(x,t) = -k \}$ is smooth for almost all $t\in (0,T)$.
Let $Z\subseteq \R$ denote the set of such $k \in \R$.
For such $k \in Z$ we define 
$U_{k}(t)  := \{ x \in M \ | \ \Sc(x,g(t)) + k <0\}$ if  $t\in (0,T)$  is a time such that     $\{ x \in M \ | \ \Sc(x,g(t)) + k <0\}$ has smooth boundary, and  we define 
$U_{k}(t): =\emptyset$ for all other $t\in (0,T)$. 
 Using the Fundamental Theorem of Calculus    for   $0<t_1<t_2$ we compute
\begin{eqnarray*}
&& \psi(t_2)  - \psi(t_1) \\
&& = e^{-kt_2} \phi(t_2)  - e^{-kt_1}\phi(t_1)\\
&& = \int_{t_1}^{t_2} \frac{d}{d\tau} \int_M e^{-k\tau} V(\tau) d\ell(\tau) d\tau \\
&& = \int_{t_1}^{t_2} e^{-k\tau}\int_M (\frac{d}{d\tau } V(\tau) - \Sc(\tau)V(\tau) -k V(\tau) ) d\ell(\tau) d\tau \\
&& = \int_{t_1}^{t_2}  \int_{U_{k}(\tau)}e^{-k\tau}( 2\frac{\partial}{ \partial \tau}(-\Sc(\tau) -k ) (-\Sc(\tau)-k)  -  (\Sc(\tau) + k) ^2\Sc(\tau) ) d\ell(\tau) d\tau \\
&& - \int_{t_1}^{t_2}  \int_{U_{k}(\tau)} ke^{-k\tau}V(\tau)  d\ell(\tau) d\tau \\
&& = \int_{t_1}^{t_2}  \int_{U_{k}(\tau)}e^{-k\tau} (2\lap_{\ell(\tau)} (\Sc(\tau) +k)) (\Sc(\tau)+k) +  4(\Sc(\tau)+k)|\Rc(\tau)|^2  \\
&& 
-  \int_{t_1}^{t_2}  \int_{U_{k}}  e^{-k\tau}(\Sc(\tau) + k)^3 d\ell(\tau) d\tau 
\cr
&&  \ \ + k \int_{t_1}^{t_2}  \int_{U_{k}(\tau)}e^{-k\tau}  (\Sc(\tau) + k) ^2 
 -k  \int_{t_1}^{t_2} \int_{U_{k}(\tau)}  e^{-k\tau}V(\tau) d\ell(\tau) d\tau \\
&&\leq \int_{t_1}^{t_2}  \int_{U_{k}(\tau)}e^{-k\tau} (2\lap_{\ell(\tau)} (\Sc(\tau) +k)) (\Sc(\tau)+k)  
-  \int_{t_1}^{t_2}  \int_{U_{k}(\tau)}  e^{-k\tau}(\Sc(\tau) + k)^3 d\ell(\tau) d\tau  \\
&&\mbox{ [ since } (\Sc(\tau) +k)(\tau )  < 0 \mbox{ on }  U_{k}(\tau) \mbox{ and } V= (\Sc(\tau) + k) ^2 \mbox {]} \\ 
&& = \int_{t_1}^{t_2}  e^{-k\tau}  \int_{U_{k}(\tau)} - 2|\nabla (\Sc(\ell(\tau) +k )|^2 d\ell(\tau) d\tau \\
&&   \ \ + \int_{t_1}^{t_2} e^{-k\tau} \int_{\boundary U_k(\tau) } (\Sc(\tau) +k)\ell(\tau)(  \nu(\ell(\tau)), \nabla (\Sc(\tau) +k) ) d\si(\tau) d\tau 
\\ 
&&  +  \int_{t_1}^{t_2} \int_{U_{k}(\tau)}  e^{-k\tau}(|(\Sc(\tau) + k)_{-}|^3   d\ell(\tau) d\tau \cr 
&& = 
\int_{t_1}^{t_2}  \int_{U_{k}(\tau)} - 2e^{-k\tau}|\nabla (\Sc(\tau) +k )_{-}|^2  d\ell(\tau) d\tau \\
&&  +  \int_{t_1}^{t_2} \int_{U_{k}(\tau)} e^{-k\tau}( |(\Sc(\tau)) + k)_{-}|^3    ) d\ell(\tau) d\tau \cr
&& = \int_{t_1}^{t_2} -2e^{-k\tau}  \int_{M} |\nabla (\Sc(\tau) +k )_{-}|^2 d\ell(\tau) d\tau \\
&&  +  \int_{t_1}^{t_2} e^{-k\tau}  \int_{M}  |(\Sc(\tau) + k)_{-}|^3\\
&& \leq -A(M,h) \int_{t_1}^{t_2} e^{-k\tau}  (\int_{M} | (\Sc(\tau) +k )_{-}|^4 d\ell(\tau) )^{\frac 1 2} d\tau  \\
&& \ \ + \int_{t_1}^{t_2} e^{-k\tau}  (\int_{M}   |(\Sc(\tau) + k)_{-}|^4 d\ell(\tau))^{\frac 1 2}   (\int_{M}  |(\Sc(\tau) + k)_{-}|^2 d\ell(\tau))^{\frac 1 2}  ) d\tau \\
&& \leq \int_{t_1}^{t_2}  e^{-k\tau}  (-A(M)  +  A(M)/2 )   (\int_{M}  |(\Sc(\tau)) + k)_{-}|^4 d\ell(\tau))^{\frac 1 2}  d\tau \cr
&&[ \mbox{ for sufficiently small } t_2]\\ 
&& \leq  0,
\end{eqnarray*}
where we have used the Sobolev inequality and $A(M)$ is the Sobolev constant, 
and we used that $\int_{M}  |(\Sc(\tau) + k)_{-}|^2 d\ell(\tau)
\leq A(M)/2 $ for $\tau \leq t_2$ and $t_2$ sufficiently small, since 
$   \int_{M}  |(\Sc(t) + k)_{-}|^2 d\ell(t) \to 0 = \int_{M}  |(\Sc(t) + k)_{-}|^2 dg(t) \to 0$ as $t\downto 0$.
Hence, since $\psi(0)=0$, $\psi(t)=0$ for all $t\in [0,T)$. That is $\Sc(\ell(t)) \geq k$ for all $t\in (0,T)$ in the smooth sense.
$\Sc(g(t)) \geq k$ for all $t\in (0,T)$ follows from the fact that $(M,\ell(t))$ and $(M,g(t))$ are isometric to one another. 
For general $k\in \R$ we can take a sequence $(k_i)_{i\in \N}$ with $k_i \to k$ and $k_i\in Z.$

\end{proof}

\begin{rmk}
From this theorem we see that   for a metric  $g_0 \in L^{\infty}\cap W^{2,2}(M^4)$ with  $\frac{1}{a} h \leq g_{0} \leq a h$    for some positive constant $a>0$ : $g_0$ has scalar curvature $\geq k$ in the weak sense of Definition \ref{weakscalarlater} $\iff$ 
there exists a sequence of smooth Riemannian metrics
$ g_{i,0}$ 
 with $\frac{1}{b} h \leq g_{i,0} \leq b h$   for some $1<b<\infty$ 
and 
 $\Sc(g_{i,0}) \geq k $  and  $g_{i,0} \to g_0 \in   W^{2,2}(M^4)$ 
$\iff$ the Ricci DeTurck flow of $g_0$ constructed in Theorem \ref{main3} has $\Sc(g(t)) \geq k$ for all $t\in (0,T)$.
 In particular, we do not need to change the constant form $k$ to $k-\frac{1}{i}$ after the first implication $\implies$.
 
\end{rmk}

\appendix
\section{Short time existence of smooth bounded data}\label{shortapp}

We present here a standard existence result for Ricci-DeTurck flows with smooth bounded initial data, based on the method of Shi \cite{Shi}. 

\begin{thm}\label{standardexist}
Let $(M,h)$ be $n$-dimensional and satisfy \eqref{hassumptionsscaled}. We assume there are constants $1<a< \infty$ and $0<c_j < \infty$ for all $j\in \N$, 
 and $g_0$ is  a smooth metric on $M$ satisfying
 \begin{align*}
 & \frac{1}{a}h \leq g_0 \leq ah \\
 & \sup_M |\grad^j g_0| \leq c_j < \infty
 \end{align*}
 
Then there exists  a smooth solution  $(M,g(t))_{t\in [0,\hat T ]}$ to \eqref{Meq} for some $\hat T >0$, and
constants $b_j(g_0,h,S)< \infty$ for all $S \leq \hat T$ such that
$\sup_M |\grad^j g(\cdot,t)| \leq b_j(g_0,h,S) < \infty$  for all $t\in [0,S].$  

\end{thm}
\begin{proof}

 We will construct a short time solution to, \eqref{Meq}, that is  
\begin{align}\label{Meqex}
  \partt g_{ij}=&\,g^{ab} (\gradh_a\gradh_b g_{ij})
  -g^{kl}g_{ip}h^{pq}R_{jkql}(h)
  -g^{kl}g_{jp}h^{pq}R_{ikql}(h)\cr
  &\,+\tfrac12g^{ab}g^{pq}\left(\gradh_i g_{pa}\gradh_jg_{qb}
    +2\gradh_ag_{jp}\gradh_qg_{ib}-2\gradh_ag_{jp}
    \gradh_bg_{iq}\right.\cr
  &\,\left.\qquad\qquad\qquad-2\gradh_jg_{pa}\gradh_bg_{iq}
    -2\gradh_ig_{pa}\gradh_bg_{jq}\right),\cr
 & = \,g^{ab} (\gradh_a\gradh_b g_{ij}) +  (g^{-1}* g* \Rm(h)* h)_{ij} + (g^{-1}*g^{-1}*\gradh g* \gradh g)_{ij}
\end{align}
with $g(0) =g_0$.
The method is essentially  the one given in \cite{Shi}, with some minor modifications.

We choose radii $R(i) \to \infty$ such that $B_i = B_{R(i)}(p)$ have smooth boundary, and 
  $M = \cup_{i=1}^{\infty}B_i.$ 
  For fixed $R=R(i)\geq 1$ we modify $g_0$ to $g_{0,R} = \eta g_0 + (1-\eta)h$ where $\eta $ is a smooth cut off function
  with $ \eta =0$ outside of $B_{R/2}(p)$ and  $ \eta =1$  on $B_{R/4}(p),$ $|\grad^k \eta|^2 \leq c(k,h)$
  (see (iv) of Theorem \ref{balllemma} for the existence of $\eta$).
We still have 
\begin{eqnarray}
&&  \sup_M |\grad^j g_{0,R}| \leq \hat c_j(c_1,\ldots,c_j,h,n,a)< \infty \label{hatcj} \cr
&&  \frac{1}{a}h \leq g_{0,R} \leq ah 
\end{eqnarray}   
for some constants $0<\hat c_j(c_1,\ldots,c_j,h,n,a)< \infty$ which don't depend on $R$. 
Equation \eqref{Meqex} is strictly parabolic and $h$ and $g_{0,R}$ are smooth and so we obtain a smooth solution $g_R(t)_{t\in [0,T)}$ to the Dirichlet problem associated to \eqref{Meqex} with $g_R(0) = g_{0,R}$ and 
$g_R(t)|_{\boundary B_{R}(p) } $ $= (g_{0,R})|_{\boundary B_{R}(p) }$ $ = h|_{\boundary B_{R}(p) }$  for a $T = T(B_{R},g_{0,R},h)>0$  using the methods of   \cite{Shi} Section 3 and 4 (which in turn uses   Theorem 7.1, Section VII of \cite{LSU} ).
Using  the argument of   Lemma 3.1 of  \cite{Shi} 
  we see, as long as a smooth solution exists and   $|g_R(t)-g_{0,R}|_h^2 \leq \ep(g_0,a,h),$   then
 \begin{align*}
 & \frac{1}{2a}h <  g_R(t)  < 2ah \\
 & \sup_{B_R(p))} |\grad^j g_R(t)| \leq r(R,g_{0,R},h,j,S) < \infty
 \end{align*}
 for all $t\leq S$ for constants $r(R,g_{0,R},h,j,S ) < \infty.$  
 On the other hand, as long as $|g_R(t)-g_{0,R}|_h^2  \leq   \ep(g_0,a,h)  \leq 1 $ we have  (we write $g(t)$ for $g_R(t)$ and $g_0$ for $g_{0,R}$ for ease of reading): 
 \begin{eqnarray*}
 \partt |g(t)- g_{0}|_h^2  && =   g^{ab} (\gradh_a \gradh_b )|g(t)- g_{0}|_h^2- 2|\gradh g|_{g,h}^2   \\
 && + 2h^{ij}h^{kl}     g^{ab} (\gradh_a \gradh_b   g_{0})_{ik}  (g(t)- g_{0})_{jl} \cr
 &&   + 2  h^{ij}h^{kl} ( g(t) -  g_0  )_{ik} (g^{-1}* g* \Rm(h)* h)_{ij} h^{ij}  \cr
 && +  h^{ij}h^{kl} (g(t) -  g_{0})_{ik}(\grad g * \grad g * g^{-1} * g^{-1})_{jl} \cr
 &&\leq  g^{ab} (\gradh_a \gradh_b )|g(t)- g_{0}|^2 - |\gradh g|_{g,h}^2   + c(\hat c_2, ,a,n)
 \end{eqnarray*}
 where $\hat c_2$ is the constant defined above in \eqref{hatcj}, and is independent of $R$.
Hence,   $|g_R(t)-g_{0,R} |^2 \leq  c(\hat c_2, a,n) t \leq \ep(g_0,a,h) $ remains true for $t\leq  \hat T:= \frac{\ep(g_0,a,h)}{ c(\hat c_2, a,n)}$ in view of the maximum principle.  
Hence,   we   may extend the solution smoothly to time $\hat T:= \frac{\ep(g_0,h)}{ c(\hat c_2 ,a,n)} \leq 1.$
As long as  $|g_R(t)-g_{0,R}|_h^2   \leq   \ep(g_0,a,h)$, we also have, using the arguments of  Lemma 4.1 and 4.2 in \cite{Shi} 
and the fact that $\sup_M |\grad^j   g_{0,R}| \leq \hat c_j < \infty$ for constants $ \hat  c_j$ which do not depend on $R$, 
interior estimates:
\begin{eqnarray}
 \sup_{  B_{1}(x_0) \times [0, S]}  |\gradh^m g_R|^2 \leq b_m= c(m,  \hat c_1, \ldots, \hat c_m, a,S,h)  \label{interiorex}
 \end{eqnarray}
 for all $x_0 \in B_{\frac{R}{10}}(p)$ for all $S \leq \hat T$.   
Building the limit of the solutions $g_{R(i)}$   as $ i \to \infty,$    after taking a  subsequence if necessary,  we     obtain a  smooth solution $g(t)_{t\in [0,\hat T]}$ to \eqref{Meqex} with $g(0) = g_0,$ in view of the Theorem of Arzel{\`{a}}-Ascoli  and the fact that $R(i) \to \infty$ as $i\to \infty$, satisfying 
$
 \sup_{ M \times [0, S]}  |\gradh^m g|^2 \leq b_m 
$  for all $S \leq \hat T$ as required.
\end{proof}

\section{Geometry Lemmata}\label{geoapp}

\begin{lemma}\label{balllemma}
Let $(M^n,h)$ satisfy \eqref{hassumptionsscaled}: 
  $(M,h)$ is  a  smooth, connected, complete  Riemannian manifold, without boundary, satisfying 
\begin{align}
& \sup_M  {}^h|\gradh^i\Riem(h)|  < \infty \mbox{ for all } i\in \N_0 \cr
& \sum_{i=0}^4 \sup_M  {}^h|\gradh^i\Riem(h)|  \leq \de_0(n) \cr
 & \inj(M,h)  \geq 100 \label{hassumptionsagain} 
 \end{align} where $\de_0(n)$ is a sufficiently small constant.
Then there exist constants  $C_S(n)>0$  and a  constant $c_0(n)$   such that :
\begin{itemize}
\item[(i)] 
$$ (\int_M  f^{\frac{2n}{n-2}}  dh)^{\frac {n-2} {n} }  \leq C_S(n) \int_{M}  |\gradh f|^2 dh $$
and  $$ (\int_M  f^{n}  dh)^{\frac 1 2}   \leq C_S(n) \int_{M}  |\gradh f|^{\frac{n}{2}}  dh $$
for any $f$ which is smooth and  whose  support has diameter less then $4$
\item[(ii)]  there exists a $c_0(n)$ such that any ball $B_{2}(x)$ of radius  $2$  can be covered by $c_0(n)$ balls, $(B_{\frac 1 2}(y_i))_{i=1}^{c_0(n)}.$ 
\item[(iii)]  there exists   a covering of $M$, $(B_{1}(x_i))_{i=1}^{\infty}$,  by balls of $M$ such that for any $i \in \N,$\\
 $\sharp \{  j \in \N \ | \  x_j 
 \in B_{4}(x_i)  \} \leq c_0(n),$
where $\sharp C$ denotes the number of elements in the set  $C$, and is defined to be infinity  if $C$ has infinitely many elements 
\item[(iv)] For every $R>1,$ $x_0 \in M,$ there exists a cut off function $\eta: M \to [0,1] \subseteq \R $ such that
$\eta = 1$ on $B_R(x_0),$ $\eta = 0$ on $M \backslash (B_{C(n)R}(x_0)),$  
$|\gradh^2(\eta)| + \frac{|\gradh \eta|^2}{\eta}  \leq \frac{C(n)}{   R^2} $ on $M$  and $|\gradh^k \eta|\leq c(k,h)$ on $M$ for all $k\in \N$.
\item[(v)]
Let $\ep>0$ be given, and $T$ a smooth zero two tensor satisfying 
$\int_{B_{1}(x)} |\gradh T |^{\frac n 2} +   |\gradh^2 T |^{\frac n 2} \leq  \ep$ for all $x \in M$.
Then 
$(\int_{B_{1}(x)} |\gradh T |^n)^{\frac 1 2} \leq c(n)C_S(n)\ep$
\end{itemize}
\end{lemma}
\begin{rmk}\label{ballremark}
If the conditions  \ref{hassumptionsagain} are  replaced by 
\begin{align*}
& \sup_M  {}^h|\gradh^i\Riem(h)|  < \infty \mbox{ for all } i\in \N_0 \cr
 & \inj(M,h)>0  
 \end{align*} 
then, scaled versions of the statements (i)-(v) hold, as we  now explain. If we scale $h$ by a large constant $c(h),$ we obtain a new metric  which satisfies  \ref{hassumptionsagain}, and hence (i)-(v) hold for this new metric. Scaling back, we obtained scaled versions of the statements (i)-(v).  For example, part one of (i) would be replaced by: there exists an $r_0>0$ such that 
$$ (\int_M  f^{\frac{2n}{n-2}}  dh)^{\frac {n-2} {n} }  \leq C_S(n) \int_{M}  |\gradh f|^2 dh $$ for any $f$ which is smooth and  whose  support has diameter less then $r_0$.
\end{rmk}
\begin{proof}

We can always find local geodesic coordinates  for any $p_0 \in M$ on the   ball  $B_{50}(p_0)$ such that in these coordinates
$ \frac{99}{100} \de \leq h \leq \frac{101}{100}  \de$ if $\de_0(n)$ is sufficiently small. This implies that the first two statements  hold in these coordinates, and hence on the manifold.
 The third statement is proved as follows. First we construct a {\it maximal} set of disjoint balls $(B_{\frac 1 2}(x_i))_{i=1}^{\infty}$ for $M,$ maximal in the sense that any  ball $B_{\frac 1 2}(p)$ for an arbitrary $p\in M$ must intersect one of these balls. This construction is carried out as follows: first choose disjoint balls $B_{\frac 1 2}(x_1) , \ldots , B_{\frac 1 2}(x_{n(R)})$   with  centres in  $B_{R}(p_0),$ such that
any newly chosen ball $B_{\frac 1 2}(y)$ with $y \in B_{R}(p_0)$ intersects one of the balls $B_{\frac 1 2}(x_1) , \ldots , B_{\frac 1 2}(x_{n(R)})$ . In the next step, choose
balls $B_{\frac 1 2}(x_{n(R)}) , \ldots , B_{\frac 1 2}(x_{n(2R)}),$ with centres in $B_{2R}(p_0)$ such that the collection
$B_{\frac 1 2}(x_1) , \ldots , B_{\frac 1 2}(x_{n(2R)}),$ is disjoint, and 
any newly chosen ball $B_{\frac 1 2}(y)$ with $y \in B_{2R}(p_0)$ intersects one of the balls $B_{\frac 1 2}(x_1) , \ldots , B_{\frac 1 2}(x_{n(2R)}).$  Continuing in this way, we obtain a collection of disjoint balls  $(B_{\frac 1 2 }(x_i))_{i=1}^{\infty}$  which are maximal.

This then implies that  $(B_{1}(x_i))_{i=1}^{\infty}$ covers $M $: if $y \in M$ satisfies  $y \notin \cup_{i=1}^{\infty}B_1(x_i),$ then $B_{\frac 1 2}(y ) \cap B_{\frac 1 2}(x_i) = \emptyset$  for all $i\in \N,$ which contradicts the maximality, and hence $y \in \cup_{i=1}^{\infty}B_1(x_i)$.
In geodesic coordinates $\phi: B_{50}(x_i) \to \B_{50}(0),$  there can be at most $c_1(n)$ euclidean balls  
$ \B_{\frac 1 4}(\ti x_{k(j)} )_{j=1}^{c_1(n)}$, $\phi(x_{k(j)}) = {\ti x}_{k(j)},$ which are disjoint, and contained in  $\B_{40}(0)$. Hence there are at most $c_1(n)$ points, $(x_{k(j)})_{j=1}^{c_1(n)},$ which are contained in $B_{30}(x_i),$ and this implies (iii). \\
Statement (iv) is proved with the help of an {\it exhaustion function}.
For $R>1$ given, let 
 $\eta(x):= {\ti \eta}(\frac{f(x)}{R}),$ for a smooth cut off function  $\ti \eta : \R \to [0,1] \subseteq \R$ with $\ti \eta(x) = 1$ for $|x| \leq 1$ and $\ti \eta(x)  = 0$ for $|x| \geq 2,$  where $f:M \to \R^+$ is a smooth so called {\it exhaustion function}, satisfying $\frac{1}{C(n)} d(x,x_0) \leq  f(x) \leq \frac{1}{2}(d(x,x_0) + 1) ,$ $|\grad f| \leq C(n)$, $|\grad^2 f| \leq  C(n),$ the existence of which is, for example,  guaranteed by  Theorem 3.6 of Shi,  \cite{Shi2}. 
 By slightly modifying $f$ on geodesic balls of radius $1$ we can also achieve $|\grad^k f| \leq C(k,h).$
  
 Differentiating $\eta$ we see that  $\frac{|\gradh \eta|^2}{\eta} + |\gradh^2 \eta| 
\leq   \frac{1}{R^2} ( \frac{|\gradh  \ti{\eta}|^2\of f}{\ti{\eta }\of f} |\gradh f|^2  + |\gradh \ti \eta|\of f |\gradh ^2 f|  +  |\gradh ^2 \ti \eta|\of f |\gradh f|^2)
\leq \frac{C(n)}{R^2},$ as, without loss of generality,  $\frac{|\gradh  \ti {\eta}|^2}{\ti{\eta }} \leq c$ for some universal constant $c$. 
Similarly $|\gradh^k \eta|^2 \leq c(k,h)$.
This finishes (iv).

We now prove (v).  Let $\eta:M \to \R$ be a smooth cut off function with 
$\eta =1$ on $B_1(x)$ and $\eta =0$ outside of $B_{4/3}(x),$ and $B_{1}(x_1), \ldots, B_{1}(x_{c_0(n)})$ a covering of
$B_2(x)$, which exists in view of (ii). Then using Kato's and Young's inequality we see 
\begin{align*}
(\int_{B_1(x)} |\grad T|^n )^{1/2}  \leq  (\int_{B_2(x)} (|\grad T|\eta)^n )^{1/2}  
&\leq  C_S \int_{B_2(x)} |\grad (\eta |\grad T|)|^{\frac n 2} \cr
&\leq c(n) C_S \int_{B_2(x)} |\grad \eta|^{\frac n 2}  |\grad T|^{\frac  n 2}  + \eta^{\frac n 2}|\grad^2
T|^{\frac n 2}\cr
& \leq \sum_{i=1}^{c_0(n)} 2 c(n) C_S \int_{B_1(x_i)}  |\grad T|^{\frac  n 2} + |\grad^2
T|^{\frac  n 2}\cr
&\leq C_S c(n)c_0(n) \ep
\end{align*}
as required.

\end{proof}
\begin{lemma}\label{smalllocallemma}
Let $(M^n,h)$ be a  smooth, connected, complete  Riemannian manifold, without boundary, satisfying 
\begin{align*}
& \nu(3):= \sum_{i=1}^3 \sup_M  {}^h|\gradh^i\Riem(h)|  <\infty \mbox{ for all }   \cr
 & \inj(M,h)  \geq i_0>0. 
 \end{align*}
 and let $ g_0$ be in $W^{2,\frac n 2}_{loc}$  and satisfy  
\begin{align*}
 &  \frac 1 a h \leq g_0 \leq a h. \cr
 \end{align*}
 Then for any $\ep>0,R>1,x_0\in M$   there exists an $r_1$ such that  
 \begin{align*}
& \sup_{x\in B_R(x_0)} ( \int_{B_{r_1}(x)} ( |\grad g_0|^n   + |\grad^2 g_0|^{\frac n 2} ) < \ep 
\end{align*}
In the case that $\int_{M}   ( |\grad g_0|^n   + |\grad^2 g_0|^{\frac n 2} )  < \infty,$ then
 for any $\ep>0,$ there exists an $r_1$ such that 
 \begin{align*}
& \sup_{x\in M}  \int_{B_{r_1}(x)}  ( |\grad g_0|^n   + |\grad^2 g_0|^{\frac n 2} ) < \ep 
\end{align*}
\end{lemma}

\begin{proof}
As the conclusion is a scale invariant conclusion, it suffices to prove it after scaling $g_0$ and $h$ by the same constant.
We scale $g_0$ and $h$ once so that  $h$ satisfies \eqref{hassumptionsscaled},   hence the statements (i) -(v)  from Lemma \ref{balllemma}
 hold for the new metrics, which we also denote by $g_0$ and $h$.

Using the covering from  (iii),  we consider only those $x_i$ with $x_i \in B_{2R}(x_0)$ $i=1,\ldots,C(n,R)$ 
and cut off functions $\eta_i: M \to [0,1] \subseteq \R$ with
 $\supp(\eta_i) \subseteq B_{\frac 3 2}(x_i),$ $\eta_i = 1$ on $B_{1}(x_i)$, $|\gradh \eta_i|^2 \leq c(n) \eta_i,$ we see using the Sobolev-inequality 
\begin{align*}
(\int_{B_{2R}(x_0)}  |\grad g_0|^n)^{\frac 1 2}  & \leq  \Big(\sum_{i=1}^{\infty} \int_{B_1(x_i)}   |\grad g_0|^n \Big)^{\frac 1 2}  \cr
 & \leq   \sum_{i=1}^{C(n,R)} \Big( \int_{B_1(x_i)}   |\grad g_0|^n \Big)^{\frac 1 2}  \cr
 & \leq   \sum_{i=1}^{C(n,R)} \Big( \int_{M}   (\eta_i|\grad g_0|)^n \Big)^{\frac 1 2}  \cr
& \leq  \sum_{i=1}^{C(n,R)} \int_{M}  |\grad ( \eta_i |\grad g_0|)|^{\frac n 2} \cr
  & \leq  \sum_{i=1}^{C(n,R)} \int_{M}  c(a,n)  |\grad \eta_i|^{\frac n 2} |\grad g_0|^{\frac n 2}
 + c(n,a) |\eta_i|^{\frac n 2} |\grad^2 g_0|^{\frac n 2}  \cr
 &  \leq c(n,a) \sum_{i=1}^{C(n,R)}  \int_{B_2(x_i)} |\grad g_0|^{\frac n 2}  +|\grad ^2 g_0|^{\frac n 2} \cr 
 & = c(n,a) \sum_{i=1}^{C(n)}  \int_{M} \chi_{B_{2}(x_i)} (|\grad g_0|^{\frac n 2} +|\grad ^2 g_0|^{\frac n 2}) \cr  
 & = c(n,a) \int_M  ( \sum_{i=1}^{C(n,R)} \chi_{B_{2}(x_i)})  ( |\grad g_0|^{\frac n 2}+|\grad ^2 g_0|^{\frac n 2}) \cr
 & \leq c(n,a) c_0(n) (\int_{B_{2R}(x_0)} |\grad g_0|^{\frac n 2} +|\grad ^2 g_0|^{\frac n 2} ) =K(n,a,R,x_0)< \infty,
 \end{align*}
 where we used   $\sum_{i=1}^{\infty} \chi_{B_{2}(x_i)}(\cdot)   \leq c_0(n)$ in the last inequality, which follows from (iii).
We claim : For any $\ep >0$  there exists  $r= r(\ep)  >0$ such that 
 $  \int_{B_{r}(x)}  |\grad g_0|^n + \int_{B_{r}(x)}|\grad^2 g_0|^{\frac n 2}  < \ep$ for all $x \in B_{R}(x_0)$, as we now show. 
Assume there are points $x_i \in \overline{B_{R}(x_0)}$ $i\in \N$ and radii $r(i) >0$, $r(i) \to 0$ as $i\to \infty, $ 
such that $\int_{B_{r(i)}(x_i)}  |\grad g_0|^n + |\grad^2 g_0|^{\frac n 2} \geq  \ep.$ 
Taking a subsequence, we see that
 $ x_i \to x$ as $i \to \infty,$ and hence $\int_{B_{\si}(x)}  |\grad g_0|^n + |
\grad^2 g_0|^{\frac n 2} \leq  \frac{\ep}{2}$ for $\si>0$ small enough, in view of the fact
 that  $\int_{B_{2R(x_0)}}  |\grad g_0|^n  + |\grad^2 g_0|^{\frac n 2}< \infty$: 
 $ f_{j}:=  \chi_{B_{\frac 1 j}(x)} |\grad g_0|^n  + |\grad^2 g_0|^{\frac n 2} \leq g := |\grad g_0|^n  + |\grad^2 g_0|^{\frac n 2},$ is in $L_1$ 
 $f_{j} \to 0$ almost everywhere as $j \to \infty$, and $\int_{B_{2R}(x_0)} g < \infty$ implies
 $ \int_{\chi_{B_{\frac 1 j }(x) }} |\grad g_0|^n  + |\grad^2 g_0|^{\frac n 2}  =  \int_{B_{2R}(x_0)} f_j \to 0$ in view of the dominated convergence theorem.
 But for $i$ large enough, $B_{r(i)}(x_i) \subset B_{\si}(x)$ which leads to a contradiction.
 Hence there exists an $r>0$ such that  $  \int_{B_{r}(x)} (  |\grad g_0|^n  +  |\grad^2 g_0|^{\frac n 2} ) < \ep$ for all $x \in B_{R}(x_0) .$  
 In the case that $\int_{M} ( |\grad g_0|^n   + |\grad^2 g_0|^{\frac n 2} ) < \infty,$
 choose $R>0$ so that 
$\int_{ (B_{R/10}(x_0))^c} ( |\grad g_0|^n   + |\grad^2 g_0|^{\frac n 2} ) < \frac{\ep}{2}.$
This implies $\int_{B_{\si}(x)} ( |\grad g_0|^n   + |\grad^2 g_0|^{\frac n 2} ) < \frac{\ep}{2}$ for all $ x\in  (B_{R/2}(x_0))^c$ for any $0<\si<1.$ Repeating the argument above, we
find a $\si>0$ such that $\int_{B_{\si}(x)} ( |\grad g_0|^n   + |\grad^2 g_0|^{\frac n 2} ) < \frac{\ep}{2}$ for all $ x\in  (B_{R}(x_0)).$  Hence,  $\int_{B_{\si}(x)} ( |\grad g_0|^n   + |\grad^2 g_0|^{\frac n 2} ) < \frac{\ep}{2}$ for all $x \in M,$ as required.

\end{proof}

\section{Estimates for ordinary differential equations}
\begin{lemma}\label{ODELem}
Let $\ep <1$ and $f :[0,T] \to \R^{+}_0,$ $Z:(0,1] \to \R^{+}_0$ be smooth, and satisfy
\begin{eqnarray*}
f(0) &&= 0 \\
\partt f(t) &&\leq \frac{\ep }{t} f(t)  + Z(t). \\
\end{eqnarray*}
Then 
\begin{eqnarray*}
f(t) && \leq  t^{\ep} \lim_{t_0 \to 0} \int_{t_0}^t \frac{Z(s)}{s^{\ep}} ds.
\end{eqnarray*}

\end{lemma}
\begin{proof}

 $F(t) :=  t^{- \ep}f(t)$  satisfies
\begin{align}
& \partt F(t)\cr 
 & \leq  -\ep t^{-1- \ep} f(t)   + t^{-\ep} \frac{\ep }{t}f(t) + t^{-\ep} Z(t) \cr 
& \leq   t^{-\ep}Z(t)
 \label{secondF} 
 \end{align}
 Using that $f$ is smooth and hence $f(t) \leq Ct$ for small $t>0$ and some constant $C$, we see 
$F(t) \leq Ct^{-\ep + 1} \to 0$ as $t\downto 0.$
Integrating \eqref{secondF} from  $t_0>0$ to $t$, 
we see $ F(t) \leq F(t_0) + \int_{t_0}^{t} \frac{Z(s)}{s^{\ep}}ds \to \lim_{t_0 \to 0} \int_{t_0}^t \frac{Z(s)}{s^{\ep}} ds,$ as $t_0 \downto 0$ and hence, from the definition of $F(t),$
\begin{eqnarray}
f(t) && \leq  t^{\ep} \lim_{t_0 \to 0} \int_{t_0}^t \frac{Z(s)}{s^{\ep}} ds.
\end{eqnarray}
\end{proof}
\begin{lemma}\label{ODECor}
Let $\ep <1$ and $f :[0,T] \to \R^{+}_0,$  be smooth, and satisfy
\begin{eqnarray*}
f(0) &&= 0 \\
\partt f(t) &&\leq \frac{\ep }{t} f(t)  + c \\
\end{eqnarray*}
Then 
\begin{eqnarray*}
f(t) && \leq   \frac{c}{1-\ep}t.
\end{eqnarray*}

\end{lemma}
\begin{proof}
For  $Z(s) = c$   we have $$\lim_{t_0 \to 0} \int_{t_0}^t \frac{Z(s)}{s^{\ep}} ds
= \lim_{t_0 \to 0} \int_{t_0}^t \frac{c}{s^{\ep}} ds = c \frac{1}{1-\ep}{t}^{1-\ep}$$ and so
$t^{\ep}\lim_{t_0 \to 0} \int_{t_0}^t \frac{Z(s)}{s^{\ep}} ds = t^{\ep}\lim_{t_0 \to 0} \int_{t_0}^t \frac{c}{s^{\ep}} ds  = \frac{c}{1-\ep}t$
\end{proof}
\section{ Metric norm comparisons}
 
We  compare the  norms of  tensor with respect to   different metrics.

\begin{thm} \label{lpmetricthm}
Let $\ell= (\ell_{ij})_{i,j  \in \{1, \ldots n\}}, g= (g_{ij})_{i,j \in \{1, \ldots n\}}, h = (h_{\al \be})_{\al,\be \in \{1, \ldots n\}},  (u_{\al \be})_{\al, \be  \in \{1, \ldots n\}} $ be positive definite symmetric matrices  and 
$(\ell)^{-1} = (\ell^{ij})_{i,j  \in \{1, \ldots n\}}, (g)^{-1}= (g^{ij})_{i,j \in \{1, \ldots n\}}, (h)^{-1} = (h^{\al \be})_{\al,\be \in \{1, \ldots n\}}, (u^{-1}) = (u^{\al \be})_{\al ,\be \in \{1, \ldots, n\}}   $ the inverses thereof.
Let $ S= (S_{\al}^{i})_{i,\al \in \{1, \ldots n\}}, T  = (T_{ij})_{i,j \in \{1, \ldots n\}},$ $N = (N^{ij})_{i,j \in \{1, \ldots n\}}$  be  matrices in $\R^{n\times n}.$
 
Then the following estimates hold for any $\ep>0$ :
\begin{eqnarray}
  |S|^2_{h,\ell} && :=  h^{\al \be}(y) S_{\al}^i    S_{\be}^j   \ell_{i j }   
  \cr
  && \leq  c(n) |S|^2_{h,g} (1 +|\ell|^2_g  )\label{firstD}
 \end{eqnarray}
 \begin{eqnarray}
  |S|^2_{h,\ell} 
  && \leq  c(n) |S|^2_{u,\ell} (1 +|u|^2_{h}  ) \label{secondD}
 \end{eqnarray}
 
 where  $|\ell|^2_{g}  =  g^{ij}g^{kl} \ell_{ik} \ell_{jl} = |g^{-1}|^2_{\ell} $ and 
  $|u|^2_{h} =  h^{\al \ga }h^{\be \ga } u_{\al \be} u_{\ga \si} = |h^{-1}|_{u}^2 ,$
  \begin{eqnarray*}
 |T|^{2}_{g}
&&  :=    g^{ik}g^{jl} T_{ij} T_{kl}\cr
&& \leq    c(n) |T|_{\ell}^2 |\ell|^2_{g}
\end{eqnarray*}
  \begin{eqnarray*}
 |N|^{2}_{g}
&&  :=    g_{ik}g_{jl} N^{ij} N^{kl}\cr
&& \leq    c(n) |N|_{\ell}^2 |g|^2_{\ell}
\end{eqnarray*}  
  \begin{eqnarray*}
 \frac{\det(g) }{\det(\ell)} \leq 
  |g|^{n}_{\ell} 
  \end{eqnarray*}
  where  $|g|^2_{\ell}  =  (\ell^{ij}\ell^{kl} g_{ik} g_{jl} ) $ 
   \end{thm}
 \begin{proof}
 We regard $g,\ell$ as positive definite symmetric linear maps from $V \otimes V $ to $\R$ where $V= \R^n$ and $h,u$ as positive definite symmetric linear maps  from $Y \otimes Y $ to $ \R$ for another copy of $Y:= \R^n$.
 $g,\ell: V \otimes V \to \R$, $h,u: Y \otimes Y \to\R,$
 $g(v^ie_i,v^je_j) = v^iv^jg_{ij}, \ell(w^ie_i,w^je_j) = w^iw^jg_{ij},$
 $h(z^{\al}e_{\al},z^{\be}e_{\be}) =   z^{\al}z^{\be}h_{\al \be},$   $u(z^{\al}e_{\al},z^{\be}e_{\be}) =   z^{\al}z^{\be}u_{\al \be},$
 and we regard $S$ $,T,$ and $N$  as linear map $S:Y^* \times V  \to \R ,$ $T: V \times V \to \R$, $N:V^{*} \times V^{*} \to \R,$
 $S( w_{\al}e^{\al},v^ie_i) = S^{\al}_i  w_{\al} v^i.$ 
From the theory of tensors,  $|S|^2_{h,\ell}, |S|^2_{u,\ell} , |\ell^{-1}|^2_g, |T|^{2}_{g},  \frac{\det(g)}{\det(\ell)},$ $|N]|_{\ell}^2,$  $|g|^2_{\ell}$  etc. are all quantities which are 
independent of coordinates: if $ (\ti e_i)_{i\in \{1, \ldots n\}} ,$ $(\hat e_{\al})_{\al \in \{1, \ldots n\} }$ are bases, 
and $ \ti \ell_{ij} = \ell(\ti e_i , \ti  e_j), \ti g_{ij} = g(\ti e_i , \ti e_j),$
 $\hat h_{\al \be} = h(\hat e_{\al} ,\hat e_{\be}), \hat u_{\al \be} = u(\hat e_{\al} ,\hat e_{\be})$ 
 with inverses given by $\ti \ell^{ij},$ $\ti g^{ij},$ $  \hat h^{\al \be}, \hat u^{\al \be}$ 
 then the quantities  defined above as  $|S|^2_{\ell,h},  |S|^2_{u,\ell} ,  |T|^2_{g}, \frac{\det(g)}{\det(\ell)}, |\ell^{-1}|^2_g $   calculated using  $\ti g_{ij},\ti g^{ij}, \hat h_{\al \be}, \hat u_{\al \be}, \hat h^{\al \be}, \hat u^{\al \be}\ti \ell_{ij},\ti \ell^{ij},  \ti T_{ij}, {\ti {\hat S}}_{\al}^{i}$ in place of
   $ g_{ij}, g^{ij},  h_{\al \be}, h^{\al \be},\ell_{ij}, \ell^{ij},  T_{ij},  S_{\al}^{i}$ then the result is the same : see for example \cite{Hamilton}.

We can always choose a basis for $Y$ such that $\hat  h_{\al \be} = \de_{\al \be},$
$\hat u_{\al \be} = r_{\al} \de_{\al\be}$   and a basis for $V$ such that 
$ \ti g_{ij}= \la_i \de_{ij}, \ti \ell_{ij} = \si_i\de_{ij}.$
That is without loss of generality,  we have $  h_{\al \be} = \de_{\al \be}, u_{\al \be} = r_{\al} \de_{\al \be}$ and 
$  g_{ij}= \la_i \de_{ij},  \ell_{ij} = \si_i\de_{ij}$.
 
\begin{eqnarray*}
&& |T|_{g}^{2} \cr
&&  = (\sum_{i,j=1}^{n}  \frac{1}{\la_i}  \frac{1}{\la_j} T_{ij} T_{ij} ) \cr
 && =   (\sum_{i,j=1}^{n}  \frac{1}{\si_i}  \frac{1}{\si_j} \frac{\si_i}{\la_i}  \frac{\si_j}{\la_j}  T_{ij} T_{ij}  ) \cr
 && \leq c(n)(\sup_{i,j \in \{1,\ldots n\}} \frac{1}{\si_i}  \frac{1}{\si_j}T_{ij} T_{ij} )
 \cdot(\sup_{i \in \{1,\ldots n\}}  \frac{\si^2_i}{\la^2_i} )  \cr
&&  \leq   c(n) (\sum_{i,j=1}^{n} \frac{1}{\si_i}  \frac{1}{\si_j}T_{ij} T_{ij}) 
  (\sum_{i=1}^n\frac{\si^2_i}{\la^2_i}) \cr
&&   = c(n) |T|_{\ell}^2   (g^{ij}g^{kl} \ell_{ik} \ell_{jl} ) \cr
&& =  c(n) |T|_{\ell}^2 |\ell|^2_{g}
\end{eqnarray*}
\begin{eqnarray*}
&& |N|_{g}^{2} \cr
&&  = (\sum_{i,j=1}^{n}   \la_i   \la_j  N^{ij} N^{ij} ) \cr
 && =   (\sum_{i,j=1}^{n}  \si_i \si_j \frac{\la_i}{\si_i}  \frac{\la_j}{\si_j} N^{ij} N^{ij}  ) \cr
 && \leq c(n)(\sup_{i,j \in \{1,\ldots n\}} \si_i  \si_j N^{ij} N^{ij} )
 \cdot(\sup_{i \in \{1,\ldots n\}}  \frac{\la^2_i}{\si^2_i} )  \cr
&&  \leq   c(n) (\sum_{i,j=1}^{n} \si_i \si_j N^{ij} N^{ij}) 
  (\sum_{i=1}^n\frac{\la^2_i}{\si^2_i}) \cr
&&   = c(n) |N|_{\ell}^2   (g_{ij}g_{kl} \ell^{ik} \ell^{jl} ) \cr
&& =  c(n) |N|_{\ell}^2 |g|^2_{\ell}
\end{eqnarray*}

Similarly
\begin{eqnarray*} 
\frac{\det(g)}{\det(\ell)} && = \frac{\la_1 \la_2 \ldots \la_n}{ \si_1   \si_2 \ldots  \si_n   }\cr
&&   = \frac{\la_1}{\si_1}\cdot \frac{ \la_2 }{\si_2} \ldots  \frac{ \la_n }{\si_n} \cr
&& \leq (\sup_{i\in \{1 \ldots n\}} \frac{ \la^2_i }{\si^2_i} )^{\frac n 2} \cr
&& \leq   \Big( \sum_{i=1}^n \frac{\la^2_i}{\si^2_i} \Big)^{\frac n 2} \cr
&& = |g |_{\ell}^{n}  
\end{eqnarray*}
and 
\begin{eqnarray*} 
  |S|^2_{h,\ell} && :=  h^{\al \be}(y) S_{\al}^i    S_{\be}^j   \ell_{i j }   
  \cr
  && = \sum_{\al,i=1}^n  S_{\al}^i   S_{\al}^ i   \si_i  \cr
  && =  \sum_{\al,i=1}^n  S_{\al}^i   S_{\al}^ i  \la_i \frac{ \si_i}{\la_i}  \cr
  && \leq c(n)(\sup_{\al, i \in \{1, \ldots, n\}} S_{\al}^i   S_{\al}^ i  \la_i)
  \sup_{  i \in \{1, \ldots, n\}} \frac{ \si_i}{\la_i} \cr 
  && \leq c(n)(\sup_ {\al, i \in \{1, \ldots, n\}} S_{\al}^i   S_{\al}^ i  \la_i)
  (1 + \sup_{  i \in \{1, \ldots, n\}} \frac{ \si^2_i}{\la^2_i})\cr 
  && \leq c(n) ( \sum_{\al,i=1}^n   S_{\al}^i   S_{\al}^ i  \la_i )
  ( 1 + \sum_{i=1}^n \frac{ \si^2_i}{\la^2_i} )\cr 
  && = c(n)|S|^2_{h,g}(1+  \sum_{i=1}^n\frac{ \si^2_i}{\la^2_i})\cr 
&& = c(n)|S|^2_{h,g} (1 +   |\ell|^2_g)
\end{eqnarray*}
Similarly
\begin{eqnarray*} 
  |S|^2_{h,\ell} && :=  h^{\al \be}(y) S_{\al}^i    S_{\be}^j   \ell_{i j }   
  \cr
  && = \sum_{\al,i=1}^n  S_{\al}^i   S_{\al}^ i   \si_i  \cr
  && = \sum_{\al,i=1}^n  r_{\al} (\frac{1}{r_{\al} } S_{\al}^i   S_{\al}^ i   \si_i )\cr 
   && \leq  c(n)   ( \sup_{\al \in \{ 1, \ldots , n \} } r_{\al} )  (\sup_{\al, i \in \{1, \ldots, n\}}
   \frac{1}{r_{\al} } S_{\al}^i   S_{\al}^ i   \si_i  )\cr 
   && \leq c(n)  (\sum_{\al=1}^n  r_{\al}    )(\sum_{\al,i=1}^n   \frac{1}{r_{\al} } S_{\al}^i   S_{\al}^ i   \si_i )\cr
    && =  c(n)  (\sum_{\al=1}^n   r_{\al}    ) |S|^2_{u,\ell}\cr
&& \leq  c(n)  (1+\sum_{\al=1}^n   r^2_{\al}    ) |S|^2_{u,\ell}\cr
&& \leq c(n)  (1 +  |u|^2_{h})|S|^2_{u,\ell}
\end{eqnarray*} 
\end{proof}

\begin{cor}\label{lpmetriccor}
Let $T=(T_{ij})$ , respectively $N = (N^{ij})$  be a zero-two respectively two-zero  tensor defined on a manifold  $\Omega$, and $g,\ell$ metrics on $\Omega$.
Then for all $p\in [1,\infty)$ there exists a $c(n,p)$ such that
 \begin{eqnarray}
  \int_{\Omega} |T|^{p}_{g } dg  \leq c(n,p)    (\int_{\Omega} |\ell|^{2p}_{g} dg)^{\frac 1 2} 
(\int_{\Omega} |T|^{4p}_{\ell}  d\ell )^{\frac 1 4} (\int_{\Omega} |g|_{\ell}^{\frac n 2}  dg )^{\frac 1 4}    \label{Testimate3}  
\end{eqnarray} 
and 
\begin{eqnarray}
  \int_{\Omega} |N|^{p}_{g } dg  \leq   c(n,p) (\int_{\Omega} | g|^{2p}_{\ell} dg)^{\frac 1 2} 
(\int_{\Omega} |N|^{4p}_{\ell}  d\ell )^{\frac 1 4} (\int_{\Omega} |g|_{\ell}^{\frac n 2}  dg )^{\frac 1 4}.   \label{Testimate4}  
\end{eqnarray} 

\end{cor}
\begin{proof}
In the following, $dg/dl$ is the well defined function on $\Omega$ given locally by $ dg/dl(x)= \frac{\sqrt{det(g(x))}}{\sqrt{det(l(x))}}$
\begin{eqnarray*}
&& \int_{\Omega} |T |^{p}_{g} dg \cr
&&  \leq c(n,p) \int_{\Omega} |\ell|^p_{g} |T|^p_{\ell} dg  \cr 
&& \leq    c(n,p) (\int_{\Omega} |\ell|^{2p}_{g} dg)^{\frac 1 2} 
(\int_{\Omega} |T|^{2p}_{\ell} dg)^{\frac 1 2}  \cr 
&&   =  c(n,p)  (\int_{\Omega} |\ell|^{2p}_{g} dg)^{\frac 1 2} 
(\int_{\Omega} |T|^{2p}_{\ell} \frac{dg}{d\ell} d\ell )^{\frac 1 2}  \cr
&& \leq  c(n,p) (\int_{\Omega} |\ell|^{2p}_{g} dg)^{\frac 1 2} 
(\int_{\Omega} |T|^{4p}_{\ell}  d\ell )^{\frac 1 4}  (\int_{\Omega} (\frac{dg}{d\ell})^2 d\ell )^{\frac 1 4} \cr
&& \leq  c(n,p) (\int_{\Omega} |\ell|^{2p}_{g} dg)^{\frac 1 2} 
(\int_{\Omega} |T|^{4p}_{\ell}  d\ell )^{\frac 1 4}  (\int_{\Omega} \frac{dg}{d\ell} dg )^{\frac 1 4} \cr
&& \leq  c(n,p) (\int_{\Omega} |\ell|^{2p}_{g} dg)^{\frac 1 2} 
(\int_{\Omega} |T|^{4p}_{\ell}  d\ell )^{\frac 1 4} (\int_{\Omega} |g|_{\ell}^{\frac n 2}  dg )^{\frac 1 4}  
\end{eqnarray*} 
Analog:
\begin{eqnarray*}
&& \int_{\Omega} |N |^{p}_{g} dg \cr
&&  \leq  c(n,p) \int_{\Omega} |g|^p_{\ell} |N|^p_{\ell} dg  \cr 
&& \leq   c(n,p) (\int_{\Omega} |g|^{2p}_{\ell} dg)^{\frac 1 2} 
(\int_{\Omega} |N|^{2p}_{\ell} dg)^{\frac 1 2}  \cr 
&&   =  c(n,p) (\int_{\Omega} |g|^{2p}_{\ell} dg)^{\frac 1 2} 
(\int_{\Omega} |N|^{2p}_{\ell} \frac{dg}{d\ell} d\ell )^{\frac 1 2}  \cr
&& \leq  c(n,p)(\int_{\Omega}|g|^{2p}_{\ell} dg)^{\frac 1 2} 
(\int_{\Omega} |N|^{4p}_{\ell}  d\ell )^{\frac 1 4}  (\int_{\Omega} (\frac{dg}{d\ell})^2 d\ell )^{\frac 1 4} \cr
&& \leq  c(n,p) (\int_{\Omega} |g|^{2p}_{\ell} dg)^{\frac 1 2} 
(\int_{\Omega} |N|^{4p}_{\ell}  d\ell )^{\frac 1 4}  (\int_{\Omega} \frac{dg}{d\ell} dg )^{\frac 1 4} \cr
&& \leq  c(n,p)(\int_{\Omega} |g|^{2p}_{\ell} dg)^{\frac 1 2} 
(\int_{\Omega} |N|^{4p}_{\ell}  d\ell )^{\frac 1 4} (\int_{\Omega} |g|_{\ell}^{\frac n 2}  dg )^{\frac 1 4}  
\end{eqnarray*} 
 \end{proof}

\end{document}